\documentclass[12pt]{amsart} 
\usepackage{verbatim, latexsym, amssymb, amsmath,color}
\usepackage{enumitem}
\usepackage{epsfig}
\usepackage{xcolor}

\usepackage{geometry}
\usepackage{hyperref}

\DeclareMathOperator{\Diff}{Diff}
\DeclareMathOperator{\Ric}{Ric}
\DeclareMathOperator{\interior}{Int}
\DeclareMathOperator{\spt}{spt}
\DeclareMathOperator{\Id}{Id}

\DeclareMathOperator{\Vol}{Vol}

\def\R{\mathbb R}

\def\N{\mathbb N}

\def\Si{\Sigma}
\def\G{\Gamma}
\def\Om{\Omega}
\def\F{\mathbf{F}}

\def\dist{\mathrm{dist}}

\newtheorem{thm}{Theorem}[section]
\newtheorem{remark}[thm]{Remark}
\newtheorem{lemma}[thm]{Lemma}
\newtheorem{coro}[thm]{Corollary}

\newtheorem{prop}[thm]{Proposition}
\newtheorem{conj}{Conjecture}
\newtheorem{defi}[thm]{Definition}
\newtheorem{question}[thm]{Question}

\renewcommand{\epsilon}{\varepsilon}

\begin{document}
\title[On the existence of minimal Heegaard surfaces]
{On the existence of minimal Heegaard surfaces}

\author{Daniel Ketover}\address{Department of Mathematics\\Rutgers University \\Piscataway, NJ 08854}
\author{Yevgeny Liokumovich}
\address{Department of Mathematics\\University of Toronto\\Toronto, ON M5S2E4}
\author{Antoine Song}
\address{California Institute of Technology\\ 177 Linde Hall, \#1200 E. California Blvd., Pasadena, CA 91125}

\thanks{D.K. was partially supported by an NSF Postdoctoral Research fellowship, ERC-2011-StG-278940, and NSF DMS-1906385. Y.L. was partially supported by NSF DMS-1711053 and NSERC Discovery grants. A.S. was partially supported by NSF-DMS-1509027. This research was partially conducted during the period A.S. served as a Clay Research Fellow.}

\email{dk927@math.rutgers.edu}
\email{ylio@math.toronto.edu}
\email{aysong@caltech.edu}

\begin{abstract}
Let $H$ be a strongly irreducible Heegaard surface in a closed oriented Riemannian $3$-manifold.  We prove that $H$ is either isotopic to a minimal surface of index at most one or isotopic to the boundary of a tubular neighborhood about a non-orientable minimal surface with a vertical handle attached. This confirms a long-standing conjecture of J. Pitts and J.H. Rubinstein. 
\end{abstract}

\maketitle

\section{Introduction}

Given a surface $\Sigma$ embedded in a Riemannian 3-manifold, it is natural to ask whether $\Sigma$ contains a minimal surface representative in its isotopy class.   If $\Sigma$ is an incompressible surface 
embedded in a $3$-manifold, a result of B. Meeks, L. Simon and S.T. Yau \cite{MSY} implies that one can minimize area in the isotopy class of $\Sigma$ to obtain a stable minimal surface. 

Many $3$-manifolds contain no incompressible surfaces and so it may not be possible to construct minimal surfaces by minimization techniques.  On the other hand, every 3-manifold admits a Heegaard surface and one can try to represent such a surface by an index $1$ minimal surface using min-max methods.

Roughly speaking, the idea is as follows. A Heegaard splitting determines a continuous family of surfaces foliating the manifold and whose two ends are graphs (the spines of the two handlebodies).  Taking a sequence of such $1$-parameter smooth families of surfaces $\{\Sigma^i_t\}_{t\in[0,1]}$ sweeping out the manifold, which are tighter and tighter in the sense that $\max_t \mbox{Area}(\Sigma^i_t)$ converges to the infimum possible among such Heegaard sweepouts, one might hope that a subsequence of surfaces $\Sigma^{j}_{t_{j}}$ converges (in a weak sense) to a minimal surface with controlled topology.  This was carried out by L. Simon and F. Smith in the 80s (see \cite{C&DL}, \cite{DP} for an exposition).  The optimal genus bounds below were obtained in \cite{Ketgenusbound}:

\begin{thm}[Simon--Smith (1983)]\label{simonsmith}
If $H$ is a Heegaard surface of genus $g$ in an oriented Riemannian 3-manifold, then there exists a sequence of surfaces $\Sigma_i$ isotopic to $H$, and pairwise disjoint embedded minimal surfaces $\Gamma_1,...,\Gamma_k$ as well as positive integers $n_1,...,n_k$ so that
\begin{equation}
\Sigma_i\rightarrow \sum_{j=1}^k n_j\Gamma_j \mbox{ in the sense of varifolds.}
\end{equation}
Moreover, there holds
\begin{equation}
\sum_{i\in\mathcal{O}} n_i g(\Gamma_i)+\frac{1}{2}\sum_{i\in\mathcal{N}} n_i (g(\Gamma_i)-1)\leq g,
\end{equation}
where $\mathcal{O}$ denotes the subcollection of indices $j$ such that $\Gamma_j$ is orientable, $\mathcal{N}$ denotes the subcollection of indices $j$ such that $\Gamma_j$ is non-orientable and $g(\Gamma_j)$ denotes the genus of $\Gamma_j$.  Moreover, $n_i$ is even whenever $i\in\mathcal{N}$.
\end{thm}

In the above theorem the genus $g(\Gamma_j)$ of a nonorientable surface $\Gamma_j$
is defined as the number of cross-caps one must add to a sphere to obtain a surface homeomorphic to $\G_i$.

Note that the min-max minimal surface produced by Theorem \ref{simonsmith} may be disconnected and some components may have integer multiplicities. 

While Theorem \ref{simonsmith} implies that every 3-manifold contains an embedded minimal surface, only in the case of the 3-sphere does Theorem \ref{simonsmith} produce a minimal surface with prescribed topology, namely, a minimal 2-sphere.  Starting with a stabilized Heegaard splitting, one does not expect to obtain an isotopic minimal surface.   For instance, considering genus $1$ splittings of the round 3-sphere, Theorem \ref{simonsmith} produces a minimal sphere, and not a torus\footnote{F.C. Marques and A. Neves (\cite{MaNeindexbound}, Theorem 1.2) showed that one obtains a minimal surface with index at most $1$ when one considers one-parameter sweepouts.  On the other hand, the only embedded minimal tori in $\mathbb{S}^3$ have Morse index $5$.}.

However, in the 80s J. Pitts and J.H. Rubinstein outlined a claim that in the case when the Heegaard surface $H$ is assumed to be \emph{strongly irreducible}, one could rule out degeneration of the min-max sequence.  

A Heegaard surface $\Sigma$ is called \emph{strongly irreducible} if every closed curve on $\Sigma$ bounding an essential disk in one of the handlebodies intersects every such curve bounding an essential disk in the other handlebody.   Strongly irreducible splittings were introduced by A.J. Casson and C.M. Gordon \cite{CassonGordon} who proved that in non-Haken manifolds, one can always destabilize (or reduce the genus of) a Heegaard splitting until it is strongly irreducible.  In other words, in non-Haken manifolds an irreducible Heegaard splitting is also strongly irreducible.    

Precisely, J. Pitts and J.H. Rubinstein conjectured (\cite[Theorem 1.8]{Rubinsteinnotes})

\begin{conj}[J. Pitts and J.H. Rubinstein 1980s \cite{Rubinsteinnotes}] \label{rubinstein}

Let $(M,g)$ be a closed oriented Riemannian $3$-manifold not diffeomorphic to the $3$-sphere and suppose that there is a strongly irreducible Heegaard splitting $H$. Then one of the following holds:
\begin{enumerate}[label=\roman*)]
\item $H$ is isotopic to an index $1$ or $0$ minimal surface 
\item $H$ is isotopic to the boundary of the tubular neighborhood of a non-orientable minimal surface with a vertical handle attached.
\end{enumerate}

\end{conj}

Let us explain the situation in case ii).  Given a non-orientable surface $\Gamma$ embedded in a 3-manifold consider the $\epsilon$-tubular neighborhood $T_\epsilon(\Gamma)$ about $\Gamma$.  The set $T_\epsilon(\Gamma)$ is a twisted interval bundle over $\Gamma$ and its boundary $\partial(T_\epsilon(\Gamma))$ is an orientable surface.   The operation of \emph{attaching a vertical $1$-handle} to $\partial(T_\epsilon(\Gamma))$ means the following: for some $p\in\Gamma$ and small $\delta > 0$ we remove from the surface $\partial(T_\epsilon(\Gamma))$ the two disks given by 
\begin{equation}
D_\pm=\{\exp_{x\in\Gamma}(\pm\epsilon n(x))\;|\;\mbox{dist}_\Gamma(x,p)\leq\delta\}
\end{equation}
and add in the annulus 
\begin{equation}
A=\{\exp_{x\in\Gamma}(t n(x))\;|\;\;\mbox{dist}_\Gamma(x,p)=\delta \mbox{ and } -\epsilon\leq t\leq\epsilon\},
\end{equation}
where $n(x)$ denotes a choice of unit normal near $p$.

Case ii) can arise already when considering genus $1$ splittings of $\mathbb{RP}^3$.  After a single neck-pinch on such a Heegaard torus one obtains a 2-sphere bounding a twisted interval bundle over $\mathbb{RP}^2$.  Since $H$ is a Heegaard surface it follows that the complement of the non-orientable surface in case ii) of Conjecture \ref{rubinstein} is itself a handlebody.   Such a non-orientable surface is called a \emph{one-sided Heegaard splitting}, as introduced by Rubinstein (Section 0 in \cite{Rubinstein}).  An example is $\mathbb{RP}^2\subset\mathbb{RP}^3$ as the complement of $\mathbb{RP}^2$ in $\mathbb{RP}^3$ is a 3-ball.  Spherical space-forms with Heegaard genus $2$ that contain embedded Klein bottles give other examples where case ii) can arise.

Since in Theorem \ref{simonsmith} the minimal surface obtained by min-max methods could be broken into several components and thus not isotopic to $H$, the idea of J. Pitts and J.H. Rubinstein was to iterate the min-max procedure until the desired situation occurred. Roughly speaking, their argument was as follows (cf. \cite[Theorem 1.8]{Rubinsteinnotes}).  By strong irreducibility, any degeneration of the min-max sequence could only be along neck-pinches bounding disks in \emph{one} of the handlebodies.   If this degeneration occurs, one then can remove the handlebodies bounded by the several minimal surfaces to obtain a manifold with minimal boundary $M'$.   As one of the minimal boundary components should have index $1$, one could minimize area for the unstable component of $\partial M'$ into $M'$ to get a new manifold $M''$ with stable boundary.  One then applies min-max to $M''$ and iterates.  Since $M''$ has stable boundary, and the min-max limit should have an unstable component, at each stage of the iteration the manifold shrinks.   If the process does not stop, one obtains infinitely many nested minimal surfaces with bounded genus, which gives rise to a Jacobi field.  If the metric is bumpy (which White proved is a generic condition) then this gives a contradiction.  Thus the process stops after finitely many steps at a minimal surface isotopic to the Heegaard surface.

The argument sketched by J. Pitts and J.H. Rubinstein was incomplete on  two points.   First, they assume that the min-max surface contains the topological information of the sweepouts in the sense that it arises from neck-pinch surgeries on the Heegaard surface.  This was proved by the first-named author \cite{Ketgenusbound}. Secondly, in order to run the iteration scheme, they suggest to apply the min-max theorem to a subdomain $M''$ of the manifold with stable minimal boundary $\partial M''$.  The key claim is that one can obtain a minimal surface in the \emph{interior} of such a subdomain. This needs to be justified as a min-max procedure may just give rise to the boundary $\partial M''$ where some $2$-sphere component may have positive integer multiplicity (see Theorem \ref{smoothminmax}). 



In this paper, we complete this second ingredient of J. Pitts and J.H. Rubinstein's program and prove their conjecture (and in fact, a bit more):

\begin{thm}  \label{introthm}
Let $(M,g)$ be a closed oriented $3$-manifold. Suppose that there is a strongly irreducible Heegaard splitting $H$ which, in the case where $M$ is a $3$-sphere, is supposed to be a $2$-sphere. Then one of the following holds:
\begin{enumerate}[label=\roman*)]
\item H is isotopic to a minimal surface $\Sigma_1$ of index at most one
\item H is isotopic to the boundary of a tubular neighborhood of a non-orientable minimal surface $\Sigma_2$ with a vertical handle attached and the double cover of $\Sigma_2$ is stable.
\end{enumerate}
If moreover the metric is bumpy, we can assume in case i) that $\Sigma_1$ has index $1$ and in case ii) that there is an orientable index $1$ minimal surface $\Sigma_3$ of genus $genus(H)-1$ which is isotopic to the boundary of the tubular neighborhood of $\Sigma_2$.  
\end{thm}

Theorem \ref{introthm} shows that for generic metrics, we can obtain minimal surfaces of index exactly one and with controlled topological type. We emphasize that it is the first such construction beyond the case of minimal 2-spheres in 3-spheres where \emph{no curvature assumption} on the ambient manifold is required 
to obtain minimal surfaces of prescribed topology.

Theorem \ref{introthm} is sharp in that case i) may not hold and instead only the second case occurs. We give an example in Proposition \ref{example} of a metric on $\mathbb{RP}^3$ containing no index $1$ or $0$ Heegaard tori.  The width of this manifold with respect to genus $1$ Heegaard splittings is realized by an index $1$ minimal 2-sphere (the surface $\Sigma_3$ in the notation of Theorem \ref{introthm}) which bounds a twisted interval bundle over a minimal $\mathbb{RP}^2$ with stable double cover.    

Note that case ii) in Theorem \ref{introthm} can often be excluded \emph{a priori} for topological reasons.  For instance, it cannot occur when $H_2(M,\mathbb{Z}_2)=0$ (see for instance Theorem 1 in  \cite{Rubinstein}).  In spherical space-forms other than $\mathbb{RP}^3$, there are no embedded projective planes, and thus case ii) cannot occur when $H$ is a strongly irreducible Heegaard torus in a lens space $L(p,q)\neq \mathbb{RP}^3$.  Note also that the assumption of strong irreducibility in Theorem \ref{introthm} is essential. We give an example in Proposition \ref{counterexample} of a metric on $\mathbb{S}^2\times\mathbb{S}^1$ for which there is no minimal Heegaard torus of index at most $1$ isotopic to the genus $1$ Heegaard surface. S. Montiel and A. Ros \cite{MR} have also shown that in flat tori formed by nearly degenerate lattices, there are no index $1$ minimal surfaces of genus $3$.

Theorem \ref{introthm} gives more information than the original conjecture of J. Pitts and J.H. Rubinstein in that in the second case we prove that the non-orientable surface has stable cover.  This improvement relies on the catenoid estimate of the first-named author, F.C. Marques and A. Neves \cite{KeMaNe}.  We also prove in this case that we still obtain a genus $g(H)-1$ orientable minimal surface embedded in $M$, which was not part of J. Pitts and J.H. Rubinstein's original claim. 


Our main result also has several applications in 3-manifold topology.  As explained in \cite{Rubinsteinnotes}, it implies that there are only finitely many irreducible Heegaard surfaces of a given genus (up to isotopy) in a hyperbolic non-Haken 3-manifold.  Such a result was conjectured by F. Waldhausen in the 70s for all non-Haken 3-manifolds.  Waldhausen's conjecture was proved by T. Li (\cite{TaoLI1} \cite{TaoLI2} \cite{TaoLI3}) using \emph{almost normal surfaces} which are the combinatorial analog of index $1$ minimal surfaces. Thus Theorem \ref{introthm} gives an analytic proof of T. Li's theorem in the hyperbolic case.  T.H. Colding and D. Gabai  \cite{ColdingGabai} used the correspondence between Heegaard splittings and minimal surfaces provided by Theorem \ref{introthm} to give effective versions of T. Li's theorem in hyperbolic manifolds.  With the first-named author, T.H. Colding and D. Gabai \cite{ColdingGabaiKetover} completed the classification problem for Heegaard splittings of non-Haken hyperbolic 3-manifolds using a $2$-parameter min-max argument.

Specializing Theorem \ref{introthm} to lens spaces, we obtain:

\begin{coro} \label{feinherb}
\begin{enumerate}
\item Any lens space not diffeomorphic to $\mathbb{S}^3$ or $\mathbb{RP}^3$ contains a minimal torus with index at most one.
\item Any $(\mathbb{RP}^3,g)$ contains either a minimal torus of index at most one or a minimal projective plane with stable universal cover. If the metric is bumpy, then either there is an index 1 minimal torus or there is an index 1 minimal sphere.

\end{enumerate}
 \end{coro}

The mapping degree method of B. White \cite{Whitebumpy} gives the existence of a minimal torus in every $\mathbb{RP}^3$ with positive Ricci curvature. This result was improved in \cite[Theorem 3.3]{KeMaNe} to obtain an index 1 minimal torus in such a manifold.

The problem of obtaining a lower index bound for minimal surfaces produced by min-max methods was studied by F.C. Marques and A. Neves in \cite{MaNeindexbound} where they proved that in the Almgren--Pitts setting, generically two-sided min-max minimal hypersurfaces coming from $1$-parameter sweepouts have index 1.  This led to the question of whether generically the index of min-max hypersurfaces coming from $k$-parameters sweepouts was equal to $k$. In the context of Allen-Cahn min-max theory it was proved by O. Chodosh and C. Mantoulidis \cite{ChodMant} in dimension $3$.  F.C. Marques and A. Neves proved these index bounds in the Almgren--Pitts setting for dimensions $3$ to $7$ \cite{MaNemultiplicityone}, by relying on the solution of the multiplicity one conjecture by X. Zhou \cite{Zhou19}. In our theorem, however, the topology of the minimal surfaces of index 1 produced can be precisely described, which is not possible in the works previously mentioned.  

Let us finally discuss the situation where the ambient manifold has positive scalar curvature.  Because the only stable orientable minimal surfaces in such manifolds are 2-spheres, we obtain an improvement of Corollary \ref{feinherb}:

\begin{coro}
In a spherical space form not diffeomorphic to $\mathbb{S}^3$ or $\mathbb{RP}^3$ with positive scalar curvature, any genus $1$ Heegaard surface admits an index 1 minimal surface in its isotopy class.
\end{coro}

In positive scalar curvature, minimal surfaces of index at most 1 enjoy area upper bounds: for general area bounds, see  \cite{SchoenYau}, \cite[Proposition A.1 (i)]{MaNe}, and for the existence of minimal surfaces with strong area bounds, see \cite[Theorem 1.2]{MaNe} for 3-spheres with positive scalar curvature and \cite[Theorem 23]{Antoine} for general 3-manifolds with positive scalar curvature. An open question is whether one can find minimal Heegaard splittings with minimal genus and with strong area upper bounds in positive scalar curvature, see Question \ref{question}.

Let us mention some connections between Theorem \ref{introthm} and the literature on minimal surfaces and min-max theory. 
In the case where $M$ is a 3-sphere, Theorem \ref{introthm} was used \cite{HasKet} by R. Haslhofer and the first-named author to show that if a 3-sphere endowed with a bumpy metric contains a stable minimal 2-sphere it contains at least \emph{three} disjoint minimal 2-spheres. Z.C. Wang and X. Zhou \cite{WZ24}, building on the work of L. Sarnataro and D. Stryker \cite{SarnataroStryker2025}, obtained a multiplicity one theorem for unstable min-max minimal surfaces in the Simon-Smith setting; they used this result to show that a $3$-sphere with a bumpy metric or positive Ricci curvature contains \emph{four} distinct minimal spheres, resolving a conjecture of S.T. Yau in these cases. Note that the argument in \cite{WZ24} in the bumpy case also relies on the aforementioned results of \cite{HasKet} and thus our Theorem \ref{introthm}.  

A large part of this paper concerns constructing minimal surfaces inside the interior of a compact manifold with stable minimal boundary. The danger is that one only obtains from a min-max procedure the boundary itself with an integer multiplicity. This issue is recurrent in min-max theory and is related to the general problem of controlling the multiplicity of min-max hypersurfaces. It appears for instance in a paper of F.C. Marques and A. Neves on minimal surfaces in 3-spheres with positive scalar curvature \cite[Theorem 2.1]{MaNeindexbound}, in the third-named author's solution to Yau's abundance conjecture for minimal surfaces in 3-manifolds \cite[Theorem 10]{Son23a} \cite{Antoinenoncompact}, and in other settings as well \cite{HasKet,SZ21,WZ24}.  While the multiplicity one results of \cite{WZ24} rule out unstable minimal surfaces occurring with multiplicity greater than $1$, the question of whether (generically) stable minimal surfaces can occur with multiplicity remains unresolved in general.

\vspace{2em}
\indent
Let us now sketch some of the main ideas in the proof of Theorem \ref{introthm}. 

As remarked above, the essential missing piece in the iteration scheme of J. Pitts and J.H. Rubinstein is to show that when running a min-max procedure in a manifold $N$ with strictly stable minimal boundary, one can find a minimal surface in the interior of $N$.  For simplicity, consider the case where $N$ is a 3-ball and $\partial N$ is a strictly stable minimal 2-sphere.  One can sweep out $N$ by embedded 2-spheres and hope to produce an embedded minimal 2-sphere inside $N$.   Since $\partial N$ is minimal, it acts as a good barrier and so there is no difficulty in running the min-max process in this local setting.  The serious difficulty, however, is that one might just obtain $\partial N$ with some integer multiplicity $k\geq 1$.  One can easily rule out the case $k=1$ using the Squeezing map (Section \ref{sec:squeezing}) of F.C. Marques and A. Neves \cite{MaNeindexbound}.

We argue by contradiction and suppose that applying the min-max process to the 3-ball $N$ gives $\partial N$ with multiplicity $k>1$.  Consider a sequence of sweepouts $\{\Sigma^i_t\}_{t\in [0,1]}$ so that $\sup_{t\in[0,1]} \mbox{Area}(\Sigma^i_t)$ approaches the min-max value 
\begin{equation}
W_N:= k\mbox{Area}(\partial N)
\end{equation} 
as $i\rightarrow\infty$.  As a min-max limit is the varifold $\partial N$ counted with multiplicity $k$, suppose that for $i$ large and some $t_0\in (0,1)$, the surfaces $\{\Sigma^i_t\}_{t\in [t_0-\varepsilon_i,t_0+\varepsilon_i]}$ are within a fixed $\eta>0$ neighborhood of the varifold $\partial N$ with multiplicity $k$ (in some metric on the space of varifolds).  Suppose for simplicity that $[t_0-\varepsilon_i,t_0+\varepsilon_i]$ is the only such interval.  Using the area non-increasing Squeezing Map (see Section \ref{sec:squeezing}) in the neighborhood of a strictly stable minimal surface  we can also ensure that
\begin{equation}\label{arealow1}
\mbox{Area}(\Sigma^i_{t_0-\varepsilon})< W_N 
\end{equation}
and
\begin{equation}\label{arealow2}
\mbox{Area}(\Sigma^i_{t_0+\varepsilon})< W_N.
\end{equation}

The key observation is that either $\{\Sigma^i_t\}_{t\in [0,t_0-\varepsilon_i]}$ or $\{\Sigma^i_t\}_{t\in [t_0+\varepsilon_i,1]}$ must \emph{itself} be a sweepout of the entire manifold $N$ (apart from a tubular neighborhood of $\partial N$).    If $\Sigma^i_0$ is equal to $\partial N$ and $\Sigma^i_{t_0-\varepsilon_i}$ is weakly close to, say, an even multiple $k$ of $\partial N$, then $\Sigma^i_0$ bounds a region of zero volume on one side and a region of $\mbox{vol}(N)$ on the other side while $\Sigma^i_{t_0-\varepsilon_i}$ has the opposite property -- the zero volume side has now become nearly the entire manifold $N$ and vice versa.  Thus if $k$ is even, then $\{\Sigma^i_t\}_{t\in [0,t_0-\varepsilon_i]}$ gives a non-trivial sweepout.  Let us suppose without loss of generality that this is the case.

\begin{figure}
   \centering	
	\includegraphics[scale=0.75]{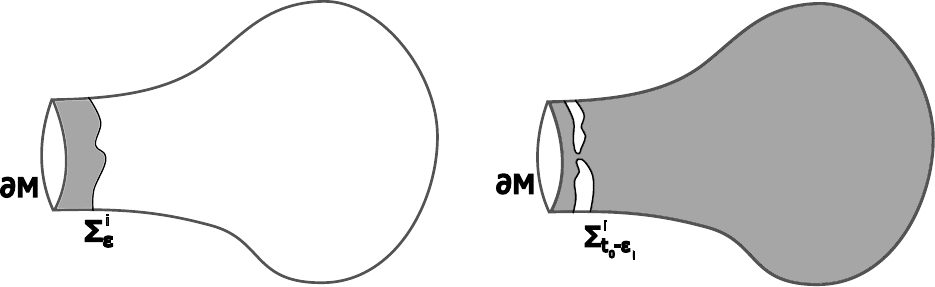}
	\caption{If $\Si^i_0 = \partial M$ and $\Si^i_{t_0- \varepsilon_i}$ is weakly close to $\partial M$ with multiplicity two, then $\{\Sigma^i_t\}_{t\in [0,t_0-\varepsilon_i]}$ sweeps out most of $M$.}
	\label{fig:sweepout}
\end{figure}


We then need to construct a sweepout supported near $\partial N$ that begins at the surface $\Sigma^i_{t_0-\varepsilon_i}$ (which weakly resembles $k$ copies of $\partial N$) and ends at the zero or trivial point surface.  More precisely, given any $\varepsilon>0$ we need an \emph{interpolating} family of surfaces $\{\Gamma_t\}_{t\in[0,1]}$ satisfying:

\begin{enumerate}[label=(\alph*)]
\item $\Gamma_0 = \Sigma^i_{t_0-\varepsilon_i}$ \label{1}
\item $\Gamma_1 = \mbox{trivial point surface}$ \label{2}
\item$\mbox{Area}(\Gamma_t)\leq \mbox{Area}(\Sigma^i_{t_0-\varepsilon_i}) +\varepsilon$. \label{3}
\end{enumerate}

Choosing $\varepsilon$ appropriately small, we can concatenate $\{\Sigma^i_t\}_{t\in [0,t_0-\varepsilon_i]}$ and $\{\Gamma_t\}_{t\in[0,1]}$ to obtain a new sweepout of $N$ with all areas less than $W_N$ (thanks to \eqref{arealow1} and \eqref{arealow2}).  This gives a contradiction to the definition of width.  The conclusion is that $N$ contains in its interior a minimal surface (with index at most $1$ by earlier work of Marques--Neves).  The interpolation result is discussed further in Section \ref{zero inter} and proved in Sections  \ref{first inter} and \ref{second inter}.  In brief, by Alexander's theorem there is always an isotopy satisfying \ref{1} and \ref{2} supported in a neighborhood of $\partial N$ (diffeomorphic to $\mathbb{S}^2\times [0,1]$) and the difficulty is to perform the isotopy satisfying the area constraint \ref{3}.  

Our proof of Theorem \ref{introthm} differs from J. Pitts and J.H. Rubinstein's original sketch in that rather than use an iteration procedure, we remove a maximal disjoint union of undesired handlebodies that could appear after min-max to obtain a compact $3$-manifold $N$ with stable minimal boundary which we call the core (see Subsection \ref{corecore}).  We then apply the local min-max theorem to this core $N$.

The organization of this paper is as follows.  In Section \ref{topuses} we introduce the necessary topological facts about strongly irreducible Heegaard splittings.   In Section \ref{zero inter} we introduce the Interpolation Theorem needed in the proof of Theorem \ref{introthm}. In Section \ref{first inter} we begin the proof of the Interpolation Theorem by showing that an embedded surface in a neighborhood of a strictly stable minimal surface $\Sigma \subset \partial M$ can be deformed through surfaces of controlled areas to a stack of $k$ graphs joined by a `thin' set comprising the necks.  In Section \ref{second inter} we complete the proof of the Interpolation Theorem 
by finding a neck to open to further reduce the number of sheets $k$ until the number of stacked sheets is either $1$ or $0$.  
In Section \ref{prelim} we prove the Local Min-Max Theorem allowing us to obtain a minimal surface in the interior of a manifold $N$ with strictly stable minimal boundary.  In Section \ref{Rubinstein conj} we prove Pitts--Rubinstein's conjecture (Theorem \ref{introthm}).  In Section \ref{sec8} we discuss optimality and important instances of our main result.

\subsection*{Acknowledgement} 

We would like to thank Toby Colding, Fernando Cod{\'a} Marques and Andr{\'e} Neves for their encouragement and numerous discussions, Francesco Lin for some topological advice and Dave Gabai for several conversations. We also thank Yangyang Li and Zhihan Wang, Rafael Montezuma, Amitesh Datta, Maggie Miller, Davi Maximo for useful comments.

We are grateful to the anonymous referee for numerous detailed comments and corrections that greatly helped to improve the exposition.

\vspace{2em}

\section{Strongly irreducible Heegaard splittings}\label{topuses}
In this section, we collect some properties about Heegaard splittings and show that performing surgeries on a strongly irreducible Heegaard surface can only make the surface degenerate in particular ways.  

\subsection{} \textbf{Topological preliminaries}\label{topprelim}

We begin with some basic definitions and notation in topology. A 3-manifold $W$ is a \textit{compression body} if there is a connected closed oriented surface $S$ such that $W$ is obtained from $S\times[0,1]$ by attaching $2$-handles along mutually disjoint loops in $S\times\{1\}$, and filling in some resulting $2$-sphere boundary components (not contained in $S\times\{0\}$) with $3$-handles. We do not require that all $2$-sphere boundary components are filled in. Denote $S\times\{0\}$ by $\partial_+W$ and $\partial W \backslash \partial_+W$ by $\partial_-W$. A compression body is called a \textit{handlebody} if $\partial_-W=\varnothing$. It is said to be \textit{trivial} if $W$ is homeomorphic to $\partial_+ W\times [0,1]$.

Let $M$ be a connected compact oriented $3$-manifold. When $M$ is closed, an embedded connected orientable surface $H$ is a \textit{Heegaard splitting} if $M\backslash H$ has two connected components each diffeomorphic to a handlebody. This notion can be generalized to the case where possibly $\partial M\neq \varnothing$ as follows. Let $(\partial_0M,\partial_1 M)$ be a partition of the boundary components of $M$. A triplet $(W_0,W_1,H)$ is called a \textit{generalized Heegaard splitting} of $(M,\partial_0M,\partial_1M)$ if $W_0$, $W_1$ are compression bodies with 
$$W_0\cup W_1=M, \quad \partial_-W_0=\partial_0M, \quad \partial_-W_1 = \partial_1M$$
 $$\text{ and } W_0\cap W_1 = \partial_+W_0=\partial_+W_1=H.$$ The surface $H$ is then also called a \textit{generalized Heegaard splitting} of $M$. The \textit{Heegaard genus} of $M$ is the lowest possible genus of a generalized Heegaard splitting of $M$. Recall that any triplet $(M, \partial_0M,\partial_1 M)$ as above admits a generalized Heegaard splitting by \cite[Theorem 3.1.10]{SaSchSch}.

Let $S$ be an orientable closed surface embedded in the interior of $M$ as above. An \textit{essential disk} in $(M,S)$ is a disk $D$ embedded in $M$ such that $D\cap S=\partial D$ and $\partial D$ is an embedded non-contractible curve in $S$. 

A generalized Heegaard splitting $(W_0,W_1,H)$ is called \textit{irreducible} when there are no essential disks $(D_0,\partial D_0) \subset (W_0,H)$, $(D_1,\partial D_1) \subset (W_1,H)$ such that $\partial D_0 = \partial D_1$.   The splitting is \textit{reducible} if it is not irreducible.  

We say that $(W_0,W_1,H)$ is \textit{strongly irreducible} if 
there are no essential disks $(D_0,\partial D_0) \subset (W_0,H)$, $(D_1,\partial D_1) \subset (W_1,H)$ such that $\partial D_0 \cap \partial D_1 = \varnothing$.  A splitting is \textit{weakly reducible} if it is not strongly irreducible.  Note that reducible splittings are also weakly reducible. 

A strongly irreducible Heegaard splitting is automatically irreducible.  Indeed, since every pair of essential disks on opposite sides must intersect, if two such disks had the same boundary curve, after small perturbation one could make them disjoint, which contradicts the strong irreducibility condition.  By Haken's lemma, a compact oriented 3-manifold that admits an irreducible Heegaard splitting is irreducible.

Conversely, in non-Haken manifolds it follows from Casson-Gordon \cite{CassonGordon} that irreducible splittings are strongly irreducible.

\subsection{} \textbf{Surgeries on strongly irreducible splittings}
In this section, we describe the possible neck-pinches that may occur beginning with a strongly irreducible Heegaard surface. 

\subsubsection{Neck-pinch surgeries}\label{defneck-pinches}
Let $\Sigma$ be a closed embedded surface in a 3-manifold  $M$.

\begin{defi} \label{def:neck-pinch}
    We say that an embedded surface $\Si'$ arises from $\Sigma$ via a \emph{neck-pinch surgery} along an embedded curve $\gamma\subset \Sigma$  if 
\begin{enumerate}
    \item $\gamma$ bounds a disk $D$  satisfying $D\cap\Sigma = \partial D = \gamma$; 
    \item $\Sigma \backslash  \Si'$ is an annulus that is a tubular neighborhood about $\gamma$ in $\Sigma$; \label{second}
    \item $\Si'\backslash \Sigma$ consists of two disks $D_1$ and $D_2$;
    \item the symmetric difference $\Si'\bigtriangleup \Sigma$ is a $2$-sphere bounding a closed $3$-ball $B$ containing $D$. \label{four}
\end{enumerate}
\end{defi}

We say the neck-pinch surgery along $\gamma$ is \emph{essential} if $\gamma$ does not bound a disk in $\Sigma$.  Otherwise, we call the neck-pinch \emph{inessential}.  If $\Sigma'$ arises as in Definition  \ref{def:neck-pinch} from $\Sigma$ or if $\Sigma'=\phi_1(\Sigma)$, where $\{\phi_t\}_{t\in [0,1]}$ is an isotopy of $M$ (fixing $\partial M$ pointwise) with $\phi_0=\mbox{id}$ let us write 
$$\Sigma'\prec \Sigma.$$

If $(W_0,W_1,\Sigma)$ is a generalized Heegaard splitting of $(M,\partial_0 M,\partial_1 M)$, and $\Sigma'$ arose from a surgery on $\Sigma$ along $\gamma$ \emph{into $W_0$} (resp. \emph{into $W_1$}) if the disk $D$ and ball $B$ in Definition \ref{def:neck-pinch} are contained in $W_0$ (resp. {in $W_1$}). After a neck-pinch surgery has been performed on $\Sigma$ into $W_0$, the resulting surface $\Sigma'$ (possibly disconnected) is two-sided and bounds closed regions \begin{equation}W_0^1:=\overline{W_0\setminus B}\subset W_0 \mbox{ and } W_1^1:=W_1 \cup  B\supset  W_1.\end{equation}

We will also be considering sequences of neck-pinches and isotopies performed successively on a generalized Heegaard surface which we now describe.  

Suppose $(W_0,W_1,\Sigma)$ is a generalized Heegaard splitting of $(M,\partial_0 M,\partial_1 M)$ and
\begin{equation}\label{sequencecompressions}
\Sigma_i\prec \Sigma_{i-1}\prec...\prec\Sigma_1\prec \Sigma_0=\Sigma.
\end{equation}
Then set $W_0^0:=W_0$ and $W_1^0:=W_1$ and for each $k\in\{0,...,i-1\}$ 
we inductively define regions $W^{k+1}_0$ and $W^{k+1}_1$ as follows.  If $\Sigma_{k+1}$ arises from an isotopy $\phi_t$ on $\Sigma_k$, we set
\begin{equation}
W_0^{k+1}:=\phi_1(W_0^{k})\mbox{ and } W_1^{k+1}:=\phi_1(W_1^k).
\end{equation}
If $\Sigma_{k+1}$ arose from $\Sigma_k$ via a neck-pinch surgery, then there are two cases.  We put
\begin{equation}\label{firstoption}
W_0^{k+1}:=W_0^k\cup B_k\mbox{ and } W_1^{k+1}=\overline{W_1^i\setminus B_k}\end{equation} 
if $\Sigma_{k+1}$ arises from $\Sigma_k$ via a surgery along ball $B_k$ (in item 4 in Definition \ref{def:neck-pinch}) contained in $W_0^k$ (in other words, the neck-pinch is \emph{into} $W_0^k$) or else define
\begin{equation}\label{secondoption} W_0^{k+1}:=\overline{W_0^k\setminus  B_k}\mbox{ and } W_1^{k+1}:=W_1^k\cup B_i,\end{equation} if the ball $B_k$ is contained in $W^k_1$.  

With these definitions, for each $k=\{0,...,i-1\}$ the (possibly disconnected) surface $\Sigma_{k+1}$ bounds region $W_0^{k+1}$ on one side and region $W_1^{k+1}$ on the other.  

By slight abuse of terminology given that the domains are changing, we will say the sequence of compressions in the chain \eqref{sequencecompressions} are all \emph{into} the $W_0$ side (resp. into the $W_1$ side) if for each $k\in \{0,...,{i-1}\}$ so that $\Sigma_{k+1}\prec\Sigma_k$ is a neck-pinch, item \eqref{firstoption} (resp. \eqref{secondoption}) holds.  Similarly, we say the \emph{essential compressions are into the $W_0$ side (resp. $W_1$)} if in the sequence of operations in \eqref{sequencecompressions}, those that are essential neck-pinches are all into the $W_0$ side (resp. $W_1$ side).



Throughout the paper we will also use the following basic fact about neck-pinch surgeries
on a compression body along a compressing disk (Remark 3.1.3 item (6) in \cite{SST}):

\begin{lemma} \label{lem: handlebody surgery}
Let $W$ be a handlebody (resp. a compression body) and $D \subset W$ a properly embedded 
disk with $\gamma=\partial D \subset \partial W$. Then, after performing a neck-pinch surgery along $D$, we obtain
a union of handlebodies (resp. compression bodies). 


\end{lemma}

\subsubsection{Tubular neighborhoods} 
Let $(M,g)$ be an oriented, connected, compact Riemannian manifold.
For any subset $A\subset M$ and $h>0$, we will use the following standard notation for its open $h$-neighborhood:
\begin{equation}\label{nh1}
N_h(A) = \{x\in M| \dist_g(x,A) <h\}
\end{equation} 
where $\dist_g$ is the Riemannian distance in $(M,g)$.

Let $\Sigma$ be a surface embedded in $M$. 
Let 
\begin{equation} \label{expon1}
\exp_{\Si}: \Sigma \times \mathbb{R} \rightarrow M
\end{equation}
denote the normal exponential map. For any $\eta>0$ smaller than the injectivity radius of $\Si$, 
$$N_\eta(\Sigma) = \exp_{\Si}(\Sigma \times (-\eta,\eta))$$ 
and it is diffeomorphic to $\Sigma\times [0,1]$ if $\Sigma$ is orientable, or a twisted I-bundle over $\Sigma$ if $\Sigma$ is non-orientable.

If $\Sigma$ is orientable, for any $-\eta<d<\eta$, let us denote the parallel surface
\begin{equation} \label{sd def1}
S_d(\Sigma):= \{\exp_x(d n(x))| x\in\Sigma\}, 
\end{equation}
where $n(x)$ denotes a choice of unit normal along $\Sigma$.  

If $\Sigma$ is non-orientable, then the unit normal $n(x)$ is locally defined, and for any $-\eta<d<\eta$ we set
\begin{equation}\label{sd def2}
S_d(\Sigma):= \{\exp_x(\pm d n(x))| x\in\Sigma\}.
\end{equation}
In this case, $S_d(\Sigma)$ is a connected orientable surface which is a double cover of $\Sigma$ under nearest point projection.

\subsubsection{Surgeries on Heegaard splittings}
We first need the following key lemma due to Scharlemann (Lemma 2.2 in \cite{Scharlemann}):

\begin{lemma}[No Nesting - Lemma 2.2 in \cite{Scharlemann}]\label{nonesting}
Suppose $(W_1,W_2, \Sigma)$ is a strongly irreducible Heegaard splitting of a closed 3-manifold $M_0$ and $F$ is a disk in $M_0$ transverse to $\Sigma$ with $\partial F\subset \Sigma$.  Then $\partial F$ also bounds a disk in some $W_i$.
\end{lemma}

Since the 3-manifold $M_0$ in the statement above is closed,  the two disks in the lemma above are isotopic since such $M_0$ has to be irreducible if it admits a strongly irreducible Heegaard splitting.
Using Lemma \ref{nonesting}, we show that all neck-pinch surgeries performed on a strongly irreducible Heegaard surface must be into the same side:

\begin{lemma}\label{sameside}
Let $(W_0,W_1,\Sigma)$ be a strongly irreducible Heegaard splitting of a closed manifold $M_0$.
Suppose a sequence of neck-pinch surgeries is performed on the surface $\Sigma$.  Then all essential neck-pinch surgeries are either into $W_0$ or into $W_1$.
\end{lemma}

\begin{proof}
To prove the lemma, it is enough to check it when all neck-pinch surgeries we consider are essential.

Suppose $\Sigma'$ is obtained from $\Sigma$ after  a sequence of essential neck-pinch surgeries  along curve $\gamma_1,...,\gamma_k,...,\gamma_K$ into the $W_0$ side for an integer $K \geq 1$. Recall that each neck-pinch surgery removes an annulus region in the surface and glues two new disks $D_{1,k},D_{2,k}$ along the boundary circles. Together, the annulus, $D_{1,k}$ and $D_{2,k}$ bound a 3-ball $B_k$ (see Definition \ref{def:neck-pinch}). 
Let us further assume that the 3-balls $B_k$ are all pairwise disjoint and that after an isotopy, 
$\Sigma'$ bounds a region $\overline{W_0\setminus \bigcup_{k=1}^K B_k}$ on one side and $W_1\cup \bigcup_{k=1}^K B_k$ on the other.

Suppose that the next essential neck-pinch surgery performed on $\Sigma'$ is along $\gamma_{K+1}\subset\Sigma'$. By a slight perturbation of $\gamma_{K+1}$, we may assume that, for each $k=1,...,K$, $\gamma_{K+1}$ is disjoint from both $p_{1,k}\in D_{1,k}$, and $p_{2,k}\in D_{2,k}$, where $D_{1,k}$ and $D_{2,k}$ are the disks that were added in the surgery process at step $k=1,...,K$ (Section \ref{defneck-pinches}) and $p_{i,k}$ is an arbitrary point in $D_{i,k}$ for $i=1,2$.


For each $k=1,...,K$, using expansions based at $p_{1,k}$ and $p_{2,k}$, we may push $\gamma_{K+1}$ off of $D_{1,k}\cup D_{2,k}$ and thereby obtain an isotopic curve $\tilde{\gamma}_{K+1}\subset \Sigma$ which still bounds a disk $F$ in $M$, but which is disjoint from the annuli of the neck-pinch surgeries (see (2) in Definition \ref{def:neck-pinch}).  By Lemma \ref{nonesting}, it follows after applying an isotopy that fixes the boundary $\partial F$ that the disk $F$ is contained entirely in $W_0$ or $W_1$.  Suppose $F$ is contained in $W_1$.   
Since $\tilde{\gamma}_{K+1}$ does not touch any of the 3-balls of the neck-pinch surgeries (see (4) in Definition \ref{def:neck-pinch}),
it follows that $\partial F = \tilde{\gamma}_{K+1}$ is disjoint from $\gamma_1,...,\gamma_K$.
Since $\gamma_1$ bounds a disk in $W_0$, 
this contradicts the strong irreducibility of $\Sigma$.  
Thus $F$ is contained in $W_0$ and by an innermost disk argument, we can ensure it is also contained in $\overline{W_0\setminus \bigcup_{k=1}^K B_k}$. 
This implies that the neck-pinch surgery along $\gamma_2$ is into the $W_0$ side as desired. Moreover, after the surgery along $\tilde{\gamma}_{K+1}$, the corresponding 3-ball $B_{K+1}$ (see (4) in Definition \ref{def:neck-pinch}) can be made to not intersect the other 3-balls $B_k$ ($k=1,...,K$). So after an isotopy, the surface $\Sigma''$ obtained after doing neck-pinch surgery along $\tilde{\gamma}_{K+1}$ bounds a region $\overline{W_0\setminus \bigcup_{k=1}^{K+1} B_k}$ on one side and $W_1\cup \bigcup_{k=1}^{K+1} B_k$ on the other.

By an inductive argument, this finishes the proof.

 \end{proof}

\begin{prop}\label{small volume handlebody}
Let $(M_0,g)$ be a closed Riemannian 3-manifold, and let $M\subset M_0$ be a compact 3-manifold. Then there exists $\beta(M,g)>0$ such that the following holds. Let $\Sigma$ be a closed connected surface in $M$,  obtained after finitely many neck-pinch surgeries on a strongly irreducible Heegaard splitting $H$ of $M_0$ and after discarding some connected components.
If
$$\mbox{Area}(\Sigma)<\beta(M,g),$$ 
then  $\Sigma = \partial K$ where $K$ is a handlebody in $M$ of volume at most $\beta(M,g)$.
\end{prop}

\begin{proof}
Before starting the proof, we recall a standard fact: small surfaces bound small volume regions
(Lemma 1 in \cite{MSY}): 
there are constants $\mu(M,g)>0$ and $c(M,g)>0$ such that if $S$ is a closed embedded surface in $M$ with $\mbox{Area}(S)<\mu(M,g)$, then there is a unique compact domain $K$ contained in $M$ such that
$$S=\partial K \text{ and } \Vol(K)\leq c(M,g) \mbox{Area}(S)^{3/2}.$$

Thus, if $\beta(M,g)$ is a small enough constant, if $\Sigma$ is a surface as in the statement, then $\Sigma=\partial K$ where $K$ is a domain in $M$ with volume at most $\beta(M,g)$. It only remains to show that $K$ is a handlebody.

Consider a triangulation $T$ of $M$.  
If $\beta(M,g)$ is small enough depending only on $(M,g)$, using the coarea formula, we can deform slightly $T$ in a way that $K$ becomes disjoint from the vertices and edges of $T$ (namely its $1$-skeleton), and $\Sigma$ intersects the $2$-dimensional faces of $T$ transversally along embedded closed curves $c_1,...,c_k$. For any curve $c_j$ that is not essential in $\Sigma$, we can deform $\Sigma$ with an isotopy in $M$ so that the newly obtained surface still bounds a region $K$ which remains disjoint from the 1-skeleton of $T$, and the curve components of $T\cap \Sigma$ are now a subset of $\{c_1,...,c_{j-1},c_{j+1},...,c_k\}$. Repeating this finitely many times, we end up with a surface called $\Sigma'$, which bounds a region called $K'$ disjoint from the 1-skeleton of $T$, and $\Sigma'$ intersects the 2-dimensional faces of $T$ transversally along curves which are all essential in $\Sigma'$. 
Each $2$-dimensional face $F$ of $T$ intersects $\Sigma'$ along  curves $\gamma_{F,1},...,\gamma_{F,I_F}$ bounding disks $D_{F,1},...D_{F,I_F}\subset F$ (such a family of disks could be empty).
By Lemma \ref{sameside}, all essential surgeries occur on the same side of the Heegaard splitting, so the disks $D_{F,1},...D_{F,I_F} $ are all \emph{disjoint}. Indeed, if one disk $D_{F,i_1}$ is contained in another $D_{F,i_2}$, we could perform essential surgeries along all the disks but $D_{F,i_1}$ and $D_{F,i_2}$ lie on different sides of the Heegaard splitting, a contradiction. 
Any such disk $D_{F,i}\subset F\cap \Sigma'$, if it exists, is necessarily contained in  $K'$ because $K'$ does not touch the $1$-skeleton of $T$. In other words, $K'$ is on the same side as all the essential neck-pinch surgeries (if any occurred).

By Lemma \ref{lem: handlebody surgery}, $K'$ is a handlebody. Since $K'$ is isotopic to  $K$, $K$ is a handlebody as desired.

\end{proof}

Before we proceed, let us recall fundamental results of Scharlemann-Thompson (Theorem 2.11 in \cite{ScharlemannThompson}) and Heath (Corollary 1.6 in \cite{Heath}) on the classification up to isotopy of irreducible Heegaard splittings of $\Sigma\times I$ or the nontrivial $I$-bundle over $\Sigma$, where $\Sigma$ is a closed surface.

\begin{prop} (Theorem 2.11 in \cite{ScharlemannThompson})\label{st}
Let $\Sigma$ be an orientable surface of genus at least $1$ and let $I=[0,1]$.  An irreducible Heegaard splitting of $M=\Sigma\times I$ endowed with a Riemannian metric $g$ is isotopic to one of the following:
\begin{enumerate}
\item $\Sigma\times\{1/2\}$;
\item  for small $\varepsilon>0$, $\partial\big(N_\varepsilon(\Sigma\times \{1/2\})\setminus \{D\times [0,1]\}\big)$, where $D$ is a closed disk in $\Sigma$.
\end{enumerate}
\end{prop}

We also need the corresponding result of Heath for the non-orientable case.  For a non-orientable surface $\Sigma$, denote by $M_\Sigma$ the non-trivial $I$-bundle over $\Sigma$.  Thus $\partial M_\Sigma$ is an orientable surface that is a double cover of $\Sigma$.
\
\begin{prop} [Corollary 1.6 in \cite{Heath}]\label{heath}
Let $\Sigma$ be a non-orientable surface, and let $M_\Sigma$ be the twisted $I$-bundle over $\Sigma$. 
An irreducible Heegaard splitting of $M_\Sigma$ is isotopic to 
$\partial M_\Sigma$
with a single vertical handle attached.

\end{prop}


Recall that the notation $S_d(\Sigma)$ was introduced in (\ref{sd def1}) and (\ref{sd def2}).
The following proposition shows how a strongly irreducible splitting may degenerate after neck-pinch surgeries:
 \begin{prop}[Surgeries on strongly irreducible Heegaard surfaces]\label{strongirreduce}

Let $\G$ be a strongly irreducible Heegaard surface in a closed, oriented 3-manifold $M_0$ endowed with any Riemannian metric $g$.  Suppose that, after finitely many neck-pinch surgeries performed on $\G$ and possibly discarding some connected components, we obtain an embedded surface $\G'$ so that there exists finitely many closed embedded surfaces $\Si_1, ...,\Si_k$, and positive integers $n_1, ..., n_k$ as well as small real numbers $\{d_{ij}\}$ 
 (where $i\in\{1,2,...,k\}$, and for each fixed $i$, $j$ varies from $1$ to $n_i$) so that $\G'$ is isotopic to 
 \begin{equation}
 \bigcup_{i=1}^k \cup_{j=1}^{n_i} S_{d_{ij}}(\Si_i).
 \end{equation}
 \noindent
  Then, if $M_0$ is not diffeomorphic to $\mathbb{RP}^3$, the following statements hold:
 \begin{enumerate}
 \item If for some $i\in\{1,...,k\}$ the surface $\Si_i$ is non-orientable then $n_i=1$ and $\Si_l$ is a 2-sphere for any $l\neq i$.  Moreover, $\G$ is isotopic to 
 the boundary of a tubular neighborhood of  $\Si_i$ with a vertical handle attached
 and $M_0\setminus \Si_i$ is a handlebody of genus one less than that of $\Gamma$.
 
 \label{first_surgeries}
 \item For any $i$, if $\Sigma_i$ is orientable and of positive genus, then $n_i=1$. \label{second_surgeries}

 \end{enumerate}
 If $M_0$ is diffeomorphic to $\mathbb{RP}^3$ the same statements hold except for the assertion in (1) that $n_i=1$.

  \end{prop}
 
 
 \begin{proof}

 Suppose that $M$ is not diffeomorphic to $\mathbb{RP}^3$. Let us first prove statement \eqref{second_surgeries}.  Thus we suppose toward a contradiction that $\Si_1$ is orientable and $n_1>1$.  In a tubular neighborhood about $\Si_1$, the surface $\G'$ consists of $n_1$ parallel ordered sheets each isotopic to $\Si_1$.  Let us denote these components by $\Phi_1,...,\Phi_k$. 
The sum of the genus of those surfaces $\Phi_1,...,\Phi_k$ is at most the genus of $\G$ since $\G'$ was obtained from $\G$ by neck-pinch surgeries.
In particular, in obtaining the surface $\G'$ from $\G$, some of the surgeries must have been essential surgeries.  By Lemma \ref{sameside}, we can suppose all essential surgeries have been performed into one side of $\Gamma$, say the handlebody $H_1$, while non-essential surgeries can be performed on either side.  By Lemma \ref{lem: handlebody surgery} a sequence of neck-pinch surgeries divides $H_1$ into handlebodies (minus some balls) $H''_1, ..., H''_l$ whose total genus is less than the genus of $\Gamma$, from which a (potentially empty) union of balls have been removed.

Let us first rule out the case $n_1>2$. In this case, there would be at least three consecutive sheets $\Phi_1, \Phi_2$, and $\Phi_3$.  Let $R_1$ be the open set bounded between $\Phi_1$ and $\Phi_2$, and let $R_2$ be the open set bounded between $\Phi_2$ and $\Phi_3$.  Thus either $R_1$ or $R_2$ must be among the collection $\{H''_1,...,H''_k\}$.  This is a contradiction as, for each $i=1,2$,  $R_i$ is diffeomorphic to $\Si_1\times (0,1)$ which is not a handlebody (minus some balls).

It remains to rule out the case where $n_1=2$.  Without loss of generality, assume that $\Phi_1$ and $\Phi_2$ are on different sides of $\Si_1$.  In this case, the open set \begin{equation}M_0\setminus (\Phi_1\cup\Phi_2)\end{equation} consists of three components: two handlebodies and the region $R_1$ between the two sheets $\Phi_1$ and $\Phi_2$, which is diffeomorphic to $\Si_1\times (0,1)$.  After renumbering, $\Phi_1 \subset \partial H''_1$ and $\Phi_2 \subset \partial H''_2$.  Thus  $\Gamma$ is constructed (up to isotopy) from the disconnected set $\Phi_1\cup\Phi_2$ by attaching $1$-handles contained in $R_1$ with boundary circles in $\partial H''_1$ and $\partial H''_2$, since the essential compressions have all been into the $H_1$ side by assumption.

We identify $R_1$ with $\Sigma_1\times(0,1)$. 
It follows from the classification of irreducible Heegaard surfaces in manifolds diffeomorphic to $S\times [0,1]$, where $S$ is a closed orientable surface (Proposition \ref{st}) that the irreducibile Heegaard surface $\Gamma$ is obtained from $\partial H''_1$ and $\partial H''_2$  by attaching a single $1$-handle along a vertical arc joining $\partial H''_1$ and $\partial H''_2$.  Thus the Heegaard surface $\Gamma$ bounds on one side the handlebody 
\begin{equation}
H_2:=(\Sigma_1\setminus D)\times [0,1], 
\end{equation}
where $D$ is a small disk in $\Sigma_1$, and on the other side, the handlebody 
\begin{equation}
H_1= H''_1\cup H''_2 \cup V,
\end{equation}
where $V$ is the closure of $D\times (0,1)$.

We now exhibit two disjoint curves $\gamma_1$ and $\gamma_2$ supported on $\G$ so that $\gamma_1$ bounds an essential disk in $H_1$ and $\gamma_2$ bounds an essential disk in $H_2$, which contradicts strong irreducibility.  

Let us consider the set $(\Sigma_1 \setminus D)\times [0,1]$ and suppose the genus of $\Sigma_1$ is $h$.  We can express the surface with boundary $\Sigma_1 \setminus D$ as a regular $4h$-gon $\mathcal{P}$ with a disk $D$ removed from the center, and where the opposite sides of $\mathcal{P}$ are identified by an equivalence relation $R$ in the usual way.   Thus we write
\begin{equation}
H_2=((\mathcal{P}\setminus D)\times [0,1]) /\sim, 
\end{equation}
where $(x,t)\sim(y,s)$ if $t=s$ and $xRy$.  

 First, define $\gamma'$ to be a curve on $\mathcal{P}\setminus D$ with the property that 
 $$\gamma_1:=(\gamma'\times 0)/\sim$$ is an essential curve in $\G=\partial H_1$ bounding an embedded disk in  $H_1''\subset H_1$. 
 
Let $q \in \gamma'$ be the closest point on $\gamma'$ to 
 $D \subset \mathcal{P}$. We remove a tiny segment $\alpha' \subset \gamma'$ containing the point $q$, with endpoints $ \partial \alpha' =\{p_1,p_2\}$. Let $\gamma_2$ be a closed curve obtained by 
 connecting $(\gamma' \setminus \alpha', 0)/\sim $ and 
 $(\gamma' \setminus \alpha', 1)/\sim $ by disjoint embedded curves $\beta_1$ (with $\partial \beta_1  = \{(p_1, 0), (p_1, 1)) \}$) and $\beta_2$ (with $\partial \beta_2  = \{(p_2, 0), (p_2, 1)) \}$) that pass through the ``neck" $\partial D\times [0,1]$ and that only intersect $\gamma_1$ at the endpoints. Observe that by orientability of $\Sigma_1$ we can perturb $\gamma_1$ to one side, and hence, by construction, we can also perturb $\gamma_2$ so that it is disjoint from $\gamma_1$.
Moreover,  $\gamma_1$ bounds an essential disk in $H_1$ and $\gamma_2$ bounds an essential disk in $H_2$ (this disk looks like a thin band trapped between $\Sigma_1\times \{0\}$ and $\Sigma_1\times \{1\}$).  Thus, $\Gamma$ is not strongly irreducible, which gives a contradiction.  We conclude that $n_1=1$.  This completes the proof of statement \eqref{second_surgeries}. 
 
We now address statement \eqref{first_surgeries}.  Suppose $\Si_1$ is non-orientable. Thus, inside a tubular neighborhood of $\Si_1$, the surface $\G'$ consists of a collection of parallel orientable surfaces $\Phi_1,...,\Phi_k$ (with $k=n_1/2$), each of which is isotopic to a double cover of $\Si_1$.  The components of \begin{equation}M_0\setminus \{\Phi_1,...,\Phi_k\}\end{equation} are the following $k+1$ sets:
\begin{itemize}
    \item The twisted interval bundle $R_0$ bounded by $\Phi_1$ about $\Si_1$.
    \item  The sets $\{R_i\}_{i=1}^{k-1}$, where $R_i$ is the region between $\Phi_i$ and $\Phi_{i+1}$.  Each $R_i$ is diffeomorphic to $\tilde{\Si}_1\times [0,1]$ (where $\tilde{\Si}_1$ is an orientable double cover of $\Si)$.
    \item The region $R_k$ comprising the component of $M\setminus\Phi_k$ not containing $\Si_1$.
\end{itemize}
Again by Lemma \ref{sameside}, all essential compressions that were performed to obtain $\G'$ from $\G$ are into the same side $H_1$ of $\Sigma$, while inessential surgeries that split off a 2-sphere component may occur to either side.  If $M_0$ is not diffeomorphic to $\mathbb{RP}^3$, $M_0$ admits no embedded projective planes, and thus $\tilde{\Si}_1$ is \emph{not} a 2-sphere.  It follows that no region among the collection $\{R_0,...,R_k\}$ is diffeomorphic to $\mathbb{S}^2\times [0,1]$ so none of them is a handlebody (minus some balls). Thus we conclude $k=1$ and $n_1=2$ since, by Lemma \ref{lem: handlebody surgery}, the essential neck-pinch surgeries divide $H_1$ into a collection of handlebodies.

It follows that $\G'$ restricted to a tubular neighborhood of $\Si_1$ consists of one connected surface $\Phi_1$ that is a double cover of $\Si_1$ via nearest point projection. Note that $M_0\setminus\Phi_1$ consists of two components, a handlebody and a twisted interval bundle. It follows from Proposition \ref{heath} that the only way one can obtain from $\Phi_1$ an irreducible Heegaard surface is by attaching a single vertical $1$-handle to $\Phi_1$.  This completes the proof when $\Sigma_1$ is non-orientable and $M_0$ is not diffeomorphic to $\mathbb{RP}^3$.

Finally, suppose that $M_0$ is diffeomorphic to $\mathbb{RP}^3$, then Bonahon-Otal (Theorem 1 in \cite{BO}) showed that there is a unique strongly irreducible Heegaard splitting, which has genus $1$.  Moreover, the genus $1$ splitting can be obtained by adding a vertical $1$-handle to the boundary of a tubular neighborhood about any embedded $\mathbb{RP}^2$, thus giving \eqref{first_surgeries} in this case as well, except that we could have 
$n_1>2$ because of attached 2-spheres.

\end{proof}

The following corollary asserts, roughly speaking, that after surgeries on a strongly irreducible surface, we cannot obtain a surface enclosed by another such surface unless the inner surface is a sphere:


\begin{coro} \label{finding_neck}
Let $M_0$ be a closed, oriented $3$-manifold and let $\Si\subset M_0$ be a strongly irreducible Heegaard surface. Let $M\subset M_0$ be a compact three-manifold with boundary with generalized Heegaard splitting $(W_0,W_1,\G)$ so that $\Gamma$ is isotopic to $\Si$ in $ M_0$.  Suppose that after performing finitely many neck-pinch surgeries on $\G$ and after possibly discarding some components, we obtain $S\cup T$, where $S$ and $T$ are disjoint closed connected surfaces embedded in $M$ and $T$ might be empty.

Then $S$ is the positive boundary of a compression body in $M$ with negative boundary in $\partial M$. Similarly, if $T\neq\varnothing$ then $T$ is the positive boundary of a compression body in $M$ with negative boundary in $\partial M$. 

Suppose that $W_S$ is any compression body in $M$ with $\partial_+W_S=S$ and $\partial_-W_S\subset \partial M$, and if $T\neq\varnothing$, that $W_T$ is any compression body in $M$ with $\partial_+W_T=T$ and  $\partial_-W_T\subset \partial M$.  Then the following items hold:
\begin{enumerate}
\item \label{firstnonesting}  
Suppose that $ M_0$ is not a 3-sphere and that $T\neq \varnothing$. 
If $S$ and $T$ both have positive genus, then $W_S\cap W_T=\varnothing$. 
If $W_S\cap W_T \neq \varnothing$ then $W_T \subset W_S$ (or $W_S \subset W_T$) and $W_T$ (resp. $W_S$) is diffeomorphic to a 3-ball minus a possibly empty union of 3-balls.

\item  \label{secondnonesting} 

Suppose that $S\cup T$ was obtained from some connected surface $V$ after one additional neck-pinch surgery (and $V$ was obtained after a sequence of neck-pinch surgeries on $\Gamma$).
If $T\neq \varnothing$ and $W_S\cap W_T =\varnothing$, then $V$ bounds a compression body $W_V$ with $\partial_-W_V =\partial_-W_S \cup \partial_-W_T$.  If $T\neq \varnothing$ and 
 $W_T \subset W_S$, then $V$ bounds a compression body $W_V$ with $\partial_- W_V=\partial_-W_S $. 
 If $T=\varnothing$, then $V$ also bounds a compression body $W_V$ with $ \partial_- W_V=\partial_-W_S$.

\end{enumerate}


    
\end{coro}
\begin{proof}


By Lemma \ref{sameside}, essential neck-pinch surgeries performed on $\G$ are into one fixed side of $\G$. Moreover, by Lemma \ref{lem: handlebody surgery}, after each neck-pinch surgery, one obtains a (possibly empty) union of positive genus surfaces bounding disjoint compression bodies $\{W_1,...,W_k\}$ together with some 2-spheres (the two-sphere bound some $3$-balls which may not be disjoint and may be contained in  compression bodies $W_i$). Hence, if $S$ and $T$ both have positive genus, then they bound disjoint compression bodies $W_i$, $W_j$ among the $\{W_1,...,W_k\}$ where $S=\partial_+W_i$ and $T=\partial_+W_j$. 

 We now prove item (1) assuming first that both $S$ and $T$ have positive genus. In particular, the genus of $S$ or $T$ is strictly less than that of $\G$. Assume that $W_S\subset M$ is a compression body with $\partial_+W_S=S$ and $W_T\subset M$ is a compression body with $\partial_+W_T=T$.  
 
We can assume without loss of generality that either $W_S\neq W_i$ or $W_T\neq W_j$ (otherwise, $W_S\cap W_T$ is empty and the statement follows). Thus assume $W_S\neq W_i$.   As $S$ was obtained after at least one essential neck-pinch surgery on a strongly irreducible Heegaard splitting, the surface $S=\partial_+W_S$ is incompressible in $M\setminus W_i$ by Lemma \ref{sameside}. But the positive boundary of a compression body is incompressible if and only if the compression body is trivial. In other words, the compression body $W_S$ is diffeomorphic to $S\times [-1,1]\setminus \cup_{i=1} B_i^k$ (where $\{B_1,..B_k\}$  a possibly empty collection of 3-balls). Moreover, one component $\{-1\}\times S$ of $\partial W_S$ coincides with $S$ and another component $\{1\}\times S$ is contained in  $\partial M$ and $T\subset W_S\cong S\times [-1,1]\setminus\cup_{i=1}^k B_i$. Adding back the handles and discarded components successively to $S$ and $T$ we obtain a connected surface $\Sigma'$ isotopic to $\Sigma$ but contained in the product region $W_S$.  By Proposition \ref{st}, we get that $\Sigma'$ is either isotopic to $S$ (in contradiction to the fact that the genus of $S$ is less than $\Sigma$), or else to parallel copies of $S$ with a vertical neck joining them.  This is excluded by item (\ref{second_surgeries}) in Proposition \ref{strongirreduce} as $\Sigma'$ is strongly irreducible in a closed 3-manifold.  This completes the case where both $S$ nor $T$ have positive genus, establishing that $W_S\cap W_T=\varnothing$ as desired.


 Suppose now that the genus of $T$ is zero.  Thus $W_T$ is a $3$-ball (minus potentially some balls) with its negative boundary contained in $\partial M$ and equal to a union of two-spheres. If $S\subset W_T$, by a result of Scharlemann \cite[Theorem 2.1]{Scharlemann}, it follows that the genus of $S$ is equal to zero.  Thus in this case $W_S\subset W_T$ (otherwise $M_0 $ is diffeomorphic to $\mathbb{S}^3$ in violation of our assumption).  Suppose $T\subset W_S$. If the genus of $S$ is positive, then by the result of \cite{Scharlemann} again we obtain $W_T\subset W_S$. If $S$ is a two-sphere, we also have that $W_T\subset W_S$ since $M_0$ is not diffeomorphic to $\mathbb{S}^3$.
 This completes the proof of item (\ref{firstnonesting}).
 
Item (\ref{secondnonesting}) follows from item (\ref{firstnonesting}) and Lemma  \ref{lem: handlebody surgery} (recall that if $T\neq \varnothing$ and 
 $W_T \subset W_S$ then by item (\ref{firstnonesting}), $W_T$ is a 3-ball).

 \end{proof}

\vspace{2em}

\section{Interpolation: statement and outline of proof}\label{zero inter}

In this section, we introduce a smooth interpolation result that will be important in the proof of the local min-max theorem (cf. Theorem \ref{smoothminmax} in Section \ref{prelim}), which in turn will imply the main result, Theorem \ref{introthm}.  The interpolation theorem enables us to deform an embedded surface close in the flat topology to a stable minimal surface with multiplicity to a certain canonical position. 

Recall that $N_h(.)$ denotes the $h$-neighborhood (see (\ref{nh1})). The notion of strongly irreducible generalized Heegaard splitting is defined in Section \ref{topprelim}. 
We can now state the smooth interpolation result.




\begin{thm}[Interpolation] \label{maininterpolation}
Let $( M_0,g)$ be a closed oriented Riemannian $3$-manifold, and let $\overline{\G}\subset M_0$ be a strongly irreducible Heegaard surface. 
Let $M\subset M_0$ be a compact three-manifold with partitioned boundary $(\partial_0 M, \partial_1 M)$ and 
let $(W_0,W_1,\G)$ be a generalized Heegaard splitting of $(M,\partial_0 M, \partial_1 M)$ 
so that $\G$ is isotopic to $\overline{\G}$ in $ M_0$. Suppose further that each component of $\partial M$ is a strictly stable minimal surface.  

There exist $h, \varepsilon >0$,
such that the following
holds. 
If the surface $ \Gamma$ satisfies
$$\mbox{Area} (\Gamma \setminus N_{h}(\partial M))< \varepsilon,$$
then for every $\delta>0$, there exists a closed surface  $\Si'$ that coincides with a collection of components of $\partial M$, and
there exists an isotopy $\{\Gamma_t\}_{t\in[0,1]}$, such that:
\begin{enumerate}
\item $\Gamma_0=\Gamma$
\item $\mbox{Area}(\Gamma_t) < \mbox{Area}(\Gamma) + \delta$
\item $\mathbf{F}(\Gamma_1,\Sigma')\leq\delta$. 

\end{enumerate}
\end{thm}

Here, $\mathbf{F}$ denotes the varifold metric
as defined in \cite[page 66]{P}.
It will also be convenient to use the following notation.
Given an embedded surface $S\subset M$   we will use 
 $|S|$ to denote the varifold induced by $S$
 and for a varifold $V$ we will use
 $||V||$ to denote the measure on $M$ induced by $V$. Given a map $F: M \to M$ we will use $F_\sharp(V)$
 to denote the push-forward of varifold $V$ by $F$
 (see \cite[2.2]{C&DL}).

\subsection{The case of connected stable minimal boundary.}

In the case where $M$ is a 3-ball with boundary a strictly stable minimal 2-sphere $\Sigma$, Theorem \ref{maininterpolation} can be thought of as a quantitative form of Alexander's Theorem or a combinatorial version of mean curvature flow performed ``by hand."


It follows from Alexander's theorem that any embedded 2-sphere in $N_\varepsilon(\Sigma)\cong \mathbb{S}^2\times [0,1]$ can be isotoped to either a round point or else to $\Sigma$ itself.  The difficulty is to obtain such an isotopy obeying the area constraint (2).  The reason that the $\delta$-constraint is important is that (as explained in the sketch in the Introduction) we will be gluing this interpolating isotopy into sweep-outs with maximal area approaching the width $W$ and we want the maximal area of the resulting sweepout to still be close to $W$.

It is instructive to consider the analogous question in $\mathbb{R}^3$ to that addressed in Theorem \ref{maininterpolation}. Suppose one is given two embeddings $\G_0$ and $\G_1$ of 2-spheres into $\mathbb{R}^3$.   We can ask whether for any $\delta>0$ there exists
an isotopy $\G_t$ from $\G_0$ to $\G_1$ obeying the constraint (assuming $\mbox{Area}(\G_1) \geq \mbox{Area}(\G_0)$):
\begin{equation}
\mbox{Area}(\G_t)\leq \mbox{Area}(\G_1)+\delta \mbox{ for all } t.
\end{equation}

It is easy to see that the answer is ``yes". 
Namely, one can even do better and find an isotopy satisfying \begin{equation} \mbox{Area}(\G_t)\leq \mbox{Area}(\G_1) \mbox{ for all } t. \end{equation}

To see this, one can first enclose $\G_0$ and $\G_1$ in a large ball about the origin $B_R$.   By Alexander's theorem there is an isotopy $\phi_t$ between $\G_0$ and $\G_1$. Suppose that this isotopy increases area by a factor at most $L^2$ along the way, for some $L>0$.   First shrink $B_R$ into $B_{R/L}$, then perform the shrunken isotopy $(1/L)\phi_t$ on $B_{R/L}$, and then rescale back to unit size.

Of course, in $3$-manifolds that we must deal with in Theorem \ref{maininterpolation} 
one does not have good global radial isotopies to exploit.  However, the same idea of shrinking still applies if we first work locally in small balls to ``straighten" our surface.  We can also use a certain squeezing map (see Section \ref{first inter}) to repeatedly press our surface closer to $\Sigma$ in the $\mathbf{F}$-topology while only decreasing area.  

Let us explain the ideas in our proof of Theorem \ref{maininterpolation} in more detail.
There are two main steps.  In the first step, we introduce
a local area-nonincreasing deformation process in a collection of small balls covering the boundary $\Sigma$.  The end result of this process is to produce a surface isotopic to $\Gamma$ and consisting of $k$ parallel  sheets graphical over $\Sigma$ that are joined by a thin set of potentially knotted and linked tubes.   

The local deformation we introduce exploits the fact that we can use the Shrinking Isotopies to ``straighten" the surface in small balls while obeying the area constraint (a similar idea was used by Colding-De Lellis \cite{C&DL} in proving the regularity of $1/j$-minimizing sequences).   Our deformation process is a kind of discrete area minimizing procedure, somewhat akin to Birkhoff's curve shortening process.  The process ``opens up" any 
 folds or unknotted necks that are contained in a single ball. However,
 at this stage we can not open necks like on Figure \ref{fig:essential}.

After the first stage of the process, we are left with $k$ parallel graphical sheets arranged about $\Sigma$ joined by potentially very complicated necks.  If $k$ is $1$ or $0$, the proposition is proved. If not, the second step is to use a global deformation to deform the surface through sliding of necks to one in which two parallel sheets are joined by a neck contained in a single ball.  Then we go back to Step 1 to open these necks.  After iterating, eventually $k$ is  $1$ or $0$.

In order to find the neck that can be ``opened up"  we use the Light bulb theorem in topology.
The version of the Light bulb theorem that will be most useful to us
is the following
(see Theorem \ref{general}).
Given a 3-manifold $M$ and two arcs, $\alpha$ and $\beta$, with boundary
points in $\partial M$ assume that one of the boundary points of $\alpha$
lies in the boundary component of $M$ diffeomorphic to a sphere. 
Then $\alpha$ and $\beta$ are isotopic as free boundary curves
if and only if they are homotopic as free boundary curves.
We will apply this theorem to a curve $\alpha$ with the property that
a portion of our surface $\G$ (the ``neck")
lies in the boundary of a small tubular neighborhood of $\alpha$.

\begin{figure}
   \centering	
	\includegraphics[scale=0.75]{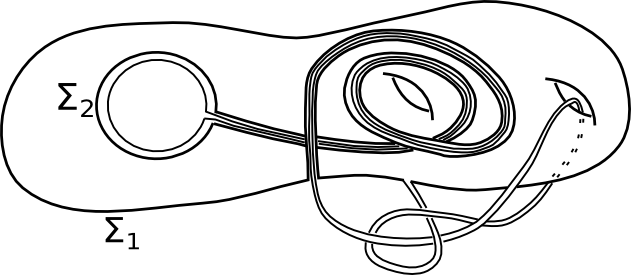}
	\caption{The surface $\G$ is within $\varepsilon$
	(in varifold norm) from $\Si_1 +2 \Si_2$, where $\Si_1$ is a stable
	minimal surface of genus 2 and $\Si_2$ is a stable minimal sphere. We can isotope
	$\G$ to $\Si_1$ while increasing its area by an arbitrarily small amount.}
	\label{fig:knotted}
\end{figure}

\subsection{The case of multiple connected components.}
In Theorem \ref{maininterpolation} we need to consider the case that $\G$ clusters 
around a minimal surface $\Si$ that has multiple connected components.
This is illustrated in Figure \ref{fig:knotted}.
The surface $\G$ is mostly contained in the tubular neighborhood
of minimal surfaces $\Si_1$ and $\Si_2$, while the part of
$\G$ outside of $N_h(\Si_1 \cup \Si_2)$ consists of a thin
set that can link with each other and knot around
handles of $\Si_1$.

After the surface has been deformed into a canonical form in the neighborhood
of each connected component $\Si_i$, we need a global argument, showing that one can always find 
a neck that can be unknotted, using the Generalized Light Bulb Theorem, and slid
into the neighborhood of one of the boundary components $\Si_i$. This process terminates
only when for each $i$ the surface $\G$ either avoids the neighborhood of $\Si_i$
or looks like a single copy of $\Si_i$ with thin necks attached.


\vspace{2em}

\section{Interpolation I: deformation to a stacked surface} \label{first inter}
In this section, let $(M,g)$ be a compact Riemannian 3-manifold with smooth boundary.

\subsection{Isotopy with surgeries.}
To describe our interpolation procedure
it will be convenient to introduce the notion
of isotopies with surgeries.
Recall Definition \ref{def:neck-pinch}
of a neck-pinch surgery. 
It will also be convenient to define a surgery that discards a connected component that is contained in a small ball. We will call this a \emph{collapse surgery}.

\begin{defi} \label{def:collapse}
We will say that an embedded surface $\G'$ in $M$ is obtained
from $\G$ by a collapse surgery if 
there exists an open set $U \subset M$ diffeomorphic to a 3-ball such that 
$\G\cap \partial U=\varnothing$ and 
$\G'= \G \setminus U$.
 \end{defi}

We will define a class of deformations which allows finitely many neck-pinch and collapse surgeries and is continuous in the $\mathbf{F}$-topology.

\begin{defi} \label{def:collapse1}
A family $\{\G_t \}_{t \in [0,1]}$ of smooth closed embedded surfaces is an
\emph{isotopy with $k$ surgeries} in $M$
if there exists a subdivision
$0 =t_0 < t_1 < \dots < t_{k+1} =1$ of $[0,1]$
into $k$ subintervals so that 
for each $[t_i, t_{i+1}]$ the following holds:
\begin{enumerate}
    \item $\{ \G_t \}_{[t_i, t_{i+1})}$ is a smooth isotopy;
    \item $\lim_{t \rightarrow t_{i+1}}\mathbf{F}(\G_t, \G_{t_{i+1}})=0$;
    \item for every $\epsilon>0$ if $t<t_{i+1}$
    is sufficiently close to $t_{i+1}$, then
    $\G_{t_{i+1}}$ is obtained from $\G_t$
    by a neck-pinch or collapse surgery.
\end{enumerate}
\end{defi}

We have the following useful lemma that relates
isotopies with surgeries to isotopies in $M$.

\begin{lemma} \label{neck-pinching isotopy}
Let $\{ \G_t \}_{t=0} ^1$ be an 
isotopy with surgeries. 
For every $\delta>0$ there exists
an isotopy $\{\G_t'\}_{t=0} ^1$ of surfaces in $M$ with
\begin{enumerate}
\item $\G_0' = \G_0$, 
\item $\mbox{Area}(\G_t ' ) < \mbox{Area}(\G_t) + \delta$ for all $t\in [0,1]$,
\item $\mathbf{F}(\G_1, \G_1') < \delta$. 
\end{enumerate}
\end{lemma}

\begin{proof}
The proof is by induction 
on the number of surgeries $k$.
Suppose the statement of the theorem
is correct for $k-1 \geq 0$ surgeries.
Let $\{ \G_t \}_{t=0} ^1$ be a family
with $k$ surgeries at times $t_1<...<t_k$. 

By the inductive assumption there exists
an isotopy $\{ \G_t'' \}_{t \in [t_1,1]}$,
such that $\G_{t_1}''= \G_{t_1}$, 
$\mathbf{F}(\G_1, \G_1'') < \delta/2$
and for all $t \in [t_1,1]$,
$$\mbox{Area}(\G_t '' ) < \mbox{Area}(\G_t) + \delta/2.$$ 
By the ambient isotopy theorem we can extend the isotopy
$\{\G_t'' \}_{t \in [t_1,1]}$ to an ambient isotopy
$\Psi_t: M \rightarrow M$.
For any $\eta>0$ choosing
$\tilde t < t_1$ sufficiently close to $t_1$
we have, by continuity in the $\mathbf{F}$
metric, that $\mathbf{F}(\Psi_t(\G_{\tilde t}), \G_t'') < \eta$
for all $t \in [t_1,1]$.

Choosing $\eta$ sufficiently small and setting
$\G_t' = \G_{t\tilde t \over t_1}$ for $t\leq t_1$
and $\G_t' = \Psi_t(\G_{\tilde t})$ for $t> t_1$
we obtain the desired isotopy.
\end{proof}

In the following lemma, we show that one can perform a neck-pinch surgery on a surface $\G$
using an intersection with a different surface. Recall that $N_h(.)$ denotes the $h$-tubular neighborhood of an embedded surface, see (\ref{nh1}).

\begin{lemma} \label{intersection neck-pinch}
Let $S \subset M$ be an embedded surface with
the second fundamental form $|A_S|_{C^2}< C$. There exists
$\varepsilon_0(C)>0$ with the following property.
Suppose $\G_0$ is an embedded surface
intersecting
$S$ transversally and $\gamma = \G_0 \cap S$ is a finite collection of closed curves
with 
 $\mbox{Length}(\gamma)=l < \varepsilon_0$. 
Then for every $\delta>0$ there exists
an embedded surface $\G_1$ obtained from $\G_0$ by an isotopy with surgeries $\{ \G_t \}_{t \in [0,1]}$, such that 
\begin{itemize}
    \item $\G_1$ is disjoint from $S$;
    \item $\G_t \setminus N_{\delta}(S)=\G_0 \setminus N_{\delta}(S)$;
    \item $\mbox{Area}(\G_t) \leq \mbox{Area}(\G_0) + \frac{l^2}{ \pi}$;
    \item $\F(\G_0, \G_t) \leq \frac{4l^2}{ \pi}$
\end{itemize}
for all $t\in [0,1]$.
\end{lemma}

\begin{proof}
For a point $q \in S$ let $E_{S,q}:T_qS \rightarrow S$ denote the
exponential map of $S$ at $q$ and, for a two-sided subset $U\subset S$,
let $\exp_U: U \times [-h,h]  \rightarrow N_{h}(S)$
denote the normal exponential map in $M$.
Let $$\Phi_q:  \{|(x,y)|\leq \varepsilon_0: (x,y) \in \mathbb{R}^2 \} \times [-h,h] \rightarrow \exp_{B_{\varepsilon_0}(p)}(B_{\varepsilon_0}(q) \times [-h,h])$$
be defined by $\Phi_q(x,y,t) = \exp_{B_{\varepsilon_0}(q)}(E_{S,q}(x,y),t)$.
Given a bound on the $C^2$-norm of the second fundamental form of $S$
we can choose $h $ and $\varepsilon_0$ sufficiently small,
so that $\Phi_q$ satisfies 
$|d\Phi_q|  \in (0.99, 1.01)$ for all $q \in S$. In particular, we have that
$l < \varepsilon_0$ is smaller than the injectivity radius of $S$.

The proof proceeds by induction on the number $k$ of connected components of 
$\gamma = \bigsqcup_{i=1}^k \gamma^i$.
For each connected component $\gamma^i$ of $\gamma$ there exists a unique
disk $D(\gamma^i)\subset S$ bounded by $\gamma^i$ that lies inside a disk of radius $\varepsilon_0$ in $S$. We will say that $\gamma^i$ is innermost if $D(\gamma^i)$ is disjoint from $\gamma^j$ for $j \neq i$. 

Choose an innermost component $\gamma^i$ of $\gamma$ and let $p$ be a point in the interior of $D=D(\gamma^i)$. Let $l_i = \mbox{Length}(\gamma^i)$.
Since $l_i \leq l < \varepsilon_0$ we have that $D \subset S \cap B_{p}(\varepsilon_0)$.
By the isoperimetric inequality in $\R^2$ and our
assumption on $d\Phi_p$ we have that \begin{equation}\label{smalldisk}\mbox{Area}(D)\leq \frac{1.01^2 l_i^2 }{0.99^2 4 \pi}
< \frac{l_i^2}{3 \pi} .\end{equation}

Let $U$ denote a small closed neighborhood of $D$ in $S$, so that
$U \cap \G_0 = \gamma^i$.
Since $\G_0$ intersects $S$ transversally, for
all sufficiently small $h_0\in (0,\min\{h, \delta\})$ there
exists a deformation of $\G_0$ in $N_{h}(U)$ to a
surface $\G'$ with 
\begin{equation}\G' \cap \overline{N_{h_0}(U)}=
\exp_{B_{\varepsilon_0}(p)}(\gamma \times [-h_0,h_0])\end{equation}
through surfaces of area at most $\mbox{Area}(\G) + \frac{l_i^2}{100}$. ($\overline{N_{h_0}(U)}$ denotes the closure of the tubular neighborhood $N_{h_0}(U)$.) Using this ``straightened surface" we will now construct the desired neck-pinch surgery obeying the desired area constraint. Let us give the details.

Define a smooth homotopy of closed curves
\begin{equation}\gamma_t \subset E_{S,p}^{-1}(D )\mbox{ for } t \in [0,1],\end{equation} 
such that $\gamma_t$ is an isotopy for $t \in [0,1)$, $\gamma_0 =  E_{S,p}^{-1}(\partial D ) $ and $\gamma_1 = \{ p\}$.
Let $D_t \subset \mathbb{R}^2$ denote the disk bounded by $\gamma_t$
and $A(t) = D_0 \setminus D_t$.
Let $L = \max_{t\in [0,1)}\mbox{Length}(\gamma_t)$.

Define an isotopy with one neck-pinch surgery
\begin{equation}\label{finaldef}\G_t' =
\begin{cases}
(\G' \setminus \overline{N_{h_0}(U)})
\cup \Phi_p \big(
 A_t \times \{-h_0,h_0 \}
\cup (\gamma_t \times [-h_0,h_0]) 
\big) &\mbox{ for } t\in [0,1);\\
(\G' \setminus \overline{N_{h_0}(U)})
\cup \Phi_p \big(
 D_0 \times \{-h_0,h_0 \}
\big) &\mbox{ for } t=1.
\end{cases}\end{equation}
 We have from \eqref{smalldisk}
\begin{equation}\mbox{Area}(\G_t' \cap \overline{N_{h_0}(U)})\leq 2*1.01^2 \mbox{Area}(D) + 1.01h_0L <  2.05 \frac{l_i^2}{3 \pi} + 1.01h_0L\end{equation} for $t \in [0,1]$. Choosing $h_0$ sufficiently small we obtain
\begin{equation}
\mbox{Area}(\G_t' \cap \overline{N_{h_0}(U)}))\leq  \frac{3l_i^2}{4 \pi}\mbox{ for all } t\in [0,1].
\end{equation}


Concatenating the deformation from $\G_0$ to $\G'$ and then the isotopy with surgery $\{\G'_t\}_{t\in [0,1]}$ we obtain an isotopy with one neck-pinch surgery to a surface $\G'_1$, such that \begin{equation} S\cap \G'_1 = \gamma \setminus \gamma_i.\end{equation} Summing up all the increases in the area we find that they are bounded by $\frac{l_i^2}{ \pi}$.
Some of the surfaces we defined above were only piecewise smooth, having corner-like singularities along embedded closed curves; 
there exists a corner smoothing that deforms them
into a family of smooth surfaces while
increasing the area by an arbitrarily small amount (cf. \cite[Section 7.5]{Mukherjee}, \cite[Section 3.2]{chambersliok}). 

By the inductive assumption there exists an isotopy with surgeries from $\G'_1$ to $\G_1$ satisfying the conclusions of the Lemma. Putting together the deformation from $\G_0$ to $\G_1'$ and the deformation from $\G_1'$ to $\G_1$ we obtain the desired isotopy with surgeries $\G_t$. The total area increase during the deformation is bounded by
\begin{equation}\frac{l_i^2}{\pi}+ \frac{(l-l_i)^2}{\pi}\leq \frac{l^2}{\pi}.\end{equation}

Finally, by \cite[2.1(20)]{P} we obtain for all $t\in [0,1]$
\begin{equation*}
    \F(\G_0, \G_t) \leq 2 (\mbox{Area}(\G_0 \setminus \G_t )+\mbox{Area}(\G_t \setminus \G_0 )) \leq \frac{4l^2}{\pi}.
\end{equation*}
\end{proof}

%
%

\subsection{Removing the thin part of $\G$}
Recall from (\ref{nh1}) that $N_h(.)$ denotes the $h$-tubular neighborhood of an embedded surface.
Given a surface that has small area outside of some open set (in our case, outside of the tubular neighborhood of some other surface) we can define an isotopy that reduces the area of that part of the surface even further
(see \cite[Lemma 7.1]{Montezuma1} for an analogous lemma in the context of Almgren--Pitts theory). 

\begin{lemma}[Reducing the area of thin hair]
 \label{contracting_hair}
Let $\Si$ be an embedded surface in $M$
then for all sufficiently small $h>0$ there exists $\varepsilon(\Si, h)>0$ with the following
property. 
For every $\delta>0$ and every embedded surface $\Gamma$ with 
$\mbox{Area}(\Gamma \setminus N_{h/2}(\Sigma))< \varepsilon$ there exists
a smooth isotopy $\Gamma_t$ with
\begin{enumerate}
\item $\Gamma_0 =\Gamma$, 
\item $\mbox{Area}(\Gamma_t) \leq \mbox{Area}(\Gamma)+ \delta$,
\item $\mbox{Area}(\Gamma_1 \setminus N_{3h/4}(\Si))< \delta$.
\end{enumerate}

\end{lemma}

\begin{proof}
We recall the following ``bounded path" version of the $\gamma$-reduction of Meeks-Simon-Yau \cite{MSY} used in the min-max setting of Simon--Smith \cite{Smith} (see \cite[Section 7]{C&DL}). Let $\Gamma$ be an embedded surface in $M$, and $U$ be an open set included in $M$. Let $\mathfrak{Is}(U)$ be the set of isotopies of $M $ fixing $M\backslash U$, with parameter in $[0,1]$. For $\delta>0$ define
$$\mathfrak{Is}_{\delta}(U) = \{\psi\in\mathfrak{Is}(U); \mbox{Area}(\psi(\tau,\Gamma)) \leq \mbox{Area}(\Gamma) + \delta \text{ for all } \tau \in[0,1]\}.$$
Suppose that the sequence $\{\psi^k\}\subset \mathfrak{Is}_{\delta}(U)$ is such that 
$$\lim_{k\to\infty}\mbox{Area}(\psi^k(1,\Gamma)) = \inf_{\psi\in\mathfrak{Is}_{\delta}(U)}\mbox{Area}(\psi(1,\Gamma)).$$
Such a sequence is called minimizing. Then in $U$, $\psi^k(1,\Gamma)$ subsequentially converges in the varifold sense to a smooth minimal surface $\hat{\Gamma}$ (which might not be smooth up to $\partial U$) \cite{MSY}.

Let us apply this $\gamma$-reduction procedure with constraint to $U:=M\backslash \bar{N}_{\frac{h}{2}}(\Si)$. Let $\{\psi^k\}\subset \mathfrak{Is}_{\delta}(U)$ be a minimizing sequence, such that the induced varifolds $|\psi^k(1,\Gamma)|$ converge to a varifold
$V$ supported on a minimal surface $\hat{\Gamma}$ in $U$.

By the monotonicity formula for minimal submanifolds of a Riemannian manifold 
(\cite[(5)]{C&DL}, \cite[(7.5)]{ColMin})
we
have
$$\mbox{Area}(\hat{\Gamma} \cap B_r(p)) \geq Cr^2 $$
for all $r< r_0$, for some constants $C>0$ and $r_0>0$ that depend only on $M$. 
If $\hat{\Gamma}$ contains a point $ p \in U \setminus N_{3h/4}(\Si)$ and $h<4r_0$, then
 \begin{equation}\mbox{Area}(\hat{\Gamma} \cap B_{\frac{h}{4}}(p) \geq C (\frac{h}{4})^2.\end{equation} Hence, if we choose $\varepsilon < C (\frac{h}{4})^2$ then the intersection of 
$\hat{\Gamma}$ with $ U \setminus N_{3h/4}(\Si)$ is empty.

It follows that the area of $(M\backslash {N_{\frac{3h}{4}}}(\Si))\cap \psi^k(1,\Gamma)$ tends to zero as $k \rightarrow \infty$ and for some $k'$ large enough, it will be smaller than
$\delta$. The lemma is proved by taking $\Gamma_t = \psi^{k'}(t,\Gamma)$.
\end{proof}

The following lemma shows that we can remove the thin part
of $\G$ that lies outside of a small 
tubular neighborhood of $\G$
via an isotopy with surgeries, while increasing the area
by an arbitrarily small amount.  We will cover the set $M \setminus N_h(\Sigma)$ by a collection of small balls and use isotopies with neck-pinch and collapse surgeries to inductively remove portions of the surface in each ball. There is a small technical difficulty: during the neck-pinch deformations, we add small disks to the surface that could potentially intersect balls from which the surface was removed at the previous steps of the inductive process. We resolve this issue using the fact that, assuming $\delta'$ is small enough, the disks will lie very close to the surface and will not intersect concentric balls of smaller radius that still cover $M \setminus N_h(\Sigma)$.

\begin{lemma} \label{contracting_hair2}
Let $\Si$ be an embedded surface in $M$. There exists $\overline{h}>0$, such that for all $h \in (0, \overline{h})$ 
there exists $\varepsilon(\Si, h)>0$ with the following
property. 
For every $\delta>0$
and every embedded surface $\Gamma$ with 
$\mbox{Area}(\Gamma \setminus N_{h/2}(\Sigma))< \varepsilon$ there exists an isotopy with surgeries $\{ \G_t \}_{t\in [0,1]}$,
such that
\begin{enumerate}
\item $\Gamma_0 =\Gamma$, 
\item $\mbox{Area}(\Gamma_t) \leq \mbox{Area}(\Gamma)+ \delta\mbox{ for all } t\in [0,1]$,
\item $\Gamma_1 \subset  N_{h}(\Si)$.
\end{enumerate}
\end{lemma}

\begin{proof}
Let $\delta' >0$ be a small constant we will choose later.
Choose $\varepsilon(\Sigma, h)$ to be sufficiently small, so that by Lemma \ref{contracting_hair} there exists an isotopy
$\{\G_t\}_{t\in [0,1]}$ with $\Gamma_0=\Gamma$ to a surface $\G'=\Gamma_1$ with 
\begin{equation}\mbox{Area}(\G' \setminus N_{3h/4}(\Si))< \delta'
\end{equation}
and \begin{equation}\mbox{Area}(\G_t) < \mbox{Area}(\G) + \frac{\delta}{2}\mbox{ for all } t\in[0,1].\end{equation}

Let $\mbox{injrad}(M)$ denote the injectivity radius of $M$.
Choose radius \begin{equation}\label{choiceofr} r< \min\{ \mbox{injrad}(M), \sqrt{\delta}, h/4\},\end{equation} with the additional property that every ball of radius $r$
in $M$ is $1.01$-bi-Lipschitz diffeomorphic to the Euclidean ball of the same radius.
Let $\{ B_{r/4}(x_i )\}_{i=1}^N$ be a collection of balls covering
$M \setminus N_{h}(\Si)$ with $x_i \in M \setminus N_{h}(\Si)$. 

Assume now that $\delta'$ is chosen sufficiently small to satisfy $\delta'< 2^{-(N+4)}r^2$.
We will inductively define an isotopy with surgeries
$\tilde{\G}_t$, so that for 
$j=1,...,N$
the following is satisfied:
\begin{enumerate}\label{itemsneeded}
    \item $\tilde{\G}_0=\G'$;
    \item $\tilde{\G}_t \cap N_{3h/4}(\Sigma) = \G' \cap N_{3h/4}(\Sigma)$;
    \item \label{thirdhere} $\mbox{Area}(\tilde{\G}_t \setminus N_{3h/4}(\Sigma)) \leq  2^{-(N-j+3)}r^2$ for $t \in [\frac{j-1}{N}, \frac{j}{N}]$;
    \item $\tilde{\G}_{\frac{j}{N}} \subset M \setminus \bigcup_{i=1}^j B_{r(\frac{1}{2} - \frac{1}{8}\sum_{i=1}^j 2^{i-N})}(x_i )  \subset M \setminus \bigcup_{i=1}^j B_{r/4}(x_i)$.
\end{enumerate}

Having established such an inductive process, we get from item \ref{thirdhere} and the choice of $r$ \eqref{choiceofr} that for all $t\in [0,1]$ there holds 
\begin{equation}
\mbox{Area}(\tilde{\G}_t \setminus N_{3h/4}(\Sigma)) \leq  2^{-(N-j+3)}r^2\leq 2^{-(N-j+3)}\delta<\frac{\delta}{8}. 
\end{equation}
Concatenating the isotopy $\{\Gamma_t\}_{t\in [0,1]}$ with $\{\tilde{\Gamma}_t\}_{t\in [0,1]}$ gives the desired isotopy purported in the statement of the lemma.

Let us now carry out the inductive construction.  Fix $j \in \{1,...,N\}$. We use the coarea inequality to find $r_j \in [r/2,r]$,
so that $\tilde{\G}_{\frac{j-1}{N}}$ intersects $\partial  B_{r_j}(x_j )$ in a collection of curves $\gamma = \sqcup \gamma_l$ of length bounded
by 
\begin{equation} \label{neckpinch coarea estimate}
 \mbox{Length}(\gamma) \leq \frac{2\mbox{Area}(\tilde{\G}_{\frac{j-1}{N}} \cap B_{r}(x_j))}{r}
\leq \frac{2* 2^{j-4-N}r^2}{r}\leq 2^{-(N-j+3)}r  
\end{equation}
We apply Lemma \ref{intersection neck-pinch} to define an isotopy with neck-pinch surgeries along the intersection $\gamma$. As a result we obtain a surface 
 $\G^j$ that can be written as a disjoint union $\G^j =\G^j_0 \sqcup \G^j_1$, with $\G^j_0$ contained in the closed ball $B_{r_j}(x_j)$
 and $\G^j_1 $ contained in $M \setminus B_{r_j }(x_j)$.  
 
By Lemma \ref{intersection neck-pinch} and \eqref{neckpinch coarea estimate}  in the process of neck-pinch surgery deformations the areas of surfaces increase by at most
 \begin{align} \label{area from length bound}
\frac{1}{\pi} \mbox{Length}(\gamma)^2 \leq  \frac{2^{-2(N-j+3)}r^2}{\pi} 
 \end{align}
Using the fact that $B_{r_j}(x_j)$ is $1.01$-bi-Lipschitz diffeomorphic
 to the Euclidean ball of the same radius, we can radially contract
 $\G^j_0$ to a point inside $B_{r_j}(x_j)$, while the area only increases by at most a factor of $1.01^2$. Thus, we obtain an isotopy with one collapse surgery (recall Definitions \ref{def:collapse} and \ref{def:collapse1}) from $\G^j_0 \sqcup \G^j_1$ to $\G^j_1$ .
 Concatenating these two deformations we obtain an isotopy with surgeries $\{\tilde{\G}_t \}_{t \in [\frac{j-1}{N},\frac{j}{N} ]}$ from $\tilde{\G}_{\frac{j-1}{N}}$ to $\G^j_1 = \tilde{\G}_{\frac{j}{N}}$,
 with areas (by \eqref{area from length bound}  and the inductive assumption) satisfying
 \begin{align*}
     \mbox{Area}(\tilde{\G}_{t} \setminus N_{3h/4}(\Sigma)) & \leq   1.01^2(2^{-(N-j+4)}r^2 + \frac{2^{-2(N-j+3)}r^2}{\pi}) \\
     &  < \frac{2^{-(N-j+4)}(\pi+2)}{\pi}r^2 < 2^{-(N-j+3)}r^2.
 \end{align*}

 In particular, we have
 $$\mbox{Area}(\tilde{\G}_{t}) < \mbox{Area}(\G' \cap N_{3h/4}(\Sigma)) + \delta/2 < \mbox{Area}(\G) + \delta.$$
 
 Recall from the proof of Lemma \ref{intersection neck-pinch} that the surface $\tilde{\G}_{\frac{j}{N}} $ is obtained from (a small perturbation of) $\tilde{\G}_{\frac{j-1}{N}} \setminus B_{r_j}( x_j)$ by attaching (small perturbations of) small disks in $\partial B_{r_j}(x_j)$ bounded 
 by connected components of $\gamma$.
 Since $B_r(p_j)$-ball is $1.01$-bi-Lipschitz diffeomorphic to the Eucledian ball, we have that each such disk will lie in an $\eta_j-$neighborhood 
 $N_{\eta_j}(\tilde{\G}_{\frac{j-1}{N}} \setminus B_{r_j}(x_j))$ for $\eta_j<\mbox{Length}(\gamma)$. We would like to ensure that these disks do not intersect the ball $B_{r/4}(p_l)$ for any $l \leq j$.  Indeed, by \eqref{neckpinch coarea estimate} and the inductive assumption we have
 \begin{align*}
     \tilde{\G}_{\frac{j}{N}}& \subset N_{2^{-(N-j+3)}r}(\tilde{\G}_{\frac{j-1}{N}} \setminus B_{r/2}( x_j))\\
     & \subset N_{2^{-(N-j+3)}r} \Big(M \setminus \big( B_{r/2}( x_j) \cup\bigcup_{i=1}^{j-1} B_{r(\frac{1}{2} - \frac{1}{8}\sum_{i=1}^{j-1} 2^{i-N})}(x_i ) \big)\Big) \\
     & \subset M \setminus \bigcup_{i=1}^j B_{r(\frac{1}{2} - \frac{1}{8}\sum_{i=1}^j 2^{i-N})}(x_i )  \subset M \setminus \bigcup_{i=1}^j B_{r/4}(x_i).
 \end{align*}
Since the collection of balls $\{ B_{r/4}(x_i )\}_{i=1}^N$ covers
$M \setminus N_{h}(\Si)$ we get that $\tilde{\G}_1 \subset N_{h}(\Si)$.
 
 This finishes the construction of the desired isotopy with surgeries.

\end{proof}

\subsection{Squeezing maps.} 
\label{sec:squeezing}
Let $\Si \subset M$ be a smooth, closed, two-sided, embedded surface which is a strictly stable minimal surface. In this paper, we will only consider the situation where either $\Si$ is a connected component of the boundary of $M$, or $\Si$ is contained in the interior of $M$.
It will be convenient to foliate an open neighborhood of $\Si$
not by level sets of the distance function,
but rather by surfaces with mean curvature vector pointing
towards $\Si$, which arise as graphs 
of the first eigenfunction
of the stability operator over $\Si$.  Such a foliation was used in \cite[Subsection 5.7]{MaNeindexbound}. 

More precisely, for such $\Si$, there exists
a neighborhood of $\Si$ denoted $\Omega_1$, and a diffeomorphism 
\begin{equation} \label{def phi}
\begin{split}
\phi:& \Si \times [0,1) \rightarrow \Omega_1 \subset M \quad \text{if $\Si$ is a subset of the boundary of $M$},\\
\phi:& \Si \times (-1,1) \rightarrow \Omega_1 \subset M \quad \text{if $\Si$ is in the interior of $M$},
\end{split}
\end{equation}
such that $\phi(\Si,0) = \Si$ and for each $t\in (-1,1)\setminus \{0\}$, $\phi(\Si,t)$ is an embedded surface with nonzero mean curvature vector pointing towards $\Si$.

Set $\Omega_h := \phi(\Sigma \times [0,h))$ if
$\Si \subset \partial M$ or $\Omega_h := \phi(\Sigma \times (-h,h))$ if
$\Si$ is in the interior of $M$

We now define the following `squeezing maps'': for $t\in [0,1]$, set
\begin{equation} \label{P_t def}
P_t(\phi(x,s)) := \phi(x,(1-t)s).
\end{equation}
Consider also the projection map 
\begin{equation} \label{proj P def}
P: \Omega_1 \rightarrow \Sigma
\end{equation}
$$P(\phi(x,s)) := x.$$
We refer to \cite[Subsection 5.7]{MaNeindexbound} for the details of this
construction.  We summarize the main properties of the squeezing-map $P_t$:

\begin{enumerate}
\item $P_0(x) = x$ for all $x \in  \Omega_1$
and $P_t(x)=x$ for all $x \in \Sigma$ and $0 \leq t< 1$;
\item 
For all positive $h>0$ small enough, 
there exists
$t(h) \in (0,1)$  with  $P_{t(h)}(\Omega_1)\subset N_{h}(\Si)$ where $N_{h}(\Si)$ is the $h$-tubular neighborhood of $\Si$, see (\ref{nh1});
\item 
\label{squeezing: area estimate}
$P_t$ is area-nonincreasing: for any surface $S \subset \Omega_1$ and for all 
$t \in (0,1)$ we have 
$$\mbox{Area}(P_t(S)) \leq \mbox{Area}(S)$$ with equality holding
if and only if $S \subset \Sigma$;

\item 
\label{squeezing: graphical convegence}
Let $U \subset \Sigma$ be an open set,
$f:U \rightarrow \mathbb{R}$ be a smooth function with absolute value
less than $1$ and let $S = \{\phi(x,f(x)): x \in U \}$.
Then we have graphical smooth convergence of $P_t(S)$ to $U$ 
as $t \rightarrow 1$.
\end{enumerate}

Property (\ref{squeezing: area estimate}) is proved in 
\cite[Proposition 5.7]{MaNeindexbound}. All other properties follow from the definition. The importance of the above is that we can use the squeezing map to push an embedded surface $S$ in a small tubular neighborhood of $\Sigma$
towards $\Sigma$ while simultaneously decreasing its area.

\begin{defi}We say that an embedded surface $S$ is \emph{graphical} if it
satisfies
$S = \{\phi(x,f(x)): x \in U \}$ for some function $f$ and a 
subset $U \subset \Sigma$.
\end{defi}

In the following sections, if $\Si$ is a strictly stable minimal surface as above, we will only consider tubular neighborhoods $N_h(\Si)$ with $h$ so small that 
$$N_h(\Si) \subset \Omega_1.$$

\subsection{Stacked surface.} \label{stacked_global}
 Let $\Sigma:=\partial M$ and assume that $\Si$ is a two-sided strictly stable minimal surface, 
 Let $\phi: \Si \times [0,1) \rightarrow \Omega_1 \subset M$ be the diffeomorphism defined in (\ref{def phi}).
 Let $N_h(\Si)$ be the $h$-tubular neighborhood of $\Sigma$ for some $h>0$ small enough so that $N_h(\Si)$ is well-defined (\ref{nh1}) and contained in $\Omega_1$, and let $U \subset \Sigma$ be an open set. 
\begin{defi}
We will say that an embedded surface $S$ contained in $N_h(\Sigma)$
is \emph{$(h_1, \dots, h_m)$-stacked over $U$ inside $N_h(\Sigma)$}
if
there exist constants
$$0 \leq h_1< \dots < h_m \leq 1$$ 
such that
$$S \cap P^{-1}(U) = \bigsqcup_{i=1}^m \phi(U, h_i).$$
We will also use the shorthand \emph{``$m$-stacked''}.
\end{defi}

Observe that by Property (\ref{squeezing: graphical convegence}) of $P_t$ for every $\eta>0$ there exists $h'>0$
so that if $h_m \leq h'$ and $S$
is $(h_1, \dots, h_m)$-stacked over $U \subset \Sigma$ then 
\begin{equation} \label{areasareclose} m \mbox{Area}(U) - m\eta \leq \mbox{Area}(S \cap P^{-1}(U)) \leq m \mbox{Area}(U) + m\eta\end{equation}
and $S \cap P^{-1}(U)$ is $m\eta$-close in the varifold $\mathbf{F}$-distance to $\Sigma \cap U $ with multiplicity $m$.

The main result of this section is the following step towards Theorem \ref{maininterpolation}.  Namely, we will first show that (in the setting and notation of Theorem \ref{maininterpolation}) we can define an isotopy with surgeries
to deform the surface $\Gamma$ into a stacked form while obeying the area constraint:

\begin{prop}[Stacking] \label{stacking2}
Let $(M,g)$ be a Riemannian 3-manifold with boundary $\Si:=\partial M$. Suppose that $\Sigma$ is a strictly stable minimal surface and $\Sigma = \sqcup_{i=1}^\nu \Sigma_i$, where $\Sigma_i$ are the distinct  
boundary components of $ M$.  

There exist $\overline{h}, \varepsilon >0$, depending on $M$ and $\Sigma$, such that for all $h \in (0, \overline{h})$ the following holds. 
Suppose that a closed embedded surface $\Gamma$ satisfies
$$\mbox{Area} (\Gamma \setminus N_{h}(\Sigma))< \varepsilon.$$
Let $\delta>0$.  Then 
there exists an isotopy 
with surgeries $\{\Gamma_t\}_{t\in[0,1]}$
and a collection of integers $\{ k_i\geq 0 \}_{i=1}^{\nu}$,
 such that:
\begin{enumerate}
\item $\Gamma_0=\Gamma$,
\item $\mbox{Area}(\Gamma_t) < \mbox{Area}(\Gamma) + \delta$,
\item $\Gamma_1 \subset N_h(\Sigma)$ and for each $i=1,2,...,\nu$, $\Gamma_1$ is 
$k_i$-stacked over $\Si_i$ in $N_h(\Si_i)$.

\end{enumerate}
\end{prop}


The remainder of this section will be devoted to the
proof of Proposition \ref{stacking2}.  Fix $\Gamma$ as in the statement of Proposition \ref{stacking2}.

Choose $\overline{h}>0$ to be sufficiently small, so that $N_{\overline{h}}(\Sigma) \subset \Om_1$.
Moreover, let $\overline{h}>0$,
$\varepsilon>0$ be chosen smaller than the corresponding constants in 
 Lemma \ref{contracting_hair2}. Fix $h \in (0, \overline{h})$.
Applying Lemma \ref{contracting_hair2} we can assume, without any loss of
generality, that $\G \subset N_h(\Sigma)$.
We make this assumption for the remainder of this section.


\subsection{Choice of constants $\hat{r}>0$, $\hat{h}>0$, and open neighborhood $\Omega_{\hat{h}}$} \label{r}
Here we make choices for two numbers $\hat{h},\hat{r} \in (0,1)$ which will be useful in the proof of Proposition \ref{stacking2}.

Fix $\delta>0$ as in the statement of Proposition \ref{stacking2}. There exists $r_1>0$ so that 
for any $r\leq r_1$ we have
\begin{equation} \label{eq:areaofballs}
||P_{\sharp}(|\G|) \llcorner B_r(x)||(M) \leq 
\frac{\delta}{200}.\end{equation}
(Varifold notation was introduced after the statement of Theorem \ref{maininterpolation}; intuitively, $||P_{\sharp}(|\G|) \llcorner B_r(x)||(M)$ is the area of the projection
of $\G$ onto $\Sigma$, counted with multiplicity.)
By continuity, we can choose $t_1$ close enough to $1$ so that 
for every $t \in [t_1,1)$ the mass of $P_t(\Gamma) \subset \Omega_{1 - t_1}$ in any ball $B_{r}(x)$ is at most $\delta/100$ for any $r<r_1$ and $x\in\Sigma$.   We replace $\Gamma$ with $P_{t_1}(\Gamma)$ (but do not relabel it).

Now we pick $\hat{r}= \hat{r}(\Si, \G, \delta)>0$, satisfying the following properties:

\begin{enumerate}[label=\roman*.]
\item  $\hat{r}$ is smaller than the minimum of the convexity 
radii of $M$ and $\Si$; \label{i}
\item $\hat{r} < r_1$, that is, for every $x \in \Si$ 
we have \begin{equation} \label{ii}\mbox{Area}(\G \cap B_{\hat{r}}(x)) < \frac{\delta}{100}.\end{equation}
\item For every $x \in \Si$ 
we have that the exponential map $\exp: B^{Eucl}_{\hat{r}}(0) \rightarrow B_{\hat{r}}(x)$
satisfies $\frac{1}{1.01} < |d \exp_y| < 1.01$ for all $y \in B^{Eucl}_{\hat{r}}(0)$. \label{iii}
\end{enumerate}
It will also be convenient to have control over the areas of stacked surfaces in balls $B_{\hat{r}}(x)$. For each $i$ let $\tilde{m}_i$ denote the largest integer
with $\tilde{m_i} \mbox{Area}(\Si_i) \leq \mbox{Area}(\G) + \delta$.
We will assume that $\hat{r}$ is sufficiently small, so that 
$$\tilde{m_i}\mbox{Area}(\Si \cap B_{\hat{r}}(x)) \leq \frac{\delta}{200}$$ for all $x \in \Si_i$ and for all $i =1,..., \nu$. Consequently, we have the following:
\begin{enumerate}
    \item[iv.] Suppose $S$ is a stacked surface in the neighborhood of $\Si_i$ with $\mbox{Area}(S) \leq \mbox{Area}(\G) + \delta$, then for some $t>0$  
we have \begin{equation} \label{iv}\mbox{Area}(P_t(S) \cap B_{\hat{r}}(x)) < \frac{\delta}{100}
    \end{equation} 
    for every $x \in \Si_i$.
\end{enumerate}

We pick $\hat{h} \in (0, 1 - t_1)$ sufficiently small, so that $\Om_{\hat{h}} \subset N_h(\Sigma)$ 
and $\Omega_{\hat{h}} \subset N_{\hat{r}/10}(\Sigma)$. By applying the squeezing map $P_t$
to $\G$ we can assume without loss of generality
that $\G \subset \Omega_{\hat{h}}$.

\subsection{Choice of triangulation and constant $c$.} \label{triangulation}

Let $\Si_i$  be a connected component of $\Si$. Fix a triangulation 
of $\Si_i$ by 2-simplices 
$$\{S_j ^i\}_{1 \leq j \leq N_i},$$
so that for each 2-simplex $S_j ^i$ ($1 \leq j \leq N_i$) in the triangulation,
there exists a point 
\begin{equation}\label{pij def}
p_j^i \in S_j^i \text{ such that }S_j^i \subset B_{\hat{r}/3}(p_j^i).
\end{equation}
Assume the numbering is chosen in such a way that 
for every $j\geq 1$, the solid triangle
$S_{j+1}^i$ shares an edge with $S_{j'}^i$ for some $j' \leq j$.
We cover $\Omega_{\hat{h}}$ by the collection of
3-dimensional cells 
$$\{\Delta_j^i = P^{-1}(S^i_j) \cap \Om_{\hat{h}} \}$$
where $P$ is the projection map defined in \ref{proj P def}.

Note that the interiors of cells $\Delta_j^i$ are disjoint and, by the definition of $\hat{h}, \hat{r}$
in Subsection \ref{r},  $\Delta_j^i$ is contained in a ball  $B_{\hat{r}/2}(p_j^i)$.
We define 
\begin{equation} \label{def of c}
   c := \min_{i,j}\{ \mbox{Area}(S_j^i)\}>0. 
\end{equation}
Note that our choice of $c$ depends on the choice of 
the
triangulation $\{S_j ^i\}$ and hence depends
on the surfaces $\G$, $\Si$ and the choice of $\delta>0$ via \eqref{eq:areaofballs},
\eqref{ii} and \eqref{iv}.




\subsection{Essential multiplicity.} \label{essmulti}
The goal of this section is to perform isotopies and surgeries that decrease the area of $\G$ and that are supported in small balls. Hence, it will be convenient to define a notion of ``essential multiplicity,"
which corresponds to the smallest number of sheets that can be obtained by such local deformations. Note that if the surface $\G$ consists of two parallel sheets joined by a knotted neck as in Fig. \ref{fig:essential}, 
then it is not possible to decrease the number of sheets by a local deformation, as this would require unknotting and then ``opening up" the neck. Such an isotopy would necessarily leave the small ball. 
In the following definition, we consider all isotopic deformations of $\G$ in a small cylindrical region containing a cell $\Delta_j^i$ and define essential multiplicity as the smallest number of curves intersecting the boundary of the region, excluding curves that can be removed by a neck-pinch surgery.
More precisely, we have the following.

Given a point $p \in \Si$  and $\rho>0$,
consider the 3-dimensional region 
\begin{equation} \label{def D_r}
D_{\rho}(p) = P^{-1}(B_{\rho}(p) \cap \Si) \subset \Omega_1
\end{equation}
and consider the 2-dimensional ``cylindrical'' surface
\begin{equation} \label{def C_r}
C_{\rho}(p) = P^{-1}(\partial B_{\rho}(p) \cap \Si) \subset \Omega_1.
\end{equation}

Suppose $\tilde{\G}$ is a closed embedded surface contained in $\Omega_{\hat{h}}$. We use the notations of Subsections \ref{r}, \ref{triangulation}. 
Let $r' \in (\hat{r}/2, \hat{r})$ and fix a point $p_j^i$ (recall $p_j^i$ was defined in \eqref{pij def}). 
Let $\mathcal{S}(\tilde{\G},p_j^i, r')$ denote the set of surfaces 
$S \subset \Omega_{h}(\Si)$, 
such that $S$ intersects $C_{\hat{r}/2}(p_j^i)$ transversally
and such that
there exists an isotopy from $\tilde{\G}$ to $S$
through surfaces $S_t'$
such that $S_t' \setminus D_{r'}(p_j^i) = \tilde{\G} \setminus D_{r'}(p_j^i)$.
Let $k_{i,j}(S)$ denote the number of connected components of
$S \cap C_{\hat{r}/2}(p_j^i)$, which do not bound a 2-dimensional disk
inside $ C_{\hat{r}/2}(p_j^i)$.

\begin{defi} \label{k_ess defi}
Let $r' \in (\hat{r}/2, \hat{r})$ and fix a point $p_j^i$.
We define the \emph{essential multiplicity} of $\tilde{\G}$ in
$B_{r'}(p_j^i)$ to be 
\begin{equation}
k_{ess}(\tilde{\G},i,j,r') := \inf \{k_{i,j}(S)| S \in \mathcal{S}(\tilde{\G},p_j^i,r')  \}.\end{equation}
\end{defi}

\begin{remark} \label{une remark}
It follows from the definition that
\begin{equation}k_{ess}(\tilde{\G},i,j,r_1) \leq k_{ess}(\tilde{\G},i,j,r_2)\end{equation}
for $r_1\geq r_2$.
Note also that in this definition we may replace $\hat{r}/2$ by any number less than $r'$ and we obtain the same value for $k_{ess}(\G,i,j,r')$.  Moreover, for any $p\in \Sigma$ and $\rho>0$ such that 
$D_{\rho}(p) \subset D_{\hat{r}/2}(p_j^i)$, if $\tilde{\G} \cap C_{\rho}(p)$ is transversal and has fewer than $k$
connected components which do not bound a disk in $C_{\rho}(p)$, 
then we can radially isotope $\tilde{\G}$ from $p$ inside $D_{r'}(p_j^i)$, while leaving it unchanged outside of $D_{r'}(p_j^i)$, to a
surface $S \in \mathcal{S}(\tilde{\G},p_j^i,r')$ with $k_{i,j}(S)<k$. In that case,  
$$k_{ess}(\tilde{\G},i,j,r') <k.$$
\end{remark}


From the coarea formula we obtain the following
lower bound for the area of the surface in a cell in terms of the essential multiplicity:


\begin{lemma} \label{essential}
Fix $\iota>0$. For $t$ small enough depending on $\Sigma$, $\G$ and $\iota$, the squeezed surface $\Gamma':=P_{1-t}(\G)$ satisfies for all $i,j$ and $r' \in (\hat{r}/2, \hat{r})$:
\begin{equation}\label{essentiallower}
\mbox{Area}(\Gamma' \cap D_{r'}(p_j^i)) \geq k_{ess}(\Gamma',i,j,r') \mbox{Area}(\Si \cap D_{r'}(p_j^i)) - \iota.
\end{equation}
\end{lemma}

\begin{proof}
For every $s>0$, picking $t>0$ sufficiently small we have
$\Gamma' \subset \phi(\Sigma \times [-s,s])$.
By the coarea formula we have
\begin{equation}
\mbox{Area}(\G' \cap D_{r'}(p_j^i)) \geq c_0(s) \int_{\rho=0}^{r'} \mbox{Length}(\G' \cap C_{\rho}(p_j^i)) d \rho,
\end{equation}
where $c_0(s) \rightarrow 1$ as $s \rightarrow 0$.

Let $R \subset (0,r')$ denote the set of values of $\rho$ 
for which the intersection $\gamma_\rho=\G' \cap C_{\rho}(p_j^i)$ is transverse.
Suppose $\gamma_{\rho'}$ has $k<k_{ess}(\G,i,j,r')$
connected components that are non-contractible in $ C_{\rho'}(p_j^i)$
for some $\rho' \in R$. Then there exists $\delta>0$, such that 
for all $\rho \in (\rho'-\delta, \rho'+\delta)$
the intersection $\G' \cap C_{\rho}(p_j^i)$ has $k$
non-contractible components. Consider an ambient isotopy
that radially expands the region 
$D_{\rho'+\delta}(p_j^i) \setminus D_{\rho'-\delta}(p_j^i)$
to $D_{r'-\delta} \setminus D_{\delta}$ and coincides with the identity
outside of $D_{r'}(p_j^i)$. Applying
this isotopy to $\G'$ we obtain an embedded surface $\G''$
with $\G'' \cap  C_{\hat{r}/2}(p_j^i) $ having $k$ non-contractible
components, contradicting the definition of $k_{ess}(\G,i,j,r')$.
It follows that $\gamma_{\rho'}$ has at least $k_{ess}(\G,i,j,r')$
non-contractible components for all $\rho' \in R$. We also have that
every non-contractible curve $\gamma$ in the cylinder $C_{\rho}$
satisfies \begin{equation}\mbox{Length}(\gamma) \geq c_1(s)\mbox{Length}(\Sigma \cap C_\rho)\end{equation}
with $c_1(s)\rightarrow 1$ as $s \rightarrow 0$.

Since $[0,r'] \setminus R$ has measure zero and applying the coarea
formula again we obtain
 \begin{align}\label{areahere}
 \mbox{Area}(\G \cap B_{r'}(p_j^i)) & \geq c_2(s) k_{ess}(\G,i,j,r')
 \int_{\rho=0}^{r'} \mbox{Length}(\Sigma \cap C_{\rho}(p_j^i)) d \rho \\
& = c_2(s) k_{ess}(\G,i,j,r') \mbox{Area}(\Si \cap D_{r'}(p_j^i)),
\end{align}
with $c_2(s)=c_0(s)c_1(s) \rightarrow 1$ as $s \rightarrow 0$. Thus by choosing $t$ sufficiently small, \eqref{areahere} gives the desired conclusion.
\end{proof}


\begin{figure} 
   \centering	
	\includegraphics[scale=0.75]{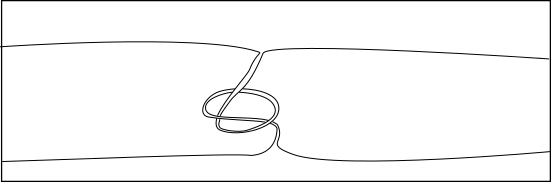}
	\caption{Two graphical sheets joined by a knotted neck.
	There is a homotopy, but no isotopy
	pushing the surface into the boundary of the cell, so $k_{ess} = 2$.}
	\label{fig:essential}
\end{figure}








\subsection{Local operations in the proof of Proposition \ref{stacking2}}

%

The following is a version of the key blow-down blow-up lemma 
in Step 2 in the proof of Lemma 7.6 in Colding--De Lellis \cite{C&DL}.  

A set $U$ is called star-shaped with respect to $x \in U$
if minimizing geodesics between points of $U$ and $x$ lie in $U$.

\begin{lemma}[Blow-down blow-up lemma]\label{radial}
Suppose $x \in U \subset B_r(x)$, $r \in (\hat{r}/2, \hat{r})$, where $U$
is an open set star-shaped with respect to $x$,
and suppose ${\G_t}$ 
is an isotopy with surgeries of surfaces and $\{ \gamma_t = \G_t \cap \partial U\}$
is an isotopy of a collection of closed curves.
Then for every $\eta>0$ there exists an isotopy with surgeries
${\overline{\G}_t}$, such that:
\begin{enumerate}
\item $\overline{\G}_0=\G_0$ and $\overline{\G}_1=\G_1$;
\item $\overline{\G}_t \setminus U =\G_t \setminus U$;
\item $\mbox{Area}(\overline{\G}_t) \leq \mbox{Area}(\G_t \setminus U) + 
\max\{\mbox{Area}(\G_0 \cap U), \mbox{Area}(\G_1 \cap U) \}+
r \mbox{Length}(\G_t \cap 
\partial U)+ \eta$ for $t \in [0,1]$. 
\end{enumerate}
\end{lemma}

\begin{proof}
Let $E: B_r(x) \rightarrow T_xM$ denote the inverse of the exponential map. 
Let $ \tau \in (0,1)$
and $\rho \in (0,1)$ be small constants to be specified later.
Let $\gamma_t = \partial (\G_t \cap U)$.

Given a subset $S \subset \R^3$ let 
$$Cone(S) = \{ tx: x \in S, t \in [0,1]\}$$
denote the cone over $S$. Define
$$C_{t}(S) = Cone(S) \setminus Cone(t S) .$$
We will call $C_t(S)$ a ``conical collar" over $S$.
Let $\phi:[0,\tau] \to [\rho,1]$  be a smooth monotone function with 
$\phi(t)=1$ for $t \in [0, \frac{\tau}{3}]$ and $\phi(\tau)=\rho $ for $t \in [\frac{2\tau}{3},\tau]$.
Let
$$\tilde{\G}_t = (\G_t \setminus U) \cup E^{-1} (\phi(t) E(\G_t \cap U) \cup C_{\phi(t)}(E(\gamma_t))) \quad \text{ for $t \in [0,\tau]$},$$
$$\tilde{\G}_t = (\G_t \setminus U) \cup E^{-1} (\rho E(\G_t \cap U) \cup C_{\rho}(E(\gamma_t))) \quad \text{for $t \in [\tau,1-\tau]$},$$
$$\tilde{\G}_t = (\G_t \setminus U) \cup E^{-1} (\phi(1-t) E(\G_t \cap U) \cup C_{\phi(1-t)}(E(\gamma_t))) \quad \text { for $t \in [1-\tau,1]$}.$$
We smooth out the corners of $\{ \tilde{\G}_t\}$ to obtain a  family
$\{ \overline{\G}_t \}$ which is an isotopy with surgeries.

By our choice of radius $\hat{r}$ (see (\ref{iii}) in Section \ref{stacked_global}), the map 
$E$ is a bi-Lipschitz diffeomorphism
with $\frac{1}{1.01} \leq |dE| \leq 1.01$.
In particular, maps $E$ and $E^{-1}$
have Lipschitz constants bounded by $1.01$. The area of the cone over a curve of length $l$ that lies inside a ball of radius $r$ in $\mathbb{R}^3$ is bounded by $\frac{lr}{2}$.
Hence, we have 
$$\mbox{Area}(C_{\phi(t)}(E(\gamma_t))) \leq \frac{1.01r}{2} (1- \phi(t))\mbox{Length}(\gamma_t)<r \mbox{Length}(\gamma_t)$$

Assume we have chosen $\tau \in (0,1)$ sufficiently small so that for all $t\in [0,\tau]\cup[1-\tau, 1]$ we have
$$\mbox{Area}(\G_t \cap U) \leq \max_{i =0,1}\{\mbox{Area}(\G_i \cap U)\}
+ \eta/2$$

Observe that for any surface $S \subset B_r(p)$ and $t \in (0,1)$
we have 
\begin{equation}\label{-11ineq} \mbox{Area} \big(E^{-1}(tE(S))\big) < \mbox{Area}(S).
\end{equation}
Indeed, consider a curve $c(s) \subset \partial B_{r'}(x)$ and define a 1-parameter family of curves $c(t,s) = E^{-1}((1-t)E(c(s)))$. Using the Gauss Lemma we have
\begin{align*}
    \frac{d}{dt}g \big(\frac{\partial c}{\partial s},\frac{\partial c}{ \partial s}\big) &= 2g \big(\nabla_{\frac{\partial c}{\partial t}}\frac{\partial c}{\partial s}, \frac{\partial c}{\partial s} \big) \\
    &= -2g \big(\nabla_{\frac{\partial c}{\partial s}}\frac{\partial c}{\partial s}, \frac{\partial c}{\partial t} \big) <0.
\end{align*}
Using the coarea formula we obtain
\begin{align*}
    \mbox{Area}(E^{-1}(tE(S))) & = \int_0^r \mbox{Length} (E^{-1}(tE(S)) \cap \partial B_{r'(x)})dr' \\
    & < \int_0^r \mbox{Length} (S \cap \partial B_{r'(x)})dr'= \mbox{Area}(S)
\end{align*}
and this confirms (\ref{-11ineq}).
Besides, we can choose $\rho >0 $ sufficiently small so that $$ \mbox{Area}\big(E^{-1}(\rho E(\G_t \cap U))\big)< \eta/2$$
for all $t \in [\tau, 1- \tau]$.

Using the above estimates we obtain
\begin{align*}
    \mbox{Area}(\tilde{\G}_t) & < \mbox{Area}(\G\setminus U)+
    \max_{i=0,1}\{\mbox{Area}(E^{-1}(E(\G_i \cap U))) \}+r\mbox{Length}(\gamma_t) + \frac{\eta}{2}
\end{align*}
for $[0,\tau]\cup[1-\tau, 1]$, and
\begin{align*}
    \mbox{Area}(\tilde{\G}_t) & < \mbox{Area}(\G\setminus U)+r\mbox{Length}(\gamma_t)+ \frac{\eta}{2}
\end{align*}
for $t \in [\tau, 1- \tau]$. Slightly increasing the bound for the perturbed family
$\overline{\G}_t$ we obtain the desired result.
\end{proof}

The following Lemma allows us to deform 
an embedded surface in one cell $\Delta^i_j$ to make
it $k$-stacked, while controlling the area of surfaces in the deformation. Later, in Lemma \ref{stacking}, we will define a procedure for extending the stacking deformation from one cell to another.
It will be convenient to separate the conclusions of the next Lemma into two parts, I and II.
Part I mainly describes the effect of the deformation inside the $3$-dimensional cell $\Delta^i_j$ and will be used in the proof of Corollary \ref{coro:area decrease}. Part II describes how the deformation affects the surface outside of $\Delta^i_j$, which will be needed for the proof of Lemma \ref{stacking}.

Let $\delta>0$ be the fixed positive constant in Proposition \ref{stacking2} and let the other parameters depending on $\delta$ (namely $\hat{h}$, $\hat{r}$ and the 3-dimensional cells $\{ \Delta_j^i \}_{j=1}^{N_i}$) be chosen as in Subsections \ref{triangulation} and \ref{r}.

\begin{lemma}[Stacking in one cell] \label{decreasing k_ess}

Fix $i\in\{1,...,\nu\}$. Let $\tilde{\G}$ be an embedded surface contained in $\Om_{\hat{h}}(\Si_i)$.
     Fix a cell $\Delta_j^i$. Suppose that \begin{equation}\mbox{Area}(\tilde{\G} \cap B_{\hat{r}}(p_j^i)) \leq \frac{\delta}{50}.\end{equation}  I. For every $\eta>0$ there exists an isotopy 
    with surgeries $\{\tilde{\G}_t\}_{t\in [0,1]}$ so that:
\begin{enumerate}[label=(\alph*)]
\item $\tilde{\G}_0 = \tilde{\G}$;
\item $\mbox{Area}(\tilde{\G}_t) \leq \mbox{Area}(\tilde{\G}) + \delta/2$ for $t \in [0,1]$;
\item $\mbox{Area}(\tilde{\Gamma}_1 \setminus \Delta_j^i) \leq \mbox{Area}(\tilde{\Gamma} \setminus \Delta_j^i) + \eta$;
\item $\mbox{Area}(\tilde{\G}_1 \cap B_{\hat{r}}(p_{j'}^i)) \leq \mbox{Area}(\tilde{\G} \cap B_{\hat{r}}(p_{j'}^i)) + \eta$ for all $j' =1,..,N_i$;
\item $\tilde{\Gamma}_1$ is $k$-stacked in $\Delta_j^i$ 
for $k \leq k_{ess}(\tilde{\G},i,j,\frac{3}{4} \hat{r})$ and moreover $$\mbox{Area}(\tilde{\G}_1 \cap \Delta_j^i)\leq k\mbox{Area}(\Si_i \cap \Delta_j^i) +\eta.$$
\end{enumerate}
II. 
Fix any collection of strictly increasing numbers $h_1<...<h_k$ in $[0,\hat{h})$
and arbitrarily small $\xi'>0$ and $\eta'>0$.
Then $\{\tilde{\G}_t\}_{t\in [0,1]}$ can be assumed to satisfy the following in addition to (a)-(d).
There exists $r' \in (3\hat{r}/4, 99\hat{r}/100)$
and $\xi \in (0,\xi')$
with the property that $\partial B_{r'}(p^i_j)$ does not intersect the $0$-skeleton
of the triangulation $\{S^i_j\}$ of $\Si$, and there exists $\tau \in (0,1)$, such that 
\begin{enumerate}
    \item[(e)] $\tilde{\G}_1$ is  $((1-\tau)h_1,...,(1-\tau)h_k)$-stacked in $D_{r'-\xi}(p^i_j)$ for $k = k_{ess}(\tilde{\G},i,j,r')$;
    \item[(f)] $\mbox{Area}(\tilde{\G}_1 \cap (D_{r'}(p^i_j) \setminus D_{r'-\xi}(p^i_j))) < \eta'$;
    \item[(g)] $\tilde{\G}_1 \setminus D_{r'}(p^i_j)
    = P_t(\tilde{\G} \setminus D_{r'}(p^i_j))$.
\end{enumerate}
\end{lemma}
\begin{proof}
Without relabeling $\tilde{\G}$ we deform it by applying the squeezing map $P_t$ (see Section \ref{sec:squeezing}) for $t$ sufficiently close to $1$, so that
$D_{99\hat{r}/100}(p^i_j)\cap \tilde{\G} \subset B_{\hat{r}}(p^i_j)$.

By the co-area inequality
and the assumption $\mbox{Area}(\tilde{\G} \cap B_{\hat{r}}(p^i_j))
\leq \frac{\delta}{50}$
there exists
a radius $r' \in (3\hat{r}/4,99\hat{r}/100)$ with 
\begin{equation} \label{eqn:coarea1}
 \mbox{Length}( \tilde{\G} \cap C_{r'}(p^i_j)) \leq \frac{\delta}{10 \hat{r}}.
\end{equation}
We note here that $r'$ can be chosen so that
 $\partial B_{r'}(p^i_j)$ does not intersect the $0$-skeleton
of the triangulation $\{S^i_j\}$ of $\Si$, since
the set of values of $r'$ for which this is false has measure zero. (We will use this property in the proof of Lemma \ref{stacking}.)

We will construct an isotopy with surgeries $\{\tilde{\G}_t\}$
satisfying properties (a), (c) and (d) and such that
\begin{equation} \label{Ptau}
    \tilde{\G}_t \setminus D_{r'}(p^i_j) = P_{\tau(t)} (\tilde{\G} \setminus D_{r'}(p^i_j))
\end{equation} 
for some function $\tau(t)$. However, the areas of $\tilde{\G}_t$
for $t \in (0,1)$ will not necessarily be small.
To make sure that property (b) is satisfied we will apply the Blow-down blow-up lemma \ref{radial} as follows. 
Note that as $t \rightarrow 1$
the 1-dimensional varifolds $|P_t(\tilde{{\G}} \cap C_{r'}(p^i_j)))|$
converge to the projected varifold $P_{\sharp}(|\tilde{\G} \cap C_{r'}(p^i_j)))|)$. Thus we can apply the squeezing map $P_t$ to $\tilde{\G}$ (without relabeling it), so that (using \eqref{eqn:coarea1})
$$\mbox{Length}(P_s(\tilde{\G} \cap C_{r'}(p^i_j))) < \frac{\delta}{5\hat{r}}$$
for all $s \in (0,1)$. In particular, \eqref{Ptau} implies
\begin{equation} \label{eqn:lengthbound}
    \mbox{Length}( \tilde{\G}_t \cap C_{r'}(p^i_j)) =\mbox{Length}(P_{\tau(t)}(\tilde{\G} \cap C_{r'}(p^i_j)))  \leq \frac{\delta}{5 \hat{r}}.
\end{equation}

Since $D_{r'}(p^i_j)$ is star-shaped with respect to $p^i_j$ we can apply Lemma \ref{radial} and (\ref{eqn:lengthbound}) to find a new isotopy $\{ \overline{\Gamma}_t \}$ beginning at $\tilde{\G}$ and ending at $\tilde{\G}_1$ 
and satisfying for $0\leq t \leq 1$
\begin{equation}\label{better}
\mbox{Area}(\overline{\G}_t \cap D_{r'}(p^i_j)) \leq  \max\{\mbox{Area}(\tilde{\G}_0\cap D_{r'}(p^i_j)), \mbox{Area}(\tilde{\G}_1\cap D_{r'}(p^i_j))\} + \delta/5.
\end{equation} 
Since $\{\tilde{\Gamma}_t\}_{t\in [0,1]}$ satisfies property (d), we thus obtain for all $t\in [0,1]$
\begin{equation}\label{best}
\mbox{Area}(\overline{\G}_t \cap D_{r'}(p^i_j)) \leq \mbox{Area}(\tilde{\G}_0\cap D_{r'}(p^i_j))+\eta + \delta/5.
\end{equation} 
Without any loss of generality we can assume $\eta<\delta/10$ and item (b) follows.  
Hence, in our construction of $\{\tilde{\Gamma}_t\}_{t\in [0,1]}$ we do not need to control the area of $\tilde{\G}_t$
for $t \in (0,1)$. To prove part I of the Lemma we only need to make sure
that $\tilde{\G}_1$ satisfies estimates (c) and (d) and that (a) holds.

Let $k = k_{ess}(\tilde{\G}, p^i_j, r')$.
By Remark \ref{une remark} we have $k \leq k_{ess}(\tilde{\G}, p^i_j, 3\hat{r}/4)$.
By the definition of essential multiplicity, we can isotopically deform $\tilde{\G}$ inside
$D_{r'}(p^i_j)$
to an embedded surface $\tilde{\G}'$, so that 
$\gamma = \tilde{\G}' \cap C_{\hat{r}/2}(p^i_j)$ consists of exactly $k$ closed curves that do not bound a disk in
$C_{\hat{r}/2}(p^i_j)$.
For each component of $\gamma$ that is instead contractible in $C_{\hat{r}/2}(p^i_j)$, 
starting with the innermost one, we perform
a neck-pinch surgery along the disk filling it
in $C_{\hat{r}/2}(p^i_j)$.
In the end we obtain an embedded surface $\tilde{\G}''$,
whose intersection with $C_{\hat{r}/2}(p^i_j)$
is given by $k$ closed curves $\tilde{\gamma}_1,...,\tilde{\gamma}_k$  that do not bound a disk in  $C_{\hat{r}/2}(p^i_j)$.

We next isotopically deform $\tilde{\G}''$ (without relabeling it)
so that each curve $\tilde{\gamma}_l$
is at a constant height, namely
\begin{equation}\label{heightsdiff}\tilde{\gamma}_l = C_{\hat{r}/2}(p^i_j) \cap \phi(\Sigma, h_l')\end{equation} for some $h_l'>0$. Recall the map $\phi$ in \eqref{heightsdiff} was defined in (\ref{def phi}.
This deformation is constructed inductively as follows. Define the ``height function" \begin{equation}H: \Omega_{\hat{h}} \rightarrow [0,\hat{h}]\end{equation} by $$H(\phi(x,h')) = h'.$$
Let 
$$\tilde{h}_k = \max \{H(q): q \in \bigcup_l \tilde{\gamma}_l\}.$$ 
After relabelling if necessary, assume that
$\phi(x, \tilde{h}_k) \in \tilde{\gamma}_k$ for some $x$,
so $\tilde{\gamma}_k$ is the ``top-most" curve. Then for a small
$\eta'>0$ the subset
$A$ of $C_{\hat{r}/2}(p^i_j)$ between $\tilde{\gamma}_k$ and
$\phi(\Sigma, \tilde{h}_k+ \eta') \cap C_{\hat{r}/2}(p^i_j)$ is diffeomorphic to an annulus
and we can isotope $\tilde{\gamma}_k$ to $\phi(\Sigma, \tilde{h}_k+ \eta') \cap C_{\hat{r}/2}(p^i_j)$ in $A$. Since $\tilde{\G}''$ intersects 
$C_{\hat{r}/2}(p^i_j)$ transversally this isotopy can be extended to an isotopy of 
$\tilde{\G}''$ in the interior of $D_{r'}(p^i_j)$. We then similarly deform the next intersection curve 
$\tilde{\gamma}_{k-1}$. 
In the end we obtain that $\tilde{\G}''$ intersects $C_{\hat{r}/2}(p^i_j)$
in $k$ closed curves at a constant height each as desired.

Given any increasing sequence $h_1 < ... < h_k$ in $[0,\hat{h})$ we can isotopically adjust the heights, so that \begin{equation}\label{decomp}\G'' \cap C_{\hat{r}/2}(p^i_j) = \bigsqcup_{l=1}^k \phi (\Si \cap C_{\hat{r}/2}, (1-\tau)h_l)\end{equation}
for some $\tau \in (0,1)$ with $(1-\tau)h_k< \tilde{h}_k + \eta'$. The decomposition \eqref{decomp} will guarantee that property (e) from Part II holds.

Next we deform $\tilde{\G}''$
so that it is $k$-stacked inside $D_{\hat{r}/2}(p^i_j)$. 
After a small perturbation in the interior of $D_{\hat{r}/2}(p^i_j)$ we can assume
that $\tilde{\G}''$ intersects
$\phi(\Sigma, (1-\tau) h_l) \cap D_{\hat{r}/2}(p^i_j)$
transversally for $l=1,..,k$.
Starting with the innermost connected component $\alpha$ of
$\tilde{\G}'' \cap \phi(\Sigma, (1-\tau)h_k) \cap D_{\hat{r}/2}(p^i_j)$ we perform neck-pinch surgeries on the surface $\tilde{\G}''$ along the disk that the curve $\alpha$
 bounds in $\phi(\Sigma, (1-\tau)h_k) \cap D_{\hat{r}/2}(p^i_j)$ (recall that we do not need to worry about the area increase, as explained in the beginning of the proof).
Note that there exists at least one such curve $\alpha$ (namely, $\tilde{\gamma}_k = \phi(\Sigma, (1-\tau) h_k) \cap C_{\hat{r}/2}(p^i_j)$).
After all such surgeries are completed we obtain a disk $\phi( \Si, (1-\tau) h_k) \cap D_{\hat{r}/2}(p^i_j)$ with boundary $\tilde{\gamma}_k$ as one of the connected components of $\G'' \cap D_{\hat{r}/2}(p^i_j)$. By radially contracting the three-ball $\phi(\Si \times ((1-\tau)h_k, \hat{h})) \cap D_{\hat{r}/2}(p^i_j)$ we remove all closed connected components of $\tilde{\G}''$ that lie above height
$(1-\tau)h_k$ in a collapse surgery (Definitions \ref{def:collapse} and \ref{def:collapse1}).
Applying the same procedure succesively to the other boundary components in order $\tilde{\gamma}_{k-1},...,\tilde{\gamma}_1$ we obtain at the end that $\tilde{\G}'' \cap D_{\hat{r}/2}(p^i_j)$ (not relabelled) is
$k-$stacked inside $D_{\hat{r}/2}(p^i_j)$.
In particular, by the choice of $\hat{r}$, it follows that 
$\tilde{\G}''$ is stacked in 
$\Delta_j^i$.

Next we would like to control the area of $\tilde{\G}''$ outside of $D_{\hat{r}/2}(p^i_j)$ to ensure items (d) and (f).
We do this by squeezing the part of the surface in the annular 
region $D_{r'}(p^i_j) \setminus D_{\hat{r}/2}(p^i_j)$
into a small neighborhood
of $C_{r'}(p^i_j)$ and apply the squeezing map $P_t$,
so that the area in the neighborhood
can be made arbitrarily close to zero. Let us give the details.

Extend polar coordinates $(\rho, \theta ) \in
B_{r'}(p^i_j) \cap \Si$ to 
 cylindrical coordinates 
$\phi(\rho, \theta, h')$, $\rho \in [0,r']$,
$\theta \in S^1$, $h' \in [0,\hat{h}]$ on
$\phi(B_{r'}(p^i_j), [0,\hat{h}])$.
After a small perturbation we can assume
that for all $h' \in [0,\hat{h}]$ and all $ \theta \in S^1$ the line segment
$$l_{h', \theta} = \{\phi(\rho, \theta,h'): \rho \in [\hat{r}/2,r'] \}$$ intersects $\tilde{\G}''$ at finitely many points. Let $$K = \max \# \{ l_{h', \theta} \cap \tilde{\G}'' : h' \in [0,\hat{h}], \theta \in S^1\}.$$
For $s \in [\hat{r}/2, r' ]$, let $f_{s}(\rho)$ denote a continuous piecewise linear
function that scales the interval $[0, \hat{r}/2]$
to $[0, s]$ and the interval $[\hat{r}/2, r' ]$
to $[s, r']$.
Consider the radial map $\Phi_s$ defined in these cylindrical 
coordinates by 
\begin{equation}
\Phi_s(\phi(\rho, \theta, h')) := \phi(f_s(\rho), \theta, h').
\end{equation}
Observe that as $s \rightarrow r'$ we have that
\begin{equation}
\limsup_{s \rightarrow r'}\mbox{Area}(\Phi_s(\phi([\hat{r}/2,r'], S^1, [0,\hat{h}]) \cap \tilde{\G}''))\leq K \mbox{Area}(\phi(r', S^1, [0,\hat{h}])).
\end{equation} 
Applying $\Phi_s$ to $\tilde{\G}''$ for
$s=r'- \xi$ and $\xi \in (0, \xi')$ sufficiently small, and then applying the
squeezing map 
$P_\tau$ for $\tau$ sufficiently close to $1$
we can guarantee that the surface is $k$-stacked in $D_s(p^i_j)$, the part of the surface in  $D_{r'}(p^i_j)\setminus D_s(p^i_j)$
has area less than $\min\{\eta', \eta/5 \}$, and outside of $D_{r'}(p^i_j)$ the surface coincides
with $P_\tau( \tilde{\G})\setminus D_{r'}(p^i_j)$ for some value of $\tau \in (0,1)$. This implies properties (d), (f) and (g), and thus concludes the proof of the lemma.  
\end{proof}


We continue to use the notations of Subsection \ref{essmulti}.
We record the following corollary of Lemma \ref{decreasing k_ess} which will be useful later for the proof of Proposition \ref{movehandles} in the next section:
\begin{coro} [\mbox{Area} decrease from neck opening] \label{coro:area decrease}
Fix $i\in\{1,...,\nu\}$. Let $\tilde{\G}$ be an embedded surface contained in $\Om_h(\Si_i)$.
     Fix a cell $\Delta_j^i$ and $r' \in (\hat{r}/2,\hat{r})$.

      Suppose that $\tilde{\G}$ intersects $C_{\hat{r}/2}(p_j^i)$ transversally, and that $\tilde{\G} \cap C_{\hat{r}/2}(p_j^i)$ has  exactly $\hat{k}>1$ connected components $\gamma_1,...,\gamma_{\hat{k}}$ which do not bound a disk inside $C_{\hat{r}/2}(p_j^i)$. Suppose that for some $u\neq v\in \{1,...,\hat{k}\}$, there is an embedded curve $\alpha_1:[0,1]\to \tilde{\G}\cap D_{\hat{r}/2}(p_j^i)$ with endpoints $\alpha_1(0)\in \gamma_u$, $\alpha_1(1)\in \gamma_v$, and that $\alpha_1$ is isotopic (with fixed endpoints) inside $D_{r'}(p_j^i)$ to an embedded curve $\alpha_2:[0,1]\to  C_{\hat{r}/2}(p_j^i)$.  Then
    $$k_{ess}(\tilde{\G},i,j,r')<\hat{k}.$$

    In particular, under the above assumptions, for any small enough $\eta'>0$, if $\mbox{Area}(\tilde{\G} \cap B_{\hat{r}}(p_j^i)) \leq \frac{\delta}{50}$ and if \begin{equation}\label{aa bd}
    \mbox{Area}(\tilde{\G}\cap \Delta^j_i)> \hat{k} \mbox{Area}(\Si_i\cap \Delta^j_i) - \eta',
    \end{equation}
    then there exists an isotopy 
    with surgeries $\{\tilde{\G}_t\}_{t\in [0,1]}$ so that:
\begin{enumerate}[label=(\alph*)]
\item $\tilde{\G}_0 = \tilde{\G}$;
\item $\mbox{Area}(\tilde{\G}_t) \leq \mbox{Area}(\tilde{\G}) + \delta/2$ for $t \in [0,1]$;
\item $\mbox{Area}(\tilde{\Gamma}_1) \leq \mbox{Area}(\tilde{\Gamma}) - \frac{c}{2}$ where $c$ is defined in (\ref{def of c}).
\end{enumerate}   
\end{coro}
\begin{proof}

\begin{figure} 
   \centering	
	\includegraphics[scale=0.8]{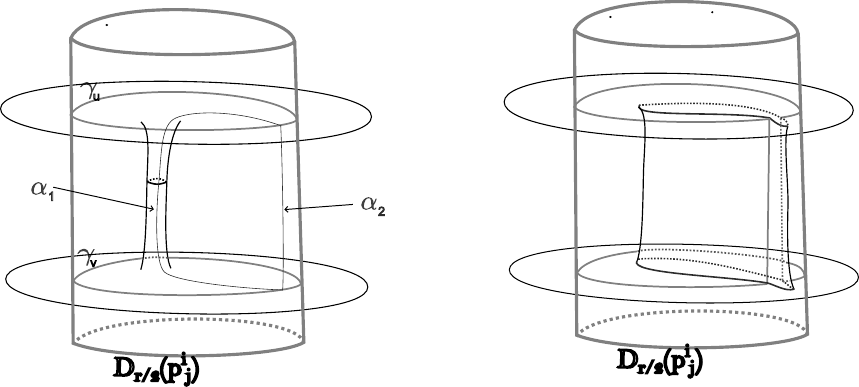}
	\caption{Deformation reducing the number of connected components of $\tilde{\G} \cap C_{\hat{r}/2}(p_j^i)$ non-contractible in $C_{\hat{r}/2}(p_j^i)$.}
	\label{fig:neckopening}
\end{figure}

To give some intuition about the construction in this proof, the model case for this corollary is as follows: we consider a catenoidal neck inside a cylindrical open region 
$D_{\hat{r}/2}(p_j^i)$
which intersects the boundary cylinder 
$C_{\hat{r}/2}(p_j^i)$
along two curves $\gamma_u,\gamma_v$ not bounding disks in 
$C_{\hat{r}/2}(p_j^i)$,
and we consider a connected embedded curve $\alpha_1$ inside this surface connecting $\gamma_u$ to $\gamma_v$ (see Figure \ref{fig:neckopening}).
We deform the neck by pulling the curve $\alpha_1$ outside of the cylinder and moving the surface along, and at the end of the deformation, the new surface intersects 
$C_{\hat{r}/2}(p_j^i)$
along a single closed connected curve bounding a disk in 
$C_{\hat{r}/2}(p_j^i)$.

Recall that we defined the projection $P:\Omega_1\to \Sigma$ in (\ref{proj P def}).
Under our assumptions, the union of $\alpha_1$ and $\alpha_2$ bound an embedded disk $\mathcal{D}$ inside $D_{\hat{r}/2}(p_j^i)$. 
After deforming $\tilde{\G}$ with an isotopy inside $D_{r'}(p_j^i)$ which preserves the topology of $\tilde{\G}\cap D_{\hat{r}/2}(p_j^i)$, we can assume that $\mathcal{D}$ is ``vertical'' in the sense that $P(D)$ is a 1-dimensional segment in $\Si$ and 
$P(\alpha_2([0,1]))$ is a point in $\Si$.

 We slightly perturb $\tilde{\G}$ (without renaming it) so that $\mathcal{D}$ intersects $\tilde{\G}$ transversally up to the boundary. Since $\tilde{\G}$ is embedded, 
 \begin{equation}\label{topol of inter}
 \mathcal{D}\cap \tilde{\G} = \alpha_1([0,1]) \cup \bigcup_{k=1}^{k_0} \beta_k \cup  \bigcup_{m=1}^{m_0} c_m
 \end{equation}
 where $\{\beta_k\}$ are disjoint embedded arcs with endpoints in $\alpha_2([0,1])$, and $\{c_m\}$ are closed embedded curves bounding disks in the interior of $\mathcal{D}$.
 By ``pushing'' $\tilde{\G}\cap \mathcal{D}$ along $\mathcal{D}$ outside of $D_{\hat{r}/2}(p_j^i)$, we can deform $\tilde{\G}$  inside $D_{r'}(p_j^i)$ via an isotopy 
 $$\{\tilde{\G}_t\}_{t\in [0,1]}\text{ with } \tilde{\G}_0=\tilde{\G}$$
 supported in an arbitrarily thin $\mu$-neighborhood $N_\mu(\mathcal{D})$ of $\mathcal{D}$. 
We can ensure that for all time $t\in [0,1]$,
$$\tilde{\G}_t \setminus N_\mu(\mathcal{D})= \tilde{\G} \setminus N_\mu(\mathcal{D})$$
and  at time $t=1$, 
\begin{equation}\label{varn}
P^{-1}\big(N_{\mu/2}(P(\mathcal{D}))\big) \cap P\big(\tilde{\G}_1 \cap N_{\mu}(\mathcal{D}) \cap D_{\hat{r}/2}(p_j^i)\big) =\varnothing.
\end{equation}
Due to (\ref{topol of inter}) and (\ref{varn}), $\tilde{\G}_1 \cap C_{\hat{r}/2}(p_j^i)$
has strictly fewer than $\hat{k}$ connected components which do not bound disks inside $C_{\hat{r}/2}(p_j^i)$. We conclude that $k_{ess}(\tilde{\G},i,j,r')<\hat{k}$.


For any $\eta'\in (0,\frac{c}{100})$, the isotopy with surgeries posited in the second part of the statement is obtained from Lemma \ref{decreasing k_ess}.  Item (c) follows because of item (e) in Lemma  \ref{decreasing k_ess}, together with (\ref{aa bd}) and from the definition of $c$ in (\ref{def of c}) since
\begin{equation}\mbox{Area}(\Si_i \cap \Delta_j^i)>c.\end{equation}
\end{proof}

In the following lemma we prove that we can always either increase the number of cells in which the surface is stacked, or reduce its essential multiplicity in some cell.


Let $\delta>0$ be the fixed positive constant in Proposition \ref{stacking2} and let the other parameters depending on $\delta$ ($\hat{h}$, $\hat{r}$ and the 3-dimensional cells $\{ \Delta_j^i \}_{j=1}^{N_i}$) be chosen as in Subsections \ref{triangulation} and \ref{r}.

\begin{lemma} [Stacking extension]\label{stacking}
Fix $i\in\{1,...,\nu\}$. Let $\tilde{\G}$ be an embedded surface contained in $\Om_{\hat{h}}(\Si_i)$.
Assume that for some $l$, $1 \leq l< N_i$, the surface 
$\tilde{\G} \subset \Om_{\hat{h}}(\Si_i)$ is $(h_1, \dots, h_m)$-stacked 
in $\bigcup_{j \leq l-1} \Delta^i_j$
and satisfies $\mbox{Area}(\tilde{\G} \cap B_{\hat{r}}(p_j^i)) \leq \frac{\delta}{50}$
for all $0 \leq j \leq N_i$.
Fix $\eta \in (0, \delta)$. There exists $\tau \in (0,1)$ and 
an isotopy with surgeries $\{\tilde{\G}_t\}_{t\in [0,1]}$ so that:
\begin{enumerate}[label=(\alph*)]
\item $\tilde{\G}_0 = \tilde{\G}$;
\item $\mbox{Area}(\tilde{\G}_t) \leq \mbox{Area}(\tilde{\G}) + \delta/2$ for $t \in [0,1]$;
\item $\mbox{Area}(\tilde{\G}_1) \leq \mbox{Area}(\tilde{\G}) + \eta$;
\item $\mbox{Area}(\tilde{\G}_1 \cap B_{\hat{r}}(p_j^i)) \leq \mbox{Area}(\tilde{\G} \cap B_{\hat{r}}(p_j^i)) + \eta$ for all $j$;
\item $\tilde{\G}_1$ is stacked in $\Delta^i_l$;
\item Either (i) $\tilde{\G}_1$ is $((1-\tau)h_1, \dots,(1-\tau)h_m )$-stacked in 
$\bigcup_{j \leq l} \Delta_j^i$, or (ii) for some $1\leq j \leq l-1$
we have $k_{ess}(\tilde{\G}_1,i,j, 3\hat{r}/4)< m$.
\end{enumerate}
\end{lemma}

\begin{proof}

In the process of constructing the desired
isotopy with surgeries we may need to apply the squeezing map
$P_{t}$ multiple times for $t = \tau_1, \dots, \tau_n$. In the end we set
$1-\tau = \prod (1-\tau_i)$. In particular,
the composition of all squeezing maps will deform $\phi(\Si \times h')$ into
$\phi(\Si \times (1-\tau)h')$.
Here, recall that the map $\phi$ is defined in (\ref{def phi}).

Recall the notation from Subsection \ref{essmulti}: for $p\in \Si$ we let
$D_{\rho}(p) = P^{-1}(B_{\rho}(p) \cap \Si) \subset \Omega_{\hat{h}}(\Si)$
and
$C_{\rho}(p) = P^{-1}(\partial B_{\rho}(p) \cap \Si) \subset \Omega_{\hat{h}}(\Si)$. Let $r' \in (3\hat{r}/4, \hat{r})$ be chosen as in part II of Lemma \ref{decreasing k_ess}
and let $k = k_{ess}(\tilde{\G}, i, l, r')$.
By Remark \ref{une remark} we have $k \leq k_{ess}(\tilde{\G}, i, l, 3\hat{r}/4)$.
We will apply the deformation $\{\tilde{\G}_t\}_{t 
\in [0,1]}$ from the proof of Lemma \ref{decreasing k_ess},
so that it has the properties (a)-(g) listed in the Lemma.
Now we consider several cases, depending on 
how the surface is stacked in other cells.

{\bf Case 1}: Suppose first that $l=1$ (that is, $\tilde{\G}$ 
has not been stacked yet in any other cell on $\Si_i$). Then  Lemma \ref{decreasing k_ess}
provides an isotopy with surgeries $\{\tilde{\G}_t\}$ such that 
$\tilde{\G}_1$ is $k$-stacked in $\Delta^i_1$ for
$k \leq k_{ess}(\tilde{\G}_1,i,1,3\hat{r}/4)$.

{\bf Case 2}: Suppose $l>1$ and $k \neq m$.
Recall that the numbering of $2$-simplices $\{S^i_j\}$
was chosen so that there exists a simplex
$S^i_{j'}$, $j'<l$, that shares a boundary edge
with $S^i_l$.
Hence, the 3-cell $\Delta^i_{j'}$
has non-empty intersection with $B_{r'}(p_l^i)$.
Since (by the assumption of the Lemma) $\tilde{\G}$ is $m$-stacked in $\Delta^i_{j'}$, it follows from the definition of $k_{ess}$ that
we must have $m>k$. By
applying the deformation of Lemma \ref{decreasing k_ess} to
$\tilde{\G}''$ we  
reduced the essential multiplicity of the surface
in $B_{3\hat{r}/4}(p_{j'}^i)$ (by part II of Lemma \ref{decreasing k_ess}).  Thus, the second case in item (f) is established.  

{\bf Case 3}: 
Suppose $l>1$ and $k = m$. Our goal is to either reduce the multiplicity in one of the cells, or align the heights, so that the surface is stacked in the union of cells
$\{\Delta^i_j\}_{j=1}^l$.
It follows that $\tilde{\G}''$
is $((1-\tau')h_1, \dots,(1-\tau')h_m )$-stacked in 
$\bigcup_{j < l} \Delta_j^i \setminus D_{r'}(p_l^i) $ for some $\tau' \in (0,1)$. By part II of Lemma \ref{decreasing k_ess}
we can assume that $\tilde{\G}''$ is
also $((1-\tau')h_1, \dots,(1-\tau')h_m )$-stacked  in $D_{r' - \xi}(p_l^i)$.

Consider the
very thin annular region 
$A = An(r', r' - \xi, p_l^i) \cap \Si$
for some $\xi> 0$, over which $\tilde{\G}''$
has not been stacked. We can assume that $\xi$
is sufficiently small, so that $A$ does not intersect the $0$-skeleton of the triangulation
$\{S^i_j \}$ of $\Si$.
Fix a $2$-cell $S_j$ for $j<l$ and let $U$ be a connected component of
$S^i_j\cap A$. We
want to deform $\tilde{\G}''$ so that it is 
stacked over $U$ as well. As in the proof of Lemma \ref{decreasing k_ess}, we can apply Lemma \ref{radial} to the ball $B_{r'}(p^i_l)$
to control areas of surfaces in the process of deformation. Hence, we do not need
to worry about the area increase during the deformations of $\tilde{\G}''$ 
inside $P^{-1}(U)$, we only need to make sure that the final surface we obtain has the desired area bound.

Observe that we can decompose
$\partial U$ as the union of $4$
arcs: $c_1$, $c_2$, $c_3$ and $c_4$
with $c_1 \subset \partial B_{r'}(p_l^i) \cap \Si$ and $c_3 \subset \partial B_{r' - \xi}(p_l^i) \cap \Si$; and with $c_2$, $c_3$ each contained in the
interior of an edge of $\partial S^i_j$. 
We have that $\tilde{\G}''$ is $m$-stacked over $c_1$ and $c_3$, but possibly not over $c_2$ and $c_4$. 
We perturb $\tilde{\G}''$ so that it intersects 
$\phi(\partial U \times [0,\hat{h}])$ transversally.
We squeeze every contractible connected component of
$\tilde{\G}''  \cap \phi(\partial U \times [0,\hat{h}])$,
starting with the innermost one, along
a disk contained in $\phi(\partial U \times [0,\hat{h}])$
and apply a neck-pinch surgery.
As in the proof of Lemma \ref{decreasing k_ess} we then remove all closed connected components
via a collapse surgery (Definitions \ref{def:collapse} and \ref{def:collapse1}). 
At the end of this process we obtain an embedded surface that has only 
essential intersections with $\phi(\partial U_j \times [0,\hat{h}])$.
If there are fewer than $m$ essential curves
it follows that we have reduced the essential multiplicity in $\Delta^i_j$ and thus the second alternative in item (f) occurs.
Otherwise, we can deform the surface $\G''$ straightening each segment $\tilde{\G}'' \cap \phi(c_i \times [0,\hat{h}])$, $i =2,4$, so that the height function $H$ is constant along the segment. As in the construction in the proof of Lemma \ref{decreasing k_ess} we can then
deform $\tilde{\G}''$, so that it is $m$-stacked over $U$. 
Proceeding this way for each cell $\Delta^i_j$, $j<l$
that intersects $B_{r'}(p_l^i) \cap \Sigma$
we obtain property (f).  


This concludes the proof.
\end{proof}
\vskip 5pt
\subsection{Proof of Proposition \ref{stacking2}.}

Recall that we chose $\overline{h}>0$ and
$\varepsilon>0$ to be small enough, so that Lemma 
\ref{contracting_hair2} is applicable and
that we fixed some $h\in (0, \overline{h})$.
Let $\delta>0$ be fixed.
By Lemma \ref{contracting_hair2} and using the squeezing map $P_t$ we can assume that $\G \subset \Om_{\hat{h}} \subset N_h(\Si)$, where $\hat{h} \in (0,1)$
was defined in Subsection \ref{r}. 

Recall the definition of $c$ from (\ref{def of c}).
Choose a positive $\hat{\delta} \leq \min \{ \frac{c}{10}, \frac{\delta}{200}\}$ and let $\delta_l = \frac{1}{2^l} \hat{\delta}$. 
Fix a connected component $\Si_i$ of $\Si$.
Recall that $N_i$ denotes the number of 2-simplices in the triangulation $\{S^i_j\}_{j=1}^{N_i}$ of the connected component $\Si_i$ of $\Si$, and that each $p^i_j$ is a point in $S^i_j$, see Subsection \ref{triangulation}.
We will construct a 
sequence of surfaces $\G^0,...,\G^n$ by successive isotopies with surgeries, such that:
\begin{enumerate}

\item $\G^0= \G \cap N_h(\Si_i)$ and 
$\G^n$ is stacked 
over $\Si_i$;


\item  for every $p_j^i$ (where $j=1,...,N_i$),  $$\mbox{Area} (\G^{k} \cap B_{\hat{r}}(p_j^i)) \leq \mbox{Area}(\G \cap B_{\hat{r}}(p_j^i))+ \sum_{l=1}^{N_ik}\delta_l < \frac{\delta}{50};$$ \label{eq local area bound}

\item $\mbox{Area}(\G^k) \leq \mbox{Area}(\G) - k\frac{c}{2}$ for $k = 0,..., n-1$; \label{areadown}

\item concatenating the $n$ isotopies with surgeries bringing $\Gamma^0$ successively to $\Gamma^n$ into a single isotopy with surgeries $\G_t$, we have:
\begin{equation}\mbox{Area}(\G_t) < \mbox{Area}(\G) + \frac{3\delta}{4}\mbox{ for } t \in [0,1].\end{equation}

\end{enumerate}

We now describe the construction of the sequence $\{ \G^k \}_{k=0}^n$ with the desired properties.

We set $\Gamma^0=\Gamma \cap N_h(\Si_i)$. By \eqref{ii} we have, for every $p^i_j$, $$\mbox{Area}(\G \cap B_{\hat{r}}(p_j^i)) \leq \frac{\delta}{100}.$$
Assume we have defined $\G^0, ..., \G^k$.  Let us describe how to obtain $\Gamma^{k+1}$ from $\Gamma^{k}$ for $k=0,...,n-1$.
We first apply Lemma \ref{decreasing k_ess} to $\G^k$ in
$\Delta^i_1$ with $\eta= \delta_{N_ik+1}$ to obtain an isotopy with surgeries from $\G^k$ to a new surface
$\G^k_1$ that is $m$-stacked in the cell $\Delta^i_1$ for some integer $m \geq 0$. We then apply Lemma \ref{stacking} to $\G^k_1$ in the cell $\Delta^i_2$ with $\eta= \delta_{N_ik+2}$ to obtain an isotopy with surgeries from $\G^k_1$ to a new surface $\G^k_2$. Now we have two possibilities: either $\G^k_2$ is $m$-stacked in $\Delta_1^i \cup \Delta_2^i$, or the second possibility in the part (f) of Lemma \ref{stacking} holds,
and we have that the essential multiplicity of the surface in $\Delta^i_1$ has dropped, $k_{ess}(\G^k_2, i, 1, 3 \hat{r}/4)< m$.

Suppose that the first case holds. We continue to successively apply Lemma \ref{stacking} in $\Delta_j^i$, $j=3,...,N_i$, to define an isotopy with surgeries from $\G^k$ to a surface $\G^k_j$ that 
is $m$-stacked in $\bigcup_{j'=1}^j \Delta_{j'}^i$. At each step we
set the constant $\eta$ that controls the area increase in Lemma \ref{stacking} to be $\delta_{N_ik+j}$.
If during each application of Lemma \ref{stacking} we never encounter
possibility (ii) of the part (f), then we have obtained 
an isotopy with surgeries from $\G^k$ to a surface $\G^k_{N_i}$
that is $m$-stacked in the neighborhood of $\Si_i$. We then set
$k+1 =n$ and $\G^n = \G^{k+1} = \G^k_{N_i}$ is the final surface in our construction. 
For each $p^i_{j'}$ by Lemma \ref{stacking} and the inductive assumption for $\G_k$ we have the local area bound
\begin{align*}
\mbox{Area} (\G^{k}_j \cap B_{\hat{r}}(p^i_{j'})) & \leq
\mbox{Area} (\G^{k} \cap B_{\hat{r}}(p^i_{j'}))
+\sum_{l=N_ik+1}^{N_ik+j} \delta_l \\
&  \leq
\mbox{Area} (\G \cap B_{\hat{r}}(p^i_{j'}))
+\sum_{l=1}^{N_ik+j} \delta_l < \frac{\delta}{50}.
\end{align*}

Let $\G'_t$ denote the isotopy with surgeries from $\G^k$ to $\G^n$ that we just constructed. By the inductive assumption for $\G^k$ and Lemma \ref{stacking} we have the global area bound:
\begin{align*}
\mbox{Area} (\G'_t) & \leq \mbox{Area}(\G^{k})  + \sum_{l=N_ik+1}^{N_i(k+1)}\delta_l + \delta/2 \\
& < \mbox{Area}(\G) + \sum_{l=1}^{N_i(k+1)}\delta_l+ \delta/2 \\
& <  \mbox{Area}(\G) +  \frac{3\delta}{4}\mbox{ for } t \in [0,1].
\end{align*}
We conclude that the properties of the sequence $\G^0, ..., \G^n$ are satisfied.

Now suppose that for some $j \leq N_i$ if we apply Lemma \ref{stacking} to $\G^k_{j-1}$ in $\Delta^i_j$ to obtain an isotopy with surgeries to 
a surface $\G^k_j$ we have that  
possibility (ii) of part (f)
holds. Then there exists a cell $\Delta_l^i$, $l < j$, such that $k_{ess}(\G^k_{j},i,l, \frac{3}{4}\hat{r})<m$. Applying Lemma \ref{decreasing k_ess} to $\G^k_j$ we obtain an isotopy with surgeries from $\G^k_j$
to what we denote $\G^{k+1}$, satisfying (using our inductive assumption for $\G^k$)
\begin{align*}
\mbox{Area} (\G^{k+1}) & \leq \mbox{Area}(\G^{k}_j) +\delta_{N_ik+j}-c \\
&  \leq \mbox{Area}(\G^{k}) + \sum_{l=N_ik+1}^{N_ik+j} \delta_l -c \\
& \leq \mbox{Area}(\G) - \frac{kc}{2} - \frac{c}{2}.
\end{align*}

Hence, we obtain the desired isotopy with surgeries from $\G^k$ to $\G^{k+1}$, proving the inductive step in our construction.
This finishes the construction of the sequence $\G^0, ..., \G^n$ with properties (1)--(4).

Observe that by property \eqref{areadown} sequence $\{\G^k\}_{k=1}^n$ must be finite. Hence, the above construction gives us an isotopy with surgeries to a surface that is stacked in the neighborhood of $\Si_i$. We apply this construction for each connected component $\Si_i$ of $\Si$. This finishes the proof of Proposition \ref{stacking2}.

\vspace{2em}

\section{Interpolation II: Finding and opening a neck} \label{second inter}

Let us recall the assumptions and notations of Theorem \ref{maininterpolation}. 
Let $( M_0,g)$ be a closed, oriented, Riemannian $3$-manifold.  Let $\overline{\G}\subset M_0$ be a strongly irreducible Heegaard surface. 
Let $M\subset M_0$ be a compact three-manifold with partitioned boundary $(\partial_0 M, \partial_1 M)$ and 
let $(W_0,W_1,\G)$ be a generalized Heegaard splitting of $(M,\partial_0 M, \partial_1 M)$ 
so that $\G$ is isotopic to $\overline{\G}$ in $ M_0$. Suppose further that each component of $\Si:=\partial M$ is a strictly stable minimal surface.  
Write 
$$\Sigma = \sqcup_{k=1}^\nu \Sigma_i,$$ where $\Sigma_i$
are the boundary components of $M$.
Let $\overline{h}, \varepsilon >0$,
be the constants given by Proposition \ref{stacking2}.



\begin{prop}[Stacking multiplicity larger than 1]\label{movehandles}
Fix any $h\in (0,\overline{h})$.
Let $\delta>0$. There exists 
$c = c(\delta, \G, \Sigma)>0$ with the following property.
Suppose that 
there exists an isotopy 
with surgeries $\{\Gamma_t\}_{t\in[0,1]}$
and a collection of integers $\{ k_i\geq 0 \}_{i=1}^{\nu}$, such that:
\begin{enumerate}
\item $\Gamma_0=\Gamma$,
\item $\mbox{Area}(\Gamma_t) < \mbox{Area}(\Gamma) + \delta$ for all $t\in [0,1]$,
\item $\Gamma_1 \subset N_h(\Sigma)$ and for each $i=1,2,...,\nu$, $\Gamma_1$ is 
$k_i$-stacked over $\Si_i$ in $N_h(\Si_i)$.
\end{enumerate}
Suppose $k_i>1$ for some $i\in \{1,...,\nu\}$.
Then there exists an isotopy with surgeries $\{\Gamma'_t\}_{t\in [0,1]}$ in $M$, beginning at $\Gamma$ so that
\begin{enumerate}
\item $\mbox{Area}(\Gamma'_t) < \mbox{Area}(\Gamma)+\delta$ for all $t\in [0,1]$,
\item $\mbox{Area}(\Gamma'_1) < \mbox{Area}(\Gamma)-\frac{c}{3}$,
\item $\Gamma'_1 \subset N_h(\Sigma)$.
\end{enumerate}

\end{prop}

The crucial point in the proposition above is the uniform area decrease in item (2) of the conclusion.
Proposition \ref{stacking2} together with Proposition \ref{movehandles}  yield our main interpolation result Theorem \ref{maininterpolation}:

\begin{proof}[Proof of Theorem \ref{maininterpolation}]

Choose $h \in (0,\overline{h}), \varepsilon >0$ to
be the constants from Proposition \ref{stacking2}.
Fix $\delta>0$ and let $c=c(\delta, \G, \Sigma)$ be defined as in Proposition \ref{movehandles}.
Let  $\hat{\delta}  \in (0, \min\{\delta, \frac{c}{100}\})$. 
Let $\mathcal{S}$ denote the set of surfaces $S$ with the following property:
there exists an isotopy with surgeries $\{S_t\}_{t\in[0,1]}$, such that:
\begin{enumerate}
\item $S_0=\Gamma$,
\item $\mbox{Area}(S_t) < \mbox{Area}(\Gamma) + \delta$ for all $t\in [0,1]$,
\item $S_1 = S \subset N_h(\Sigma)$ and $S_1$ is 
stacked in $N_h(\Si_i)$ for each $i=1,2,...,\nu$. 
\end{enumerate}
By Proposition \ref{stacking2} 
the set $\mathcal{S}$ is non-empty.

Let $\G' \in \mathcal{S}$ be an embedded surface satisfying
\begin{equation} \label{eq:c/3}
    \mbox{Area}(\G') < \inf_{S \in \mathcal{S}} \mbox{Area}(S) + \hat{\delta},
\end{equation}
We claim that $\G'$ is $k_i$-stacked in the neighborhood 
of each boundary component $\Sigma_i$ with $k_i=1$ or $0$.
Indeed, if $k_i>1$, then by Proposition \ref{movehandles} and (\ref{eq:c/3}) there exists an isotopy with surgeries from $\G$ to an embedded surface
$\G''$, so that
\begin{itemize}
    \item the areas of surfaces in the deformation are less than $\mbox{Area}(\G)+ \delta$;
    \item $\mbox{Area}(\G'')< \mbox{Area}(\G') - \frac{c}{3}<
    \inf_{S \in \mathcal{S}} \mbox{Area}(S) - \frac{c}{4}$.
\end{itemize}
Now we can apply Proposition \ref{stacking2} to the surface $\G''$ to define an isotopy with surgeries from 
$\G''$ to $\G'''$, so that 
\begin{itemize}
    \item the areas of surfaces in the deformation are less than $\mbox{Area}(\G'')+ \hat{\delta}$;
    \item $\G'''$ is stacked in the neighborhood of each $\Si_i$.
\end{itemize}
Concatenating these isotopies with surgeries we obtain an isotopy with surgeries from $\G$ to a stacked surface $\G'''$. It follows that $\G''' \in \mathcal{S}$ and
$$\mbox{Area}(\G''') < \inf_{S \in \mathcal{S}} \mbox{Area}(S) - \frac{c}{100},$$
a contradiction. We conclude that $\G'$ is $k_i$-stacked in the neighborhood of each $\Si_i$ with $k_i = 1$ or $0$.

By Lemma \ref{neck-pinching isotopy} there exists an isotopy 
from $\G$ to an embedded surface $\G_1$ arbitrarily close in the $\mathbf{F}$-topology
to $\G'$. Applying the squeezing map $P_t$ (see (\ref{P_t def})) if necessary,
it follows that $\G_1$ will satisfy the conclusions of 
Theorem \ref{maininterpolation}.
\end{proof}

The rest of this section is devoted to the  proof Proposition \ref{movehandles}.

\subsection{Preliminaries on the light bulb theorem.}
Proposition \ref{movehandles} tells us that if we can deform the surface $\Gamma$ to a surface that looks like several copies of the connected components of the boundary with thin necks in between, then we can use one of these necks to reduce the area of the surface by a definite amount. One of the 
difficulties in proving it is that if the neck is knotted as in the Fig. \ref{fig:essential}, it may be impossible to ``open up" this neck by a local deformation. To prove existence of a global deformation that will unknot such necks we will need the light bulb theorem in topology, which we recall:

\begin{prop} [Light bulb theorem \cite{Rolfson}] \label{lb}
Let $\alpha(t)$ be an embedded arc in  $\mathbb{S}^2\times [0,1]$ so that $\alpha(0)=\{x\}\times\{0\}$ and $\alpha(1)=\{y\}\times\{1\}$ for some $x,y\in\mathbb{S}^2$.  Then there is an isotopy $\phi_t$ of $\alpha$ so that 
\begin{enumerate}
\item $\phi_0(\alpha)=\alpha$ 
\item $\phi_1(\alpha)$ is the vertical arc $\{x\}\times [0,1]$.
\end{enumerate}
\end{prop}

The light bulb theorem can be interpreted physically as that one can untangle a lightbulb cord hanging from the ceiling and attached to a lightbulb by passing the cord around the bulb many times.  

The simplest nontrivial case of Proposition \ref{movehandles} consists of two parallel spheres joined by a knotted neck.  Here Proposition \ref{lb} allows us to untangle this neck so that it is vertical and contained in one of the balls $B_i$. Then we can apply Corollary \ref{coro:area decrease} to reduce the area of the surface.


We will in fact need the following generalization of the light bulb theorem (cf. Proposition 4 in  \cite{HT}):

\begin{prop}[Generalized light bulb theorem]\label{general}
Let $M$ be a $3$-manifold and $\alpha$ an arc with one boundary point on a sphere component $\Si'$ of $\partial M$ and the other on a different boundary component.  Let $\beta$ be a different arc with the same end points as $\alpha$.  If $\alpha$ and $\beta$ are homotopic, then they are isotopic.  

Moreover, if $\gamma$ is an arc freely homotopic to $\alpha$ (i.e. joined through a homotopy where the boundary points are allowed to slide in the homotopy along $\partial M$), then they are freely isotopic (i.e., they are joined by an isotopy with the same property).
\end{prop}

\noindent
\emph{Sketch of Proof:}
 The homotopy between $\alpha$ and $\beta$ can be realized by a family of arcs $\alpha_t$ so that $\alpha_t$ is embedded or has a single double point for each $t\in [0,1]$.  If $\alpha_{t_0}$ contains a double point, the curve $\alpha_{t_0}$ consists of three consecutive embedded sub-arcs $\alpha_{t_0}([0,a])$, $\alpha_{t_0}([a,b])$, and $\alpha_{t_0}([b,1])$ so that without loss of generality $\alpha_{t_0}([0,a])$ connects to the 2-sphere $\Si'$. Fix a small constant $\varepsilon>0$. For $t_1$ $\varepsilon$-close to $t_0$ so that $\alpha_{t_1}$ is an embedding close to $\alpha_{t_0}$ in the smooth topology, we can pull the arc $\alpha_{t_1}([b-\varepsilon,b+\varepsilon])$ 
 along the arc $\alpha_{t_1}([0,a])$ and then pull it ``over'' the sphere $\Si'$, and then pull the arc back to a position close to its original position, but on the other side of $\alpha_{t_1}([a-\varepsilon,a+\varepsilon])$. 
 We then use this deformation to smoothly replace $\alpha_t$ for $t\in [t_0-\varepsilon,t_0+\varepsilon]$. Repeating this procedure for each double point, we obtain the desired isotopy.  The proof is illustrated in Figure \ref{fig:lightbulb}.

\qed



\begin{figure} 
   \centering	
	\includegraphics[scale=0.7]{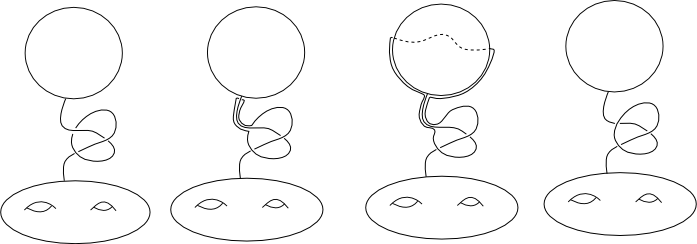}
	\caption{Changing the over-under crossing in the proof of the Light Bulb Theorem.}
	\label{fig:lightbulb}
\end{figure}

We will need to apply Proposition \ref{general}  to an arc inside a compression body in the following situation.

\begin{coro} \label{arc isotopy}
    Let $W$ be a compression body and suppose $R_1$
    is a connected component of $\partial_- W$ that is 
    diffeomorphic to $\mathbb{S}^2$. Suppose that $\theta_i: [0,1] \rightarrow W$, $i=0,1$, are two embedded arcs with $\theta_i(0) \in R_1$, 
    $\theta_i(1) \in \partial_+ W$ for $i=0,1$. Then there exists an isotopy of arcs $\theta_t$ in $W$ with
      $\theta_t(0) \in R_1$, 
    $\theta_t(1) \in \partial_+ W$, $t\in [0,1]$.
\end{coro}

\begin{proof}
Since the boundaries of $\theta_0(0)$, $\theta_0(1)$ and $\theta_1(0)$, $\theta_1(1)$
belong to respectively the same connected components of $\partial W$,
there exists an isotopy of $\theta_0$
to an embedded arc $\tilde{\theta}$ with $\tilde{\theta}(0) = \theta_1(0)$ and $\tilde{\theta}(1) = \theta_1(1)$.
By Proposition \ref{general} it is enough to show
that arcs $\tilde{\theta}$ and $\theta_1$ are homotopic. 

We deform $\tilde{\theta}$ to the curve $\theta_1 \cup (-\theta_1) \cup \tilde{\theta}$.
By the definition of a compression body the map
$\pi_1(\partial_+W) \rightarrow \pi_1(W)$ induced by the inclusion $i: \partial_+W \rightarrow W$ 
is surjective. It follows by the long exact sequence for the pair $(W, \partial_+W)$ that $\pi_1(W, \partial_+W) = 0$. Hence, we can deform arc $-\theta_1 \cup \tilde{\theta}$ to an arc $\alpha$ that lies in
$\partial_+W$ through arcs with endpoints in $\partial_+W$.
 Contracting $\alpha$ by moving its endpoint inside $\partial_+W$ we obtain the desired homotopy from $\tilde{\theta}$ to $\theta_1$.
\end{proof}

\subsection{Proof of Proposition \ref{movehandles}}

Let $\delta>0$. Let $\overline{h}, \varepsilon >0$ be the constants in Lemma \ref{stacking2} and let $h\in (0,\overline{h})$. 

By Proposition \ref{stacking2} applied to $\G$, there is an isotopy with surgeries $\{\G_t\}$ with 
$$\mbox{Area}(\G_t)< \mbox{Area}(\G)+\delta,$$
starting from $\G_0=\G$ and ending at an embedded surface $\G_1$ contained in $N_h(\Si)$, which is stacked over each connected component
$\Si_i$ of $\Si= \partial M$ inside $N_h(\Si)$.
For each connected component $\Si_i$ of $\Si$, $\G_1$ is $k_i$-stacked in $N_h(\Si_i)$: it is a union of sheets 
\begin{equation}\label{sheet ordering}
S_{i,1}=\phi(\Si_i,h_1),...,S_{i,k_i}=\phi(\Si_i,h_{k_i}),
\end{equation}
for some numbers
$$0\leq h_1<...<h_{k_i}\leq \hat{h} $$
where $\hat{h}$ was defined in Subsection \ref{r}.
The surface $\Sigma_i$ is a boundary component of $M$, $S_{i,1}$ is the sheet closest to $\Sigma_i$ while $S_{i,k_i}$ is the innermost sheet inside $M$. By Theorem \ref{strongirreduce} it follows
that $k_i>1$ is only possible if $\Sigma_i$ is a sphere. So in the proof of this proposition, we are assuming that there is a sphere component $\Sigma_i$ for which
\begin{equation}\label{assump5.1}
k_i>1.
\end{equation}

By Lemma \ref{neck-pinching isotopy}, for any $\eta>0$ we can isotopically
deform $\G$ through a smooth isotopy $\{\G'_t\}$ from $\G'_0=\G$ to $\G'_1$
so that for all $t$, 
$$\mbox{Area}(\G'_t)< \mbox{Area}(\G_t)+\eta\quad  (< \mbox{Area}(\G)+\delta+\eta),$$
and $\G'_1$ and $\G_1$ are $\eta$-close as varifolds:
\begin{equation}\label{cloe}
\mathbf{F}(\G'_1, \G_1)< \eta.
\end{equation}

By (\ref{cloe}), for any $\tilde{\epsilon}>0$, choosing $\eta>0$ sufficiently small and using the coarea inequality and Sard's lemma
we can find surfaces 
$$C'_{i,j} = \phi(\Si_i,h^{(j)}),$$ $j=1,...,k_i$,  
$$h_1<h^{(1)}<h_2<...< h^{(j-1)}<h_j < h^{(j)}<... < h_{k_i}<h^{(k_i)},$$
so that $\G'_1 \cap C'_{i,j}$ is a collection of closed curves,
whose total length is smaller than $\tilde{\epsilon}$, where $\tilde{\epsilon}$ can be made arbitrarily small by choosing $\eta>0$ small enough. Let $\{ C_q \}$ be the family of all the surfaces of the form $C'_{i,j}$ for some $i,j$.
Consider the connected components of $\G'_1 \cap \bigcup C_q$: we call these closed curves  $\gamma_1,...,\gamma_K$  so that
\begin{equation}\label{inters gammas}
\G'_1 \cap \bigcup_{l} C_q = \sqcup_{j=1}^K \gamma_j.
\end{equation}
Given $j\in \{1,...,K\}$, let $C_{q'}$ be the surface in $\{C_q\}$ containing $\gamma_j$ and let $D_j$ denote the small disk in  $C_{q'}$ such that $\partial D_j = \gamma_j$.
We number the curves $\gamma_j$ so that if 
$\gamma_{j_1}$ and $\gamma_{j_2}$ are contained in the same surface $C_{q'}$ and the corresponding disks satisfy
$D_{j_1} \subset D_{j_2}$, then $j_2 > j_1$.



Using Lemma \ref{intersection neck-pinch}, we can define an isotopy with one neck-pinch surgery along $D_K$ (which is an innermost disk), starting from $\G'_1$ and ending at a new closed surface $\G'_2$. Inductively we construct a sequence of 
closed surfaces $\G'_1, \dots, \G'_{K+1}$, such that for each $n\in \{1,...,K\}$,
$\G'_{n+1}$ is obtained from $\G'_n$ via an isotopy with one neck-pinch surgery along $D_n$. 
By (\ref{cloe}) and Lemma \ref{intersection neck-pinch}, if $\eta$ was chosen small enough, then all these isotopies with surgeries can be defined so that none of them increase the area by more than $\delta$, and 
\begin{equation*}
    \sum_{n=1}^K{\bf F}(\G'_n, \G_{n+1}')\leq \frac{4}{\pi}\sum_{n=1}^K \mbox{Length}(\gamma_i)^2
\leq \frac{4 \tilde{\varepsilon}^2}{\pi} 
\end{equation*}
Consequently, for an arbitrarily small $\eta_1>0$ we can choose $\eta>0$ to be sufficiently small, so that
\begin{equation}\label{g'n}
{\bf F}(\G'_n, \G_1) \leq \sum_{k=1}^{n-1} {\bf F}(\G'_k, \G_{k+1}') + {\bf F}(\G'_1, \G_1)< \frac{4 \tilde{\varepsilon}^2}{\pi} + \eta < \eta_1
\end{equation}


\vspace{1em}

In the rest of the proof, we will  argue that one of the surfaces $\G'_n$ has a neck which can be moved into a small region and be ``opened up". We will then apply Corollary \ref{coro:area decrease} and Lemma \ref{neck-pinching isotopy} to reduce the area of $\G'_1$ by a uniform amount in some cell $\Delta_i$ via a smooth isotopy.

For each $n$, consider the change in the topology of the surface when going back from $\G'_{n+1}$ to $\G'_{n}$. We have two possibilities: either a handle is attached to a connected component of $\G'_{n+1}$ without changing the other components, or two components of $\G'_{n+1}$ are connected without changing the other components.
Observe that due to (\ref{cloe}), for every sufficiently small $\eta_2>0$ we can find  sufficiently small $\eta>0$, so that
the following holds:
for each sheet $S_{i,j}$ of the surface $\G_1$ in the neighborhood of $\Sigma_i$, there is a unique connected component $\tilde{S}_{i,j}$ of $\G'_{K+1}$ in the tubular neighborhood  $N_h(\Si_i)$, with 
\begin{equation}\label{stildeij}
\sum_{i,j}{\bf F}(\tilde{S}_{i,j}, S_{i,j})< \eta_2.
\end{equation}
Moreover, any other connected component of $\G'_{K+1}$ which is not $\eta_2$-close in the ${\bf F}$-topology to a sheet of $\G_1$ has area smaller than $\eta_2$. 
 In particular, we can choose $\eta, \eta_2$ small enough (depending only on $(M,g)$) so that all such small area connected components 
 bound a handlebody of small volume by Proposition \ref{small volume handlebody} (here we use that by assumption, $\G$ is a strongly irreducible Heegaard splitting of the closed 3-manifold $( M_0,g)$ which contains $M$).

\vspace{1em}

Recall that $ \partial M = \Si = \bigcup_{i} \Si_i$ and that by assumption, we have  a strongly irreducible generalized Heegaard splitting $(W_0,W_1,\Gamma)$
of $(M,\partial_0 M, \partial_1 M)$. 
A key observation is the following lemma, which we informally describe as follows: there exists the smallest integer $l_0>0$, such that the surface $\G'_{K+1-l_0}$ has a connected component containing two subsurfaces that are right next to each other and close (in $\mathbf{F}$ topology) to a boundary component $\Si_i\subset \partial M$; moreover, $\G'_{K+2-l_0}$ is obtained from $\G'_{K+1-l_0}$ by a neck-pinch that disconnects these two subsurfaces.

\begin{figure}
\label{fig:neckcuts}
   \centering	
	\includegraphics[scale=1.15]{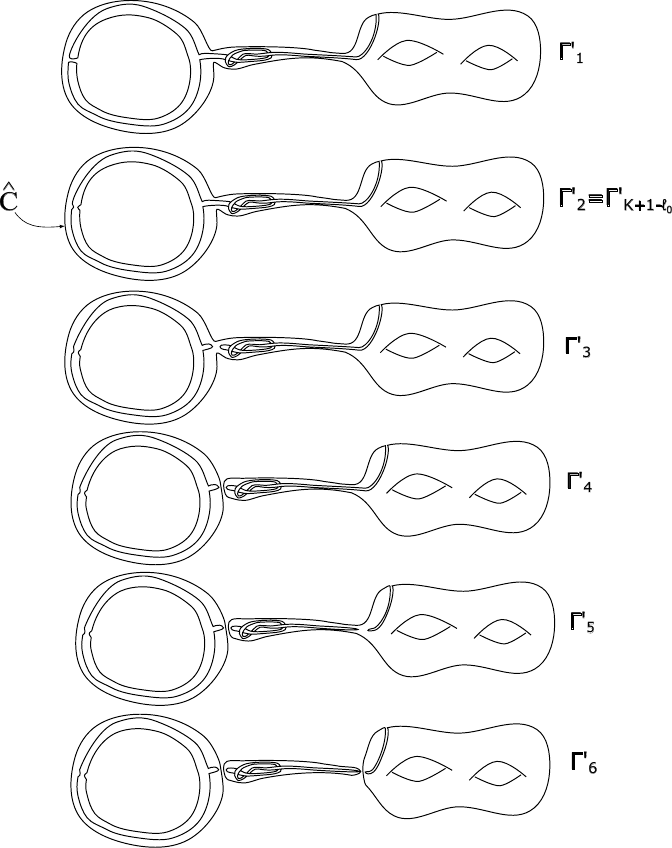}
	\caption{The sequence of surfaces $\G'_n$ with $K+1=6$ and $l_0 =4$.}
\end{figure}

\begin{lemma}[Definition of $l_0\geq1$] \label{observation}
Let $\eta_3>0$. If $\eta$ was chosen small enough, then there is an integer 
$$l_0\in \{1,...,K\}$$ such that the following holds. 

For some connected component $\hat{C}$ of
$\G'_{K+1-l_0}$, and for some sphere component $\Si_{i}$
of $\Si$, we have
\begin{equation}\label{tildes3}
{\bf F}(\hat{C}\cap N_h(\Sigma_i), S_{i,j_0} \cup S_{i,j_0+1})< 3\eta_3
\end{equation}
where $S_{i,j_0}$ and  $S_{i,j_0+1}$ are two consecutive sheets of $\G_1$ in the tubular neighborhood $N_h(\Si_{i})$.

Moreover, for each $n=K+2-l_0,...,K+1$ and for any connected component $C'$ of $\G'_n$, 
\begin{enumerate}
    \item either $\mbox{Area}(C')<\eta_3$ and $C'$ bounds a handlebody of small volume, 
    \item or 
    $$C'=V \cup \bigcup_{m=1}^Q A_m$$ where  $A_1,...,A_Q$ are some compact surfaces with boundary, respectively contained in disjoint tubular neighborhoods:
    $$A_1\subset N_h(\Sigma_{i_1}), ..., A_Q\subset N_h(\Sigma_{i_Q})\quad \text{($i_1,...,i_Q$ are all distinct)},$$ such that for each $m\in \{1,...,Q\}$, there is a sheet $S_{i_m,j_m}\subset \G_1$ with 
    $${\bf F}(A_m, S_{i_m,j_m})< \eta_3$$
    and $V$ is a compact surface with boundary such that $\mbox{Area}(V)<\eta_3$. Moreover, $C'$ bounds a compression body $W'$ with 
    $$\partial_+W' = C'\quad  \text{and} \quad \partial_-W' = \bigcup_{m=1}^Q\Si_{i_m}.$$

\item 
Additionally, for two components $C'$, $C''$ of $\G'_n$, if $W'$, $W''$ denote the handlebody or compression body bounded by $C'$ and $C''$ given by (1) and (2) above, then either $W'\cap W''=\varnothing$, or  $W'\subset W''$, or $W''\subset W'$.
\end{enumerate}
\end{lemma}

\begin{remark}
In plain words, Items (1) and (2) in the lemma state that for each connected component $C' $ of $\G'_{n}$ for $n = K+2-l_0,...,K+1$,
 \begin{itemize}
 \item either $C'$ has small area, in which case it bounds a small volume handlebody; 
 \item or $C'$ is sheet-like, in which case it bounds a compression body $W'$ whose negative boundary $\partial_-W \subset \partial M$ is close to $C'$ in the $\mathbf{F}$-topology. 
 \end{itemize}
 Here, $C' $ being ``sheet-like'' means that in at least one $N_h(\Si_i)$, 
 $${\bf F}(C'\cap N_h(\Si_i), S_{i,j_{C'}})< \eta_3,$$
 for some sheet $S_{i,j_{C'}}$ of $\G_1$. In Figure \ref{fig:neckcuts}, we have illustrated an example where $K+1=6$ and $l_0=4$.
\end{remark}

\begin{proof}

By choosing $\eta$ small enough, we can take $\eta_1$ in (\ref{g'n}) and $\eta_2$ in (\ref{stildeij}) to be smaller than $\eta_3$.
By a backward induction on the surfaces $ \G'_{K+1}, \G'_{K},  \G'_{K-1}...$, we will prove that for every connected component Items (1), (2), (3) in the lemma hold until it stops being true, in which case (\ref{tildes3}) holds for some connected component $\hat{C}$.

We start by showing Items (1), (2), (3) for $\G'_{K+1}$.
For each component $\Si_i$ of $\Si$, each sheet $S_{i,j}$ of $\G_1$ near $\Si_i$ is a graph over $\Si_i$ so it bounds a trivial compression body contained inside the tubular neighborhood $N_h(\Si_i)$ whose positive boundary is $S_{i,j}$ and negative boundary is $\Si_i$. 

Set $\tilde{S}_{i,j}$ to be the unique connected component  of $\G'_{K+1}$, such that
\begin{equation}\label{small eta volume}
{\bf F}(\tilde{S}_{i,j}, S_{i,j})< \eta_2
\end{equation}
(see (\ref{stildeij})). 
Let us justify that for each $S_{i,j}$, the surface $\tilde{S}_{i,j}$ also bounds a compression body $W_{K,i,j}$ with negative boundary $\Si_i$, namely: 
$$\partial_+W_{K,i,j} = \tilde{S}_{i,j},\quad \partial_-W_{K,i,j} = \Si_i \subset \partial M.$$
First, note that after doing some neck-pinch surgeries on $\tilde{S}_{i,j}$, we obtain the union of a surface isotopic to $S_{i,j}$ and some closed surfaces isotopic to surfaces of arbitrarily small area.
To see this, we first observe that our $\F$-distance estimates (\ref{small eta volume}) imply that there exists a point $q \in \Sigma_i$, such that the fibre of the projection map $P^{-1}(q) \subset \Om_h(\Sigma_i)$ intersects $\tilde{S}_{i,j}$ exactly once. It follows that $\tilde{S}_{i,j}$ is homotopically nontrivial  inside $\Om_h(\Sigma_i)$. 
Applying Proposition \ref{stacking2} to $\tilde{S}_{i,j}$ we obtain an isotopy with surgeries to a surfaces that is $k$-stacked in the tubular neighborhood $N_h(\Sigma_i)$. Moreover, we have that $k<2$ by the area estimates and $k\geq 1$ since $\tilde{S}_{i,j}$ is homotopically nontrivial  inside $\Om_h(\Sigma_i)$. We conclude that $k=1$. To summarize, there exists a finite number of neck-pinch surgeries that can be applied to $\tilde{S}_{i,j}$ to obtain a surface isotopic to $\Sigma_i$ (and hence $S_{i,j}$) plus a finite union of surfaces isotopic to surfaces of arbitrarily small areas.
The area of those additional surfaces can be made smaller than the $\beta(M,g)$ of Proposition \ref{small volume handlebody}.
Proposition \ref{small volume handlebody} implies that those small area surfaces all bound handlebodies.
Now, by applying Corollary \ref{finding_neck} (2) repeatedly, we finish the justification. 


{
For $i\neq i'$, and given $j,j'$, the compression bodies $W_{K,i,j}$ and $W_{K,i',j'}$, bounded by $\tilde{S}_{i,j}$ and $\tilde{S}_{i',j'}$, are disjoint:
\begin{equation}\label{disj purple}
W_{K,i,j}\cap W_{K,i',j'} =\varnothing.
\end{equation}
Indeed, if their intersection was non-empty, then since by embeddedness, $\tilde{S}_{i,j} \cap \tilde{S}_{i',j'}=\varnothing$ and since 
$$N_h(\Si_i) \cap N_h(\Si_{i'})=\varnothing,$$ we would have that  
$$\tilde{S}_{i,j}\subset W_{K,i',j'} \text{ and }\tilde{S}_{i',j'}\subset W_{K,i,j},$$
and $\tilde{S}_{i,j}$ separates $\partial_+W_{K,i',j'} = \tilde{S}_{i',j'}$ from $\partial_-W_{K,i',j'}$. 
This is impossible because  by (\ref{stildeij}), 
only a tiny fraction of the surface $\tilde{S}_{i,j}$ (resp. $\tilde{S}_{i',j'}$) is contained outside of $N_h(\Si_i)$ (resp. $N_h(\Si_{i'})$).
}

We also note for a fixed $i$, the compression bodies $\{W_{K,i,j}\}_{j}$ share the same negative boundary $\Si_i$ and have connected positive boundaries. So, for any two of them, one is contained in the other. Because of our numbering of the sheets $S_{i,j}$, see (\ref{sheet ordering}), they satisfy the following monotonicity property for any $i$: 
\begin{equation}\label{inclu purple} W_{K,i,j'}\subset W_{K,i,j} \text{ for $j'<j$}.\end{equation}

On the other hand, the connected components of $\G'_{K+1}$ which are not of the form $\tilde{S}_{i,j}$ for some $i$ and $j$, have areas smaller than $\min\{\eta_3,\beta(M,g)\}$ if $\eta$ is small enough by (\ref{g'n}), and so they bound handlebodies of small volumes by Proposition \ref{small volume handlebody}. Note that such small handlebodies clearly cannot contain any surface of the form $\tilde{S}_{i,j}$ for some $i,j$.

Then, Item (3) for $\G'_{K+1}$ follows from (\ref{disj purple}) and (\ref{inclu purple}).
The above also proves Items (1), (2) for $\G'_{K+1}$, since any connected component $C'$ of $\G'_{K+1}$ either is of the form $\tilde{S}_{i,j}$ so that it bounds the compression body $W_{K,i,j}$ with negative boundary $\Si_i$, with volume less than half of $\Vol(M,g)$, and it is contained in $N_h(\Si_i)$ (the $A_1$ and $V$ in the statement of Item (2) correspond to $\tilde{S}_{i,j}$ and the empty set respectively), or it has area at most $\eta_3$ and bounds a compression body. 
This concludes the proof of the base of induction.

\vspace{1em}

{
 Next, assume that Items (1), (2), (3) are true for $\G'_{n+1}$. Below, a connected component $C' $ is said to be ``sheet-like'' if in at least one $N_h(\Si_i)$, 
 $${\bf F}(C'\cap N_h(\Si_i), S_{i,j_{C'}})< \eta_3,$$
 for some sheet $S_{i,j_{C'}}$ of $\G_1$.
}

{
 Three cases can occur when going from $\G'_{n+1}$ to $\G'_n$:
 \begin{enumerate} [label=(\Roman*)]
 \item either a connected component $\hat{C}$ of $\G'_n$ is the connected sum of
 a connected component $C'$ of $\G'_{n+1}$ with a small area surface $C''$, and the other components of $\G'_n$ are also components of $\G'_{n+1}$; 
 \item or a connected component $\hat{C}$ of $\G'_n$ is the result of adding a small handle to a connected component $C'$ of $\G'_{n+1}$ and the other components of $\G'_n$ are also components of $\G'_{n+1}$; 
 \item or else a component $\hat{C}$ of $\G'_n$ is the result of connecting two components $C'$, $C''$ of $\G'_{n+1}$ which are both sheet-like, and the other components of $\G'_n$ are also components of $\G'_{n+1}$. 
 \end{enumerate}
 }
 
 In Cases (I) and (II), Items (1) and (2) are also true for $\G'_n$ by Corollary \ref{finding_neck} (2). Item (3) is also clearly true in these cases since Item (3) holds for $C'$ and $C''$.  
 
In Case (III), using the induction hypothesis, we write for some $Q',Q''\geq 1$: 
$$C' = \bigcup_{m=1}^{Q'} A'_m \cup V'\quad \text{with } \quad A'_1\subset N_h(\Sigma_{i'_1}), ...,A'_{Q'}\subset  N_h(\Sigma_{i'_{Q'}}),$$
and 
$$C''= \bigcup_{m=1}^{Q''} A''_m \cup V'',\quad \text{with } \quad
A''_1\subset N_h(\Sigma_{i''_1}), ...,A''_{Q''}\subset  N_h(\Sigma_{i''_{Q''}}),$$
where the surfaces $A'_m,V',A''_m,V''$ are as in Item (2) of the lemma.
By the induction hypothesis, $C'$, $C''$ respectively bound compression bodies $W'$, $W''$ satisfying Item (2) in the lemma. We also have by Corollary \ref{finding_neck} (1) that
$$\text{either} \quad W' \cap W''=\varnothing \quad \text{or} \quad W'\subset W'' \quad\text{or} \quad W''\subset W'.$$

Case (IIIa): assume $W' \cap W''=\varnothing$. Then 
$$\partial_-W' \cap \partial_-W'' =\bigcup_{m=1}^{Q'}\Sigma_{i'_m}\cap\bigcup_{m=1}^{Q''}\Sigma_{{i''}_m} =\varnothing.$$
Items (1), (2) of the lemma are thus true for $\G'_n$ by Corollary \ref{finding_neck} (2) and the induction hypothesis when $\eta$ is small enough. Item (3) is true by the induction hyptothesis and the same argument that showed (\ref{disj purple}).

Case (IIIb): assume $W'\subset W''$ (the last case $W''\subset W'$ is treated in a similar way). 
Then we have
$$\varnothing \neq \partial_-W'=\bigcup_{m=1}^{Q'}\Sigma_{i'_m} \subset \partial_-W''=\bigcup_{m=1}^{Q''}\Sigma_{{i''}_m}.$$ 


We can assume, after renumbering the $i'_m$ and $i''_m$, that 
$$ i'_1=i''_1\text{ and }\Sigma_{i'_1} = \Sigma_{{i''}_1}.$$ Recall that $C'$ and $C''$  in $\G'_{n+1}$ get connected together to yield $\hat{C}$ in $\G'_n$, and that this attachment does not modify the other components of $\G'_{n+1}$. Recall also that by induction any component of $\G'_{n+1}$ bounds either a small volume handlebody or a compression body whose negative boundary is inside $\Sigma = \partial M$. 
By the induction hypothesis and Item (3), any two of those compression bodies are either  pairwise disjoint (when their negative boundary are disjoint) or one is contained in the other (when their negative boundary share some components).
It follows that $A'_1$ and $A''_1$ have to be close to \emph{consecutive} sheets of $\G_1$ in $N_h(\Si_{i'_1})$: for some $j_0$
$${\bf F}(A'_1, S_{i'_1,j_0})< \eta_3 \text{ and } {\bf F}(A''_1, S_{i'_1,j_0+1})< \eta_3.$$
But this implies (\ref{tildes3}). In this case, we finish the induction and  we set  $l_0:= K+1-n$ for that particular $n$. 

In summary, at each step of the induction, we have essentially four possibilities: Cases (I), (II), (IIIa) and (IIIb). In all cases except (IIIb), we can continue the induction. But in Case (IIIb),  (\ref{tildes3}) holds, in other words $\G'_{n+1}$ has a component containing two consecutive sheet-like subregions.

Since Items (1), (2), (3) of the lemma do not hold for $\G'_1$ by our assumption (\ref{assump5.1}), there exists a largest integer $n_0 \in \{1,...,K+1\}$ such that Items (1), (2), (3) are not true for $\G'_{K+1-n_0}$ (i.e. Case (IIIb) happened), in which case we set  $l_0:= K+1-n_0$. This completes the proof of the lemma.
\end{proof}

Consider the curve $\gamma$ along which the neck-pinch 
surgery was performed when passing from
$\G'_{K+1-l_0}$ to $\G'_{K+2-l_0}$. 
The neck-pinch surgery is performed on the  connected component $\hat{C}$ of $\G'_{K+1-l_0}$ in the notation of Lemma \ref{observation}.
In other words, the surface $\hat{C}$ is the result of two connected components of $\G'_{K+2-l_0}$ merging together. 
Using the notations of Lemma \ref{observation}, the two components each contain a subregion  close (in the varifold sense) to $S_{i,j_0}$ and $S_{i,j_0+1}$ respectively. 
We now show the following: \newline
\begin{lemma} \label{r1r2}
The two connected components of $\G'_{K+2-l_0}$ which become merged  in $\G'_{K+1-l_0}$ can be named 
$R_1$ and  $R_2$ so that the following holds: $R_2$ bounds a compression body $W$ in $M$ with 
$R_1\subset \partial_-W$,  $R_1$ is a 2-sphere, and the interior of $W$ is disjoint from $\G'_{K+2-l_0}$.
\end{lemma}
\begin{proof}

Suppose first that $ M_0$ is a $3$-sphere, so  that $\G$ is a 2-sphere or a 2-torus. Since the genus does not increase under surgery, the genus of the union $R_1 \cup R_2$ is at most $1$. We can choose $R_1$ to be a component with genus $0$, namely a 2-sphere. If $R_2$ is a 2-sphere, then the statement of the lemma is clear. 
If $R_2$ is a 2-torus, since the Heegaard splitting surface $\Gamma$ is a torus, all the neck-pinch surgeries that occur between $\G'_1$ and $\G'_{K+2-l_0}$ were non-essential and $R_2$ is isotopic to $\G$.
Thus $R_2$ bounds a compression body on both of its sides. 
In particular, $R_2$ bounds a compression body $\hat{W}$ containing the 2-sphere $R_1$. 
All the other connected components of $\G'_{K+1-l_0}$ 
 are 2-spheres, and those that are in $\hat{W}$ are contained in a 3-ball inside $\hat{W}$ we will call $\hat{B}$. The lemma is then proved in this case by taking $W = \hat{W}\setminus \hat{B}$.

 Next assume that that $ M_0$ is not topologically a 3-sphere. 
 With the notation $j_0$  introduced in the definition of $l_0$ (Lemma \ref{observation}), let $R_1$ (resp. $R_2$) denote the component of $\G'_{K+2-l_0}$ with a ``sheet-like'' piece close to $S_{i,j_0}$ (resp. $S_{i,j_0+1}$), where $S_{i,j_0}$ and  $S_{i,j_0+1}$ are consecutive sheets of $\G_1$ in $N_h(\Si_i)$ for some 2-sphere component $\Si_i$ of $\Si=\partial M$. Then by Lemma \ref{observation}, $R_1$, $R_2$ respectively bound compression bodies $W_1$, $W_2$ whose negative boundaries contain $\Si_i$. In particular, $W_1\cap W_2\neq \varnothing$.
Since $W_2$ is not a subset of $W_1$, by Corollary \ref{finding_neck} (1),  $W_1\subset W_2$, $W_1$ is topologically a 3-ball minus some 3-balls, and $R_1 = \partial_+ W_1$ is a 2-sphere. By Lemma \ref{observation} (3), all the connected components of $\G'_{K+1-l_0}$ contained in $W_2$ bound  handlebodies or compression bodies inside $W_2$. By Corollary \ref{finding_neck} (1), all connected components of $\G'_{K+1-l_0}$ contained in $W_2$ are actually 2-spheres. 
Set $W$ to be $W_2$ minus all the handlebodies and compression bodies bounded by those 2-spheres, so $W$ is clearly a compression body. We see that the 2-sphere $R_1$ is inside $\partial_-W$ because by Lemma \ref{observation} (1), $R_1$ and $R_2$ contain subregions lying next to each other and close to a boundary component $\Si_i$.
The lemma is proved.
\end{proof}

We are finally ready to finish the proof of Proposition \ref{movehandles}. 
Recall that we defined the special radius $\hat{r}>0$ in Subsection \ref{r}, the special points $p^i_m\in \Si$ in (\ref{pij def}), the cylindrical regions $D_{\hat{r}/2}(p^i_m)$ and the cylindrical surfaces $C_{\hat{r}/2}(p^i_m)$ in (\ref{def D_r}) and (\ref{def C_r}).
Fix the sphere component $\Si_i\subset \Si$ given by Lemma \ref{observation}, and let 
$$p^i_m\in \Si_i$$ be any one of the points defined in (\ref{pij def}). 

\begin{figure}
   \centering	
	\includegraphics[scale=1.7]{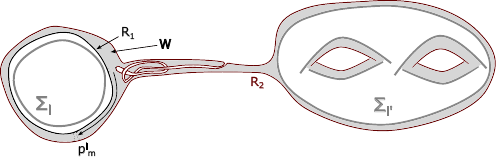}
	\caption{We move the neck connecting $R_1$ and $R_2$ to a small ball around $p^i_m$, where it can be opened.}
	\label{fig:neckfinding}
\end{figure}

Let $\hat{C}$ be the connected component of $\G'_{K+1-l_0}$ given by Lemma \ref{observation} and let $R_1,R_2$ be the corresponding two connected components of $\G'_{K+2-l_0}$ analyzed in  Lemma \ref{r1r2} (see Figure \ref{fig:neckfinding}).
Recall that according to Lemmas \ref{observation} and \ref{r1r2}, $R_2$ bounds a compression body $W$ with  $R_1\subset \partial_-W$, $R_1$ is a 2-sphere and the interior of $W$ avoids $\G'_{K+1-l_0}$. Besides $R_1,R_2$ contain, respectively, subregions $R_1',R_2'$ that are close (in the varifold sense) to the consecutive sheets $S_{i,j_0}$ and $S_{i,j_0+1}$ respectively.
Hence,  $\hat{C}$ intersects the cylinder $C_{\hat{r}/2}(p^i_m)$ transversally after an arbitrarily small perturbation of $\G'_{K+1-l_0}$, and 
\begin{equation}\label{ccr2}
\hat{C}\cap C_{\hat{r}/2}(p^i_m)
\end{equation}
has exactly two connected components 
$\chi_1$ and $\chi_2$ 
which do not bound a disk inside the surface $C_{\hat{r}/2}(p^i_m)$ (the possible other components are small curves bounding disks inside $C_{\hat{r}/2}(p^i_m)$). 
Let
$$Z_1\subset R_1 \text{ and } Z_2\subset R_2$$
denote connected components of $\G \cap D_{\hat{r}/2}(p^i_m)$ with $\chi_1 \subset \partial Z_1$
and $\chi_2 \subset \partial Z_2$.  The region in $D_{\hat{r}/2}(p^i_m)$ between $Z_1$ and $Z_2$ is contained in the compression body $W$.


Let $\eta_4>0$. 
Recall that, by (\ref{inters gammas}) and the construction of $\G'_n$, the surface $\G'_{K+1-l_0}$ intersects one of the surfaces $C_q$ along a closed curve $\gamma$ whose length is arbitrarily small (as long as $\eta$ was chosen small enough). If $\eta>0$ was chosen small enough, by slightly perturbing $\G'_{K+1-l_0}$ and $\G'_{K+2-l_0}$ without increasing their areas by more than $\eta_4$, we can assume that there is a connected embedded curve $\theta:[0,1]\to M$ with endpoints on $\G'_{K+2-l_0}$ such that 
$$\G'_{K+1-l_0} = \big(\G'_{K+2-l_0} \cup \exp_\theta([0,1]\times \eta_4\mathbb{S}^1)\big) \setminus \big(\exp_\theta(\{0\}\times \eta_4\mathbb{D}^2) \cup \exp_\theta(\{1\}\times \eta_4\mathbb{D}^2)\big)$$
where $\exp_\theta$ is the normal exponential map of $\theta$ and $\mathbb{D}^2$ is the Euclidean unit 2-disk whose boundary is the circle $\mathbb{S}^1$. (Strictly speaking, here $\G'_{K+1-l_0} $ may not be smoothly embedded around the ends of the tube $\exp_\theta(\{0\}\times \eta_4\mathbb{D}^2$, but we can smooth out corners along an embedded closed curve as in \cite[Section 3.2]{chambersliok}.)

Let $\chi_1,\chi_2$ be the curve components and $Z_1,Z_2$ be the surface components defined right after (\ref{ccr2}).
Since $W$ is a compression body and $R_1$ is a 2-sphere
we can apply Corollary \ref{arc isotopy} to obtain an
isotopy
$\theta_t$ of the curve $\theta$ inside $W$ 
which moves $\theta$ to an embedded curve $\tilde{\theta}$ satisfying the following properties: 
\begin{itemize}
    \item $\tilde{\theta}$ is vertical in the sense that the projection $P(\tilde{\theta})$ is a point in $\Sigma=\partial M$ ($P:\Omega_1\to \Sigma$ is defined in (\ref{proj P def} );
    \item $\tilde{\theta}\subset D_{\hat{r}/2}(p^i_m)$ and $\tilde{\theta}$ has endpoints in $Z_1\subset R_1$ and $Z_2\subset R_2$;
    \item  there is an embedded curve  $$\tilde{\alpha}_1:[0,1]\to Z_1\cup Z_2\cup \tilde{\theta}$$
with endpoints 
$$\tilde{\alpha}_1(0)\in \chi_1,\quad \tilde{\alpha}_1(1)\in \chi_2$$ 
which is isotopic with fixed endpoints inside $D_{\hat{r}/2}(p^i_m)$ to an embedded curve $\alpha_2:[0,1]\to C_{\hat{r}/2}(p^i_m)$. 
\end{itemize}

\begin{figure}
\label{tubemoving2}
   \centering	
	\includegraphics[scale=0.9]{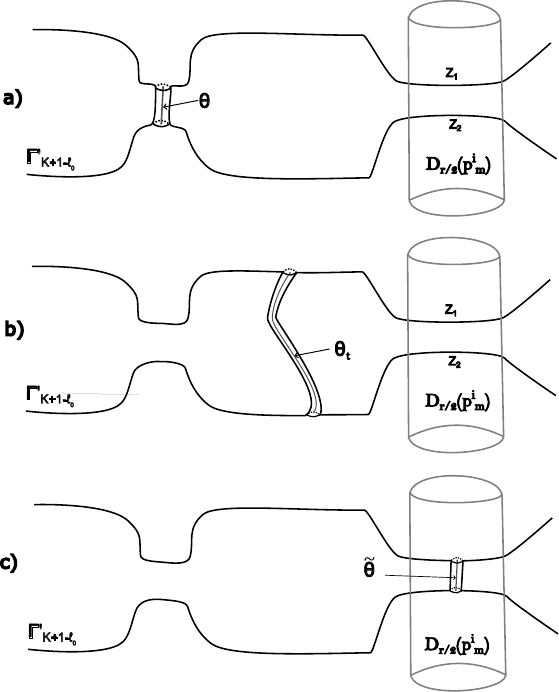}
	\caption{We move the neck inside $W\setminus R_1$ to a neck contained in $D_{\hat{r}/2}(p^i_m)$ which can be opened to decrease the area.}
	\label{fig:sliding}
\end{figure}

Let $\mathbf{n}(x) \in T_xM$ denote a unit normal 
to $\G'_{K+2-l_0}$ at a point
$x \in \G'_{K+2-l_0}$. We make a small perturbation to the homotopy
$\theta_t$, so that 
$|g(\theta_t'(0), \mathbf{n}(\theta_t(0))|> \iota>0$
and $|g(\theta_t'(1), \mathbf{n}(\theta_t(1))|> \iota>0$ for some $\iota>0$. It follows then that for sufficiently small $\eta_4>0$
 the intersection of the neighborhood and the surface $N_{\eta_4}(\theta_t) \cap \G'_{K+2-l_0}$ is a union of two small disks $D_t^0 \ni \theta_t(0)$ and $D_t^1 \ni \theta_t(1)$ in $\G'_{K+2-l_0}$.
Let $A_t$ denote the cylindrical
region in $\partial N_{\eta_4}(\theta_t) $ bounded by curves $\partial D_t^0$ and $\partial D_t^1$.

By a continuous smoothing of  $$\G'_{K+2-l_0}\cup A_t,$$
we get an isotopy from $\G'_{K+1-l_0}$ to an embedded surface $\G''$ which coincides with $\G'_{K+2-l_0}$ except inside $D_{\hat{r}/2}(p^i_m)$. We also have
\begin{equation} \label{G area}
    \mbox{Area}(\G'')\leq \mbox{Area}(\G'_{K+2-l_0})+\eta_5,
\end{equation}
where $\eta_5>0$ can be made arbitrarily small by choosing $\eta_4>0$ small enough.
The effect of this isotopy was to move a neck inside $D_{\hat{r}/2}(p^i_m)$ (see Figures \ref{fig:neckfinding} and \ref{fig:sliding}), and the last step below will be to open it up to substantially decrease the area of the surface.

We observe that $\G''$
satisfies the assumptions of the Corollary \ref{coro:area decrease}. The local area bound
$$\mbox{Area}(\G'' \cap B_{\hat{r}}(p_j^i)) \leq \frac{\delta}{50}$$
follows from \eqref{iv} and \eqref{G area}. 
Curve $\alpha_1$ in the statement of the Corollary is obtained by perturbing
$\tilde{\alpha}_1$ defined above.
Hence, applying Corollary \ref{coro:area decrease} one can decrease the area by a uniform amount by opening such a ``neck'' of $\G''$: there is an isotopy with surgeries $\G''_t$ from $\G''_0:=\G''$ to another surface $\G''_1$ such that 
\begin{enumerate}
\item $\mbox{Area}(\G''_t)\leq \mbox{Area}(\G'')+\delta/2 < \mbox{Area}(\G) + \delta$
\item  $\mbox{Area}(\G''_1)\leq \mbox{Area}(\G'')-c/2$ where $c=c(\G,\Si, \delta)$ is defined in (\ref{def of c}),
\item $$\mbox{Area}(\G''_1\setminus N_h(\Si))<\eta_1.$$
\end{enumerate}

Thanks to Lemma \ref{neck-pinching isotopy}, using this isotopy with surgeries, the isotopy $\{\G'_t\}$ in the beginning of this proof and the neck-pinch surgeries from $\G'_n$ to $\G'_{n+1}$, we can build a smooth isotopy of $\G$ satisfying all the requirements of the proposition. The proof is then finished.

\qed

\vspace{2em}

\section{Local min-max theory} \label{prelim}

\subsection{}\textbf{Definitions from min-max theory} \label{mmdef}

Let $N$ be an oriented connected compact $3$-manifold possibly with boundary, subset of a closed oriented Riemannian 3-manifold $(M_0,g)$. The surfaces considered in this subsection are all embedded. Since $N$ is oriented, we will only consider smooth sweepouts $\{\Sigma_t\}$ where all the slices are oriented. Recall that if $\Sigma$ is an embedded surface, we denote by $|\Sigma|$ the varifold associated to $\Sigma$ with multiplicity one. Let $\mathcal{V}_2(M_0)$ be the closure in the weak topology of the set of $2$-dimensional rectifiable varifolds in $M_0$. The $\mathbf{F}$-norm for varifolds (see \cite{P}) is denoted by $\mathbf{F}$. The Almgren map $\mathcal{A}$ (see \cite{Alm1}) associates to a continuous family of surfaces $\{\Sigma_t\}$ a $3$-dimensional  integral current $\mathcal{A}(\{\Sigma_t\})$.

\begin{defi} \label{defsmoothsweep}
Let $\{\Sigma_t\}_{t\in[a,b]}$ be a family of oriented closed embedded surfaces in $N$. We say that $\{\Sigma_t\}$ is a  \textit{smooth sweepout of $N$} if 
\begin{enumerate}
\item for all $t\in [a,b]$, $\Sigma_t$ is a smooth embedded surface in $N$,
\item $\Sigma_t$ varies smoothly in $t\in (a,b)$,
\item there is a partition $(A,B)$ of the components of $\partial N$ such that, $\Sigma_a=A$, $\Sigma_b=B$, and $|\Sigma_t|$ converges to $|\Sigma_a|$ (resp. $|\Sigma_b|$) in the $\mathbf{F}$-norm as $t\to a$ (resp. $b$),
\item $\mathcal{A}(\{\Sigma_t\})= [|N|]$, where $[|N|]$ denotes the $3$-dimensional integral current induced by  $N$ with its orientation.
\end{enumerate}
\end{defi}




Let $\Pi$ denote a set of smooth sweepouts parametrized by $[0,1]$.
Denote by $\Diff_0$ the set of diffeomorphisms of $N$ isotopic to the identity map and leaving the boundary fixed. The set $\Pi$ is called \emph{saturated} if: for any map $\psi\in C^\infty([0,1]\times N, N)$ such that $\psi(t,.)\in \Diff_0$ for all $t$ and $\psi(0,.)=\psi(1,.)=\Id$, and for any $\{\Sigma_t\}_{[0,1]}\in \Pi$, we have $\{\psi(t,.)(\Sigma_t)\}_{[0,1]}\in \Pi$. 

The \textit{width of $N$} associated with a saturated set $\Pi$ in the sense of Simon--Smith is defined to be 
$$W(N,\Pi) =\inf_{\{\Sigma_t\} \in \Pi}\sup_{t\in[0,1]} \mbox{Area}(\Sigma_t).$$
Given a sequence of smooth sweepouts $\{\{\Sigma^i_t\}\}_i\subset \Pi$, denote by $\mathbf{\Lambda}(\{\{\Sigma^i_t\}\}_i)$ the set 
\begin{align*}
\{V\in \mathcal{V}_2(N) ;\quad & \exists  \{i_j\} \to \infty, t_{i_j}\in [0,1] \\ 
\text{ such that }& \lim_{j\to \infty} \mathbf{F}(|\Sigma^{i_j}_{t_{i_j}}|, V) =0 \}.
\end{align*}
The sequence $\{\Sigma^i_t\} \in\Pi$ is \textit{minimizing} if $\lim_{i\to \infty} \max_{t\in[0,1]}\mbox{Area}(\Sigma^i_t)=W(N,\Pi)$, it is \textit{pulled-tight} if moreover any element $V\in \mathbf{\Lambda}(\{\{\Sigma^i_t\}\}_i)$ with $||V||(N)=W(N,\Pi)$ is stationary. By \cite[Proposition 4.1]{C&DL}, it is always possible to deform a minimizing sequence $\{\Sigma^i_t\} \in\Pi$ into a pulled-tight minimizing sequence.

\subsection{}\textbf{Existence of local min-max surfaces} \label{subsection:existence}


Let $N$ 
 be an oriented connected compact $3$-manifold with boundary, contained in a closed oriented 3-manifold $(M_0,g)$.

 The following theorem will be crucial for proving Pitts--Rubinstein's conjecture in Section \ref{Rubinstein conj}. The point is to show that in the setting of Simon--Smith \cite{Smith}, under certain conditions, a local min-max procedure gives a minimal surface \textit{inside the interior} $Int(N)$ of the domain $N$.

Before stating the theorem, let us introduce the notion of \emph{$H$-compatible splitting}. Let $H$ be a Heegaard splitting of the closed manifold $M_0$.
A triple $(W_0',W'_1,H')$ is called an $H$-compatible splitting if:
\begin{itemize}
\item 
$(W_0',W'_1,H')$ is a generalized Heegaard splitting of
$N$, so that in particular $\partial N = \partial_-W'_0 \cup \partial_-W'_1$ (here, $\partial_-W'_0$ and $\partial_-W'_1$ are  orientable embedded surfaces, not necessarily connected and possibly empty);
\item  $H'$ is isotopic to $H$ in $M$; 
\item $\partial_-W'_0$ and $\partial_-W'_1$ are stable minimal surfaces.
\end{itemize}

We are now ready to state the main result of this section:


\begin{thm} [Local Min-max]\label{smoothminmax}
Let $N$ 
 be an oriented connected compact $3$-manifold contained in a closed oriented 3-manifold $(M_0,g)$.
 Suppose that $\partial N$ is a strictly stable minimal surface in $(M_0,g)$. Let $\partial N = \Gamma_0 \cup \Gamma_1$ be a partition of the boundary components of $\partial N$ and let $H$ be a strongly irreducible Heegaard splitting of $M_0$.

Let $\Pi$ be the nonempty saturated set generated by all smooth sweepouts $\{\Sigma_t\}$ of $N$ such that 
\begin{itemize}
\item for $i\in\{0,1\}$, $\Sigma_i = \Gamma_i $,
\item for all $t\in(0,1)$, the surface $\Sigma_t$ is such that
there is  some $H$-compatible splitting $(W_0',W'_1,\Sigma_t)$ with $\partial_-W_i'=\Gamma_i$ for $i\in \{0,1\}$.
\end{itemize}

Then there exists a min-max sequence $\Sigma^j_{t_j}$ converging to $\sum_{i=1}^k m_i\Sigma^\infty_i$ as varifolds, where $m_i \in \N$ and $\Sigma^\infty_i\subset N$ are disjoint closed embedded connected minimal surfaces such that
$$ \sum_{i=1}^k m_i \mbox{Area}(\Sigma^\infty_i) = W(N,\Pi),$$
the index of $\bigcup_{i=1}^k \Sigma^\infty_i$ is at most $1$,
and there is a least one $i_0\in\{1,...,k\}$ such that
$$\Sigma^\infty_{i_0}  \subset  \interior(N).$$

\end{thm}
\begin{remark}\label{remark 6.3}
In the statement, the minimal surface $\bigcup_{i=1}^k \Sigma_i^\infty$ is obtained by Simon-Smith min-max theory. Hence, by Theorem \ref{simonsmith}, the multiplicity $m_i$ is even when $\Sigma_i^\infty$ is unorientable. 
\end{remark}

\begin{remark}
Condition $(iii)$ that $\G_0$ and $\G_1$ are strictly stable minimal surfaces can be replaced by 
a weaker condition that $\partial N$ is minimal and there exists a squeezing map $P_t$ 
from Section
\ref{sec:squeezing} in a small neighborhood of $\partial N$. 
For example, 
if $D$ is a 3-ball with smooth minimal stable (but not strictly stable) boundary and
there exists a foliation of the tubular neighborhood
of $\partial D$ by strictly mean concave spheres
(mean curvature points towards $\partial N$),
then a squeezing map $P_t$ can be constructed 
and Theorem \ref{smoothminmax} guarantees existence of 
 an embedded minimal 2-sphere in the interior of $D$. 
\end{remark}

\begin{proof}       
Since by Condition $(iii)$,
$\Gamma_0\cup\Gamma_1 = \partial N$ is a strictly stable minimal surface, we have
\begin{equation} \label{boundtrue}
\max \{\mbox{Area}(\Gamma_0),\mbox{Area}(\Gamma_1)\}< W(N,\Pi).
\end{equation}
In a somewhat different setting this was proved by F. Morgan and A. Ros in \cite{MorganRos}. We give another proof of (\ref{boundtrue}) in the Appendix.

By the maximum principle or using the squeezing map $P_t$ 
from Section
\ref{sec:squeezing}, we can also find a small $\bar{\delta}>0$ so that 
$$N_{\bar{\delta}} := N \cup \{x\in M_0; d(x,\partial N)\leq \bar{\delta}\} $$
is a strictly mean convex domain and if a closed minimal surface is contained in $N_{\bar{\delta}}$ then it is contained in $N$. The saturated set $\Pi$ naturally induces a saturated set $\Pi_{\bar{\delta}}$ associated with $N_{\bar{\delta}}$. It is then not difficult to check that for $\bar{\delta}$ small, $W(N_{\bar{\delta}},\Pi_{\bar{\delta}})=W(N,\Pi)$. If $\bar{\delta}$ is chosen small enough, by (\ref{boundtrue}), we can apply the version of the Simon--Smith theorem proved in \cite[Theorem 2.1]{MaNe} to get the existence of the varifold $V=\sum_{i=1}^k m_i\Sigma^\infty_i$, then the genus bound and the nature of convergence follow from \cite{Ketgenusbound}. The index of the union $\bigcup_{i=1}^k\Sigma^\infty_i$ is bounded by one according to Theorem 6.1 and paragraph 1.3 in \cite{MaNeindexbound}.

The goal is to show the existence of a component $\Sigma_{i_0}^\infty$ inside the interior $\interior(N)$. The arguments will share some similarities with \cite[Deformation Theorem C]{MaNeindexbound} (in particular  several constructions in its proof will be useful), however we have to deal with \emph{smooth isotopies}. On the other hand, here we only need to rule out the case where the whole min-max minimal surface is included in the boundary. 
Let $\{\{\Sigma^i_t\}\}_i \subset \Pi$ be a pulled-tight minimizing sequence and suppose by contradiction that for all $V\in \mathbf{\Lambda}(\{\{\Sigma^i_t\}\}_i)$ with smooth support and mass $W(N,\Pi)$, $\spt(V)$ is included in $\partial N$. 
Given a sweepout in $\Pi$, we orient $\Sigma_t$ with the unit normal $\nu$ pointing towards $\Gamma_1$. In $N$, each $\Sigma_t$ hence bounds a manifold with boundary $B(\Sigma_t)$ such that $\nu$ is the outward normal.

Denote by $S_1$, ..., $S_p$ the connected components of $\partial N = \Gamma_0\cup \Gamma_1$. Let $V$ be a varifold with mass $W(N,\Pi)$, of the form:
\begin{equation} \label{form}
V = m_1|S_1|+...+m_p |S_p|,
\end{equation}
where $m_i$ are nonnegative integers. We call the finite set of such $V$ by $\hat{\mathcal{V}} $. Let $\{\Sigma_t\}_{t\in[0,1]}$ be a smooth sweepout in $\Pi$. We are applying the following discussion to the pulled-tight minimizing sequence $\{\{\Sigma^i_t\}\}_i$ so we are assuming $\{\Sigma_t\}$ to be one of these sweepouts. We can in particular make $\max_t \mbox{Area}(\Sigma_t)- W(N,\Pi)$ arbitrarily small. Given $\alpha>0$, consider the subset 
$$\mathbf{V}_{\alpha}:=\{t\in[0,1] ; \quad \exists V \in \hat{\mathcal{V}} , \quad\mathbf{F}(|\Sigma_t|, V) \leq \alpha\}.$$

The rough idea of the remainder of the proof is to construct from \emph{part} of the sweepout $\{\Sigma_t\}$ another sweepout $\{\hat{\Sigma}_t\}$ for which the corresponding $\mathbf{V}_{\eta}$ is empty, where $\eta>0$ is a constant depending on $\alpha$ and $N$. A technical point which is already in Claims 1-4 of the proof of \cite[Deformation Theorem C]{MaNeindexbound} is that we will make sure that the surfaces $\hat{\Sigma}_t$ are not close to any stationary integral varifold which was far from the surfaces $\Sigma_t$. Without loss of generality, we suppose that $\mathbf{V}_{\alpha}$ is non-empty 
and is a finite union of closed intervals. If $\alpha$ is sufficiently small, then for any $t\in \mathbf{V}_{\alpha}$, $\Sigma_t$ bounds $B(\Sigma_t)$ which has volume either close to $0$ or close to $\Vol(N)$. Let $[a_1,b_1]$,...,$[a_q,b_q]$ be the intervals in $\mathbf{V}_{\alpha}$ such that $B(\Sigma_t)$ has volume close to $0$, where $a_1\leq b_1<...<a_q \leq b_q$. We have $b_q<1$ since $\alpha$ is small enough. 
Let $[c_1,d_1]$,...,$[c_{q'},d_{q'}]$ be the intervals composing $\mathbf{V}_{\alpha} \cap (b_q,1]$. On these intervals the volume of $B(\Sigma_t)$ is close to $\Vol(N)$. Then $\{\hat{\Sigma}_t\}$ will be constructed from the restriction $\{\Sigma_t\}_{t\in [b_q,c_1]}$ by appropriately closing its ends, i.e. by deforming $\Sigma_{b_q}$ and $\Sigma_{c_1}$ respectively to $\Gamma_0$ or $\Gamma_1$, thanks to the interpolation result Theorem 
\ref{maininterpolation}. This new sweepout $\{\hat{\Sigma}_t\}$ will not necessarily be homotopic to $\{\Sigma_t\}$ (contrarily to the analoguous situation in \cite[Deformation Theorem C] {MaNeindexbound}) but it will be a smooth sweepout and thus in the family $\Pi$. In the sequel we will only explain the case where both intervals of the form $[a_i,b_i]$ and $[c_i,d_i]$ exist. When there is no interval of the form $[a_i,b_i]$, but there are intervals of the form $[c_i,d_i]$ we will only need to deform $\Sigma_{c_1}$ to $\Gamma_1$ (and similarly if there is no interval of the form $[c_i,d_i]$, but there are intervals of the form $[a_i,b_i]$). If $\mathbf{V}_{\alpha}$ is empty then we do not modify the sweepout $\{\Sigma_t\}$.



Let $N_h(\partial N)$ be an $h$-tubular neighborhood of $\partial N = \Gamma_0\cup \Gamma_1$, where $h>0$ is smaller than $h_0$ given by Theorem \ref{maininterpolation}. Let $\alpha_1>0$ be such that if for $V\in \hat{\mathcal{V}}$, a stationary integral varifold $Z$ has $\spt(Z)\subset \spt(V)$, $\mathbf{M}(Z)=\mathbf{M}(V)$ but $Z\neq V$, then there is a connected component $\Omega^Z$ of $N_h(\partial N)$ so that $||Z||(\Omega^Z) > ||V||(\Omega^Z) +\alpha_1$. The $\alpha$ considered in the previous paragraph will then be taken independent of $\{\Sigma_t\}$ and at least small enough so that the following holds: if for $V\in \hat{\mathcal{V}}$, a varifold $Z'\in \mathcal{V}_2(M_0)$ (not necessarily stationary or integral) has $||Z'||(\Omega^{Z'}) > ||V||(\Omega^{Z'}) +\alpha_1$ for a connected component $\Omega^{Z'}$ of $N_h(\partial N)$, then $\mathbf{F}(Z',V)>2\alpha$ (see \cite[Subsections 5.11 and 5.14]{MaNeindexbound}).

If $\alpha$ was chosen small enough then $\Sigma_{b_q}$ satisfies the following: 
$$\mbox{Area}(\Sigma_{b_q}\backslash N_h(\partial{N})) < \epsilon_0/2$$
for the $\epsilon_0$ given by Theorem \ref{maininterpolation}. 
We can now apply the interpolation result Theorem \ref{maininterpolation} to $ \Sigma_{b_q}$ since by assumption 
there is  some $H$-compatible splitting $(W_0',W'_1,\Sigma_{b_q})$ with $\partial_-W_i'=\Gamma_i$ for $i\in \{0,1\}$.
We can thus deform $ \Sigma_{b_q}$ into an embedded surface close to $\Gamma_0$ in the $\mathbf{F}$ topology, with area control. Let $\{Y_t\}_{t\in [0,1]}$ be this continuous family of surfaces interpolating between $\Sigma_{b_q}$ and $Y_1$ close to $\Gamma_0$. On can deform $Y_1$ to $\Gamma_0$ continously in the $\mathbf{F}$-topology. By arguing as in Claims 1-4 of \cite{MaNeindexbound}, one can choose $\delta>0$ small enough so that for a positive $\eta<\alpha/2$ independent of the index $i$ of $\{\Sigma^i_t\}$ and $\delta$ when they are respectively large and small enough,  
$$
\forall t\in[0,1],\quad\mathbf{F}(|Y_t|,V') > \eta
$$ 
for all stationary integral varifold $V'$ with mass $W(N,\Pi)$. To see this remember that in the proof of Proposition \ref{stacking2}, we first applied Lemma \ref{contracting_hair}, then we could use the squeezing maps $P_t$ to push the surface arbitrarily close to $\partial N$, and finally we applied the stacking deformations. Along the deformations of Propositions \ref{stacking2} and \ref{movehandles}, we did not increase the area of the deformed surface by more than $\delta$ in the $h$-tubular neighborhood of any components of $\partial N$. Compared to \cite{MaNeindexbound}, in their notations $H^1_i$ corresponds to our deformation in Lemma \ref{contracting_hair2}, $H^2_i$ corresponds to pushing the surface with the squeezing maps $P_t$, and $H^3_i$ corresponds to our deformations in Propositions \ref{stacking2} and \ref{movehandles}. The analogues of Claims 1-4 in \cite{MaNeindexbound} hold true in our situation (the choice of $\alpha_1$ and $\alpha$ come into play here).

The above procedure can be realized in a symmetric way for $t=c_1$ (remember that $[c_1,d_1]$,...,$[c_{q'},d_{q'}]$ are the intervals composing $\mathbf{V}_{\alpha} \cap (b_q,1]$): we can deform $\Sigma_{c_1}$ to $\Gamma_1$ with similar properties on the interpolation. As a result of concatenating and reparametrizing the successive deformations from $\Gamma_0$ to $Y_0=\Sigma_{b_q}$, from $\Sigma_{b_q}$ to $\Sigma_{c_1}$ (this part is given by $\{\Sigma_t\}_{t\in[0,1]}$), and from $\Sigma_{c_1}$ to $\Gamma_1$, we obtain a family denoted by $\{\hat{\Sigma}_t\}_{t\in[0,1]}$, satisfying for an arbitrary $\delta>0$ and for some $\eta<\alpha/2$ the following properties:
\begin{itemize}
\item $\{\hat{\Sigma}_t\} \in \Pi$,
\item $\max_t \mbox{Area}(\hat{\Sigma}_t) \leq \max_t \mbox{Area}({\Sigma}_t) + \delta$,
\item $\mathbf{F}(|\hat{\Sigma}_t|,V) > \eta$ for all $t\in[0,1]$ and all $V$ of the form (\ref{form}) with $\mathbf{M}(V)=W(N,\Pi)$,
\item $\mathbf{F}(|\hat{\Sigma}_t|,V') > \eta$ for all $t\in[0,1]$ and all stationary integral varifold $V'$ with $\mathbf{M}(V')=W(N,\Pi)$ which are not in $\mathbf{\Lambda}(\{\{{\Sigma}^i_t\}\}_i)$.
\end{itemize} 
We emphasize that this sweepout $\{\hat{\Sigma}_t\}$ is not a priori homotopic to the original sweepout $\{\Sigma\}_t$. Moreover the $\eta$ is independent of the index $i$ of $\{\Sigma^i_t\}_{t\in [0,1]}$ and of $\delta$.

Now we can conclude the proof. Let $\{\{\Sigma^i_t\}\}_i$ be a pulled-tight minimizing sequence and $\delta_i$ a sequence going to $0$. Transforming each sweepout of this sequence as above with parameter $\delta_i$ instead of $\delta$ (but keeping $\alpha>0$, $\eta$ fixed) produces a new pulled-tight minimizing sequence $\{\{\hat{\Sigma}^i_t\}\}_i\subset \Pi$. Note similarly to \cite[Deformation Theorem C]{MaNeindexbound} that by construction:
\begin{align} \label{inclusi}
\begin{split}
 & \mathbf{\Lambda}(\{\{\hat{\Sigma}^i_t\}\}_i) 
\cap \{\text{stationary integral varifolds of mass $W(N,\Pi)$}\} \\
\subset  \quad&  \mathbf{\Lambda}(\{\{{\Sigma}^i_t\}\}_i) 
\cap \{\text{stationary integral varifolds of mass $W(N,\Pi)$}\}.
\end{split}
\end{align} 
But by construction of $\{\{\hat{\Sigma}^i_t\}\}_i$, any varifold $V' \in \mathbf{\Lambda}(\{\{\hat{\Sigma}^i_t\}\}_i)$ with $||V'||(N)=W(N,\Pi)$ is $\eta$-far from any varifold of the form (\ref{form}). So the usual min-max theorem (see \cite{C&DL}) would produce a varifold in the first intersection in (\ref{inclusi}) with smooth support, i.e. a minimal surface whose area counted with multiplicity is $W(N,\Pi)$, and which is not entirely contained in the boundary $\partial N$. This contradicts our assumption on $\mathbf{\Lambda}(\{\{{\Sigma}^i_t\}\}_i)$ so the theorem is proved.

\end{proof}

\vspace{2em}

\section{Main existence theorem for minimal Heegaard splittings} \label{Rubinstein conj}



\subsection{}\label{corecore}\textbf{$H$-compatible splittings and $H$-cores}
\label{subsection:hcore}

Let $(M,g)$ be a connected closed oriented $3$-manifold.
In this subsection, we study  compact $3$-manifolds $N\subset M$ containing a Heegaard splitting of $M$, and whose boundary $\partial N$ 
is a stable minimal surface in $(M,g)$.
Roughly, the goal of this subsection is to find, under certain conditions, an innermost compact manifold $N$ with such properties.

Let $(W_0,W_1,H)$ be a Heegaard splitting of $M$. Recall that the notion of $H$-compatible splitting was introduced in Subsection \ref{subsection:existence}.
Note that $(W_0,W_1,H)$ is itself an $H$-compatible splitting. 

Set 
\begin{align*}
\mathcal{S}_g(H) := \{&\text{ 3-manifolds $N\subset M$ such that $N=W_0'\cup W'_1$}\\ 
&\text{for a certain $H$-compatible splitting $(W_0',W'_1,H')$}\}
\end{align*}
The inclusion relation $\subset$ gives a partial order on $\mathcal{S}_g(H)$. An element $N$ is then \emph{minimal} in the sense of this partial order (not to be confused with the notion of minimal surface) if there is no element $\tilde{N} \in \mathcal{S}_g(H)$ such that $\tilde{N}\subset  N$ and $\tilde{N}\neq N$. Besides, elements of $\mathcal{S}_g(H)$ are connected.

\begin{lemma} \label{max spheres} 
Let $(M,g)$ be a closed oriented 3-manifold.
Suppose that $H$ is a strongly irreducible Heegaard splitting of $M$. Then 
\begin{enumerate}
\item either $H$ is isotopic to a stable minimal surface, \label{one}
\item or $H$ is isotopic to the boundary of a tubular neighborhood of a non-orientable minimal surface $\Sigma \subset M$ with a vertical handle attached, and $\Sigma$ has stable oriented double cover, \label{two}
\item or for any $N\in\mathcal{S}_g(H)$, there exists a minimal element $N_{min}\in\mathcal{S}_g(H)$ such that 
$$N_{min}\subset  N.$$
\end{enumerate}

\end{lemma}

\begin{proof}

Let $N=N_0\in\mathcal{S}_g(H)$, (which could be $N=M$).  Consider a sequence $\{N_i\}\subset \mathcal{S}_g(H)$ such that 
$$...\subset N_2 \subset N_1 \subset N_0 \text{ and}$$ 
\begin{equation*} \label{limvol}
\Vol(N_{i+1}) \leq  \inf\{\Vol(N') : N' \subset N_i, N' \in \mathcal{S}_g(H)\}+\frac{1}{i+1}. 
\end{equation*}
{We can assume without loss of generality that the sequence does not terminate.}  We want to show that $\partial N_i$ converges subsequentially to a minimal surface. By stability of the minimal surfaces $\partial N_i$, \cite[Theorem 3]{Sc} gives a uniform upper bound $K$ on the norm of the second fundamental form of $\partial N_i$ (i.e. a ``curvature bound'').

{
Consider the normal exponential map
$$\partial N_i \times (-\delta_0,\delta_0)\to M$$
$$\exp_i^N:(x,s) \mapsto \exp_x(s\nu)$$
where $\nu$ is the choice of unit normal vector on $\partial N_i$ pointing inside $N_i$, and $\delta_0>0$ is a small positive number such that  $\exp_i^N$ is locally a diffeomorphism.
The uniform bound $K$ on the second fundamental form implies a uniform positive lower bound on $\delta_0$ depending only $K$ and $(M,g)$.
}

{
The strategy of the proof is to try to take a subsequential limit of the minimal surfaces $\partial N_i$. We will see that when $\lim_{i\to \infty}\Vol(N_i) = 0$, then either item (1) occurs (for instance $\partial N_i$ is the union of two connected components isotopic to $H$ and very close to each other)  or item (2) occurs (for instance $\partial N_i$ is an oriented surface converging with multiplicity two to a non-orientable surface). Otherwise, when  $\lim_{i\to \infty}\Vol(N_i) > 0$, we will see that $N_i$ converges to a desired minimal element $N_{min}\in\mathcal{S}_g(H)$.}

{
To prove that $\partial N_i$ converges smoothly subsequentially, it suffices to show that their areas are uniformly bounded from above. Indeed, the  standard smooth compactness \cite[Theorem 3]{Sc} \cite[Proposition 7.14]{CM} for minimal surfaces with bounded second fundamental form and bounded area can then be employed.  
Note that in our case, the sequence of stable minimal surfaces is nested, i.e. for each $\partial N_i$, all $\partial N_j$ are on one side of $\partial N_i$ for $j>i$. 
}

{
By compactness, the sequence of sets $N_j$ converges in the Hausdorff metric to some limit closed set.
For $\mu >0$, let $i_\mu$ be such that for all $j\geq i_\mu$ and all $q \in \partial N_j$.
\begin{equation} \label{dist to Nimu}
    d(\partial N_{i_\mu}, q) < \mu.
\end{equation}
Fix $q \in \partial N_j$ and let $p$ be the closest point 
of $\partial N_{i_\mu}$ to $q$.
Note that we can choose $\mu$ sufficiently small compared to the curvature upper bound $K$ and the focal distance lower bound $\delta_0$, so that  
the following holds: there is an intrinsic embedded geodesic ball $B_{\partial N_{i_\mu}}(p,r)$ in $\partial N_{i_\mu}$ centered at $p$, with radius $r>0$ depending only on $K$ and $(M,g)$, such that around $q$, $\partial N_j$ is a graph over $B_{\partial N_{i_\mu}}(p,r)$ via the normal exponential map. In other words, there is a smooth positive function 
$$f:B_{\partial N_{i_\mu}}(p,r)\to (0,\infty)$$
such that for all $x\in B_{\partial N_{i_\mu}}(p,r)$,
$$\exp_{i_\mu}^N(x,f(x)) \subset \partial N_j,$$
and 
$$ \exp_{i_\mu}^N(p,f(p)) = q.$$
If $\mu$ is chosen small enough with respect to $K$,  $\delta_0$ and $(M,g)$, by a continuity argument and using that $\partial N_j \cap \partial N_{i_\mu} = \varnothing$, we can realize larger and larger regions of $\partial N_j$ as local graphs over regions of $N_{i_\mu}$. By compactness, we eventually obtain connected components:
 $$C_{{i_\mu},1}\subset \partial N_{i_\mu},$$
 $$C_{j,1}\subset \partial N_j$$ 
so that $C_{j,1}$ is locally a graph over $C_{{i_\mu},1}$. Since all the surfaces $\partial N_i$ are closed and orientable, the only possibility is that there is a smooth positive function 
$$f_1: C_{{i_\mu},1} \to (0,\delta_0)$$
such that 
$$x\mapsto\exp_{i_\mu}^N(x,f(x))$$
yields a diffeomorphism between $C_{{i_\mu},1}$ and $C_{j,1}$. By the curvature bound $K$, necessarily the function $f_1$ has uniformly bounded $C^1$-norm. By \eqref{dist to Nimu}, the same argument applies to each connected component of $\partial N_j$.
}


{
We conclude that each connected component of $\partial N_{j_k}$ is a graph over some component of $\partial N_{i_\mu}$ of some function with uniformly bounded $C^1$-norm. Since each  $N_{j_k}$ is connected, the number of connected components of $\partial N_{j_k}$ is at most twice the number of components of $\partial N_{i_\mu}$. Combined with the uniform bound on the $C^1$-norm, 
$$\mbox{Area}(\partial N_{j_k}) \leq \kappa.\mbox{Area}(\partial N_{i_\mu})$$
for all $k\geq0$, for some $\kappa>0$ depending only on $K$, $\mu$, $(M,g)$. This shows the claim that $\partial N_j$ have uniformly bounded area.
}

{
We can now apply the standard compactness results \cite[Proposition 7.14]{CM}. The sequence $\partial N_i$ subsequentially converges smoothly to an embedded minimal surface $\Sigma$. 
By definition of $\mathcal{S}_g(H)$, each $N_i$ is connected and contains a (connected) Heegaard splitting of $M$.
When the limit of $\Vol(N_{i})$ is zero then by definition of smooth convergence and connectedness of $N_i$, $\partial N_i$ is the graph of a 2-valued function over $\Sigma$. The convergence is said to occur with multiplicity two.
Either $\Sigma$ is an oriented stable minimal surface isotopic to $H$ or $\Sigma$ is a connected non-orientable minimal surface bounding a compression body on one side. In the non-orientable case, the oriented double cover is stable because the convergence of $\partial N_i$ to $\Sigma$ occurs with multiplicity two, and so a standard argument produces a nontrivial positive Jacobi field over the double cover, see for instance \cite[Proof of Theorem 2.3]{Sharp}. By strong irreducibility and Proposition \ref{strongirreduce}(i), the Heegaard splitting $H$ is isotopic to this double cover with a vertical handle attached. 
When $\lim_{i\to \infty}\Vol(N_{i})$ is not zero, $\Sigma$ bounds the desired minimal element $N_{min}$ of $\mathcal{S}_g(H)$.
}

\end{proof}

It will be convenient to introduce this definition:

\begin{defi} \label{coredef}
Let $M$ be as above and $H$ be a Heegaard splitting. An \textit{$H$-core} of $M$ is a minimal element $\mathfrak{C}$ of $\mathcal{S}_g(H)$. 

Let $C$ be a positive real number. A \textit{$C$-bounded $H$-core} of $M$ is an element $\mathfrak{C}$ of $\mathcal{S}_g(H)$ such that $\mbox{Area}(\partial \mathfrak{C})\leq C$, and which is minimal among elements $N$ of $\mathcal{S}_g(H)$ satisfying $\mbox{Area}(\partial N)\leq C$.

\end{defi}

We can reformulate Lemma \ref{max spheres} as follows:

\begin{coro} \label{existence core}

{Let $(M,g)$ be a closed oriented 3-manifold} and let $H$ be a strongly irreducible Heegaard splitting. 

\begin{enumerate}
\item Suppose that $H$ is 
\begin{itemize}
\item neither isotopic to a stable minimal surface 
\item nor isotopic to the boundary of a tubular neighborhood of a non-orientable minimal surface $\Sigma \subset M$ with a vertical handle attached, where $\Si$ has stable oriented double cover.
\end{itemize}
Then, for any $N\in\mathcal{S}_g(H)$, there exists an $H$-core $\mathfrak{C}$ of $M$ such that
$$\mathfrak{C} \subset N.$$ 

\item Fix $C>0$. Suppose that $H$ is 
\begin{itemize}
\item neither isotopic to a stable minimal surface of area at most $C$
\item nor isotopic to the boundary of a tubular neighborhood of a non-orientable minimal surface $\Sigma \subset M$ with a vertical handle attached, where $\Si$ has stable oriented double cover and area at most $C$.
\end{itemize}
Then, for any $N\in\mathcal{S}_g(H)$, there exists a $C$-bounded $H$-core $\mathfrak{C}$ of $M$ such that
$$\mathfrak{C} \subset N.$$
\end{enumerate}
\end{coro}

\begin{proof}

The first part immediately follows by Lemma \ref{max spheres}.
The second part is easier to prove since the stable minimal surfaces in consideration have area bounded by $C$ and so it follows from the curvature bound of \cite[Theorem 3]{Sc} and the usual smooth compactness result for embedding minimal surfaces with bounded second fundamental form and area \cite[Proposition 7.14]{CM}.
\end{proof}

\subsection{}\textbf{Main theorem}

We finally prove our main theorem, Theorem \ref{introthm}. It will readily follow from Theorem \ref{positivegenus} and Theorem \ref{sphere case} below.

A Riemannian metric $g$ is said to be bumpy if no smooth immersed closed minimal hypersurface has a non-trivial Jacobi vector field. White showed that bumpy metrics are generic in the Baire sense \cite{Whitebumpy,Whitebumpy2}.

We first start with the main case where $M$ is not diffeomorphic to a 3-sphere.
\begin{thm} \label{positivegenus}
Let $(M,g)$ be a closed oriented $3$-manifold not diffeomorphic to the $3$-sphere. Suppose that there is a strongly irreducible Heegaard splitting $H$. Then 
\begin{enumerate}[label=\roman*)]
\item either $H$ is isotopic to a minimal surface $\Sigma_1$ of index at most one, 
\item or isotopic to the boundary of a tubular neighborhood of a non-orientable minimal surface $\Sigma_2$ with a vertical handle attached, and $\Si_2$ has stable oriented double cover.
\end{enumerate}

If moreover the metric is bumpy, then there is an oriented minimal surface of index 1 isotopic to $\Sigma_1$ in the first case, and isotopic to the boundary of a tubular neighborhood of $\Sigma_2$ in the second case. In particular its genus is equal to either $genus(H)$ or $genus(H)-1$.

\end{thm}

Before proving Theorem \ref{positivegenus}, we need two lemmas. 
$H$-cores and $C$-bounded $H$-cores in $(M,g)$ are defined in Definition \ref{coredef}, and exist whenever $H$ is not isotopic to a stable minimal surface or to the boundary of a tubular neighborhood of a non-orientable minimal surface $S$ with a vertical handle attached, where $S$ has stable oriented double cover, by Corollary \ref{existence core}. 

The next elementary lemma ensures that $H$-core in $(M,g)$ can be approximated by $C$-bounded $H$-cores in $(M,g_m)$ where $g_m$ are bumpy metrics approximating $g$.

\begin{lemma} \label{approximation}

Let $(M,g)$ and $H$ be as in Theorem \ref{positivegenus}.
Suppose that $H$ is not isotopic to a stable minimal surface or to the boundary of a tubular neighborhood of a non-orientable minimal surface $S$ with a vertical handle attached, where $S$ has stable oriented double cover. Let $\mathfrak{C}$ be an $H$-core in $(M,g)$. Then for all constants $C>0$ large enough, there is a sequence of bumpy metrics $g_m $ converging smoothly to $g$ and $C$-bounded $H$-cores $ \mathfrak{C}_{m}$ with respect to $g_m$, such that $\partial \mathfrak{C}_{m}$ converges smoothly to $\partial \mathfrak{C}$ with respect to $g$.
\end{lemma}

\begin{proof}
By \cite[Propositions 2.3]{IMN} we can choose a sequence of positive functions $\tilde{\lambda}_m$ converging smoothly to $1$ so that the minimal surface $\partial \mathfrak{C}$ is strictly stable for $\tilde{\lambda}_m g$. 
A small tubular neighborhood $V_m$ of $\partial \mathfrak{C}$ then has boundary $\partial V_m$ which is strictly mean convex with respect to $\tilde{\lambda}_m g$. 
Then using the genericity of bumpy metrics proved in \cite{Whitebumpy}, we modify slightly the metric $\tilde{\lambda}_m g$ into $g_m$ so that any stable minimal surface of $(M,g_m)$ is strictly stable 
and $\partial V_m$ is still strictly mean convex. 
By minimizing the $g_m$-area of the surface $\partial \mathfrak{C}$ inside $V_m$ and applying Meeks--Simon--Yau \cite{MSY}, we find a strictly stable embedded minimal surface $S_m$ for $g_m$ inside $V_m$, so that $S_m = \partial A_m$ for some compact 3-manifold $A_m\subset M$,
$S_m$ converges smoothly to $\partial \mathfrak{C}$ as $m\to \infty$ and $A_m$ converges in the Hausdorff distance to the $H$-core  $\mathfrak{C}$. Note that for $m$ large enough, $A_m\in \mathcal{S}_{g_m}(H)$ (see Subsection \ref{subsection:hcore} for the notation $\mathcal{S}_{g}(H)$).

Let $C>0$ be a constant larger than twice the area of $\partial \mathfrak{C} \subset (M,g)$. For all $m$ sufficiently large, the assumptions of the second part of Corollary \ref{existence core} have to be satisfied for $H$ and $g_m$, otherwise we could take a smoothly converging subsequence of minimal surfaces by the curvature bound satisfied by stable minimal surfaces \cite[Theorem 3]{Sc}, and their limit would contradict our assumptions on $g$. Thus, the second part of Corollary \ref{existence core} gives the existence of a $C$-bounded $H$-core $\mathfrak{C}_m$ in $(M,g_m)$ with
$$\mathfrak{C}_m\subset A_m.$$
As $m\to \infty$, the boundary $\partial \mathfrak{C}_m$ converges smoothly by \cite[Theorem 3]{Sc} and the limit is $\partial \mathfrak{C}$, by minimality of the $H$-core $\mathfrak{C}$ in  $\mathcal{S}_{g}(H)$. For clarity, note that the volume of $\mathfrak{C}_m$ remains bounded away from $0$ as $m\to \infty$, since $H$ is not isotopic to a stable minimal surface or to the boundary of a tubular neighborhood of a non-orientable minimal surface $S$ with a vertical handle attached, where $S$ has stable oriented double cover.

\end{proof}

The next technical lemma will also be useful in the proof of Theorem \ref{positivegenus}. It enables us to sweepout certain compression bodies by a family of surfaces with controlled area whenever there are no ``obvious'' obstructions to do so. The notion of smooth sweepout is defined in Definition \ref{defsmoothsweep}.

\begin{lemma} \label{constructsweep}  
Suppose that $(N,g)$ is a compression body endowed with $g$ a bumpy metric, such that $\partial_+N = \Gamma$ is mean-convex (the mean curvature vector is pointing outwards) and $\partial_- N= \Gamma'$ is an embedded  stable minimal surface (which is not necessarily connected).
Then one of the following cases occurs:
\begin{enumerate}

\item either there is a compression body $\hat{N}$ strictly included in $N$ such that $\partial_+\hat{N} = \Gamma$ and $\hat{\Gamma}':=\partial_- \hat{N}$ is an embedded  stable minimal surface (which is not necessarily connected),

\item or there is a smooth sweepout $\{\hat{\Sigma}_t\}_{t\in[0,1]}$ of $N$ such that
$\hat{\Sigma}_0 = \Gamma$, $\hat{\Sigma}_1=\Gamma'$, $\hat{\Sigma}_t$ is isotopic to $\Gamma$ for $t\in[0,1)$, and moreover
$$\max_{t\in[0,1]} \mbox{Area}(\hat{\Sigma}_t) =\mbox{Area}(\Gamma).$$

\end{enumerate}

\end{lemma}

\begin{proof}

Let us minimize the area of an embedded surface $\Gamma_0$ close to $\Gamma$ (and with smaller area) inside $N$ using the $\gamma$-reduction of \cite{MSY} with the constraint that the ``area is less than $\mbox{Area}(\Gamma)$'' as in the proof of Lemma \ref{contracting_hair} (see \cite[Section 7]{C&DL}).
The boundary $\partial N$ acts as a barrier, and we get a stable minimal surface $S\subset N$ in the limit. This procedure yields a smooth homotopy $\{\Gamma_t\}_{t\in [0,1)}$ such that as $t\to 1$, $\Gamma_t$ is Hausdorff close to $S$ union some curves. The constraint on the area means that 
\begin{equation}\label{noot}
\max_{t\in [0,1)} \mbox{Area}(\Gamma_t) <\mbox{Area}(\Gamma).\end{equation}

Suppose first that $S$ is different from $\Gamma'$. Then by the definition of $\gamma$-reduction \cite{MSY} and the topological description of surgeries on strongly irreducible Heegaard splittings, Lemma \ref{sameside}, we deduce that $\Gamma$ and $S$ bound a compression body that we call $\hat{N}$, which is strictly contained in $N$. 
We set $\hat{\Gamma}'=S$. Item (1) holds in this case.

Suppose next that $S = \Gamma'$.
By the interpolation theorem Theorem \ref{maininterpolation}, and the definition of $\gamma$-reduction \cite{MSY}, as $t\to 1$, $\Gamma_t$ is close to $S$ union some curves ``without multiplicity'', meaning that after some surgeries on thin necks of $\Gamma_t$, the new surface is close to $S$ in the $\mathbf{F}$-topology (instead of just the Hausdorff topology).
By (\ref{noot}), it is then not hard to modify the family $\{\Gamma_t\}$ on both ends ($t$ close to $0$ and $t$ close to $1$) and obtain a smooth sweepout $\{\hat{\Sigma}_t\}$ as in item (2) of the lemma.


\end{proof}

\begin{proof}[Proof of Theorem \ref{positivegenus}]
Let $H$ be a strongly irreducible Heegaard splitting of $M$, which is necessarily irreducible by Haken's Lemma. First suppose that the metric is bumpy. 
Assume that $H$ is neither isotopic to a stable minimal surface, nor isotopic to
the boundary of a tubular neighborhood of a non-orientable minimal surface $S$ with a vertical handle attached, where $S$ has stable oriented double cover.

Under this assumption on $H$, by Corollary \ref{existence core} applied to $N=M$, there is an $H$-core $\mathfrak{C}$. 
By definition of $H$-cores, there is an $H$-compatible splitting $(W_0',W_1',H')$ such that $W_0'\cup W_1' =\mathfrak{C}$ (see Subsections \ref{subsection:existence} and \ref{subsection:hcore} for the definitions of $H$-compatible and $H$-core).
This means that there is at least one smooth sweepout of $\mathfrak{C}$, as defined in Definition \ref{defsmoothsweep}. 
Let $\Pi$ be the nonempty saturated set generated by all the possible smooth sweepouts $\{\Sigma_t\}_{t\in [0,1]}$ of $\mathfrak{C}$  
such that for any $t\in (0,1)$, there
is an $H$-compatible splitting $(W_0',W_1',\Sigma_t)$ of $\mathfrak{C}$, satisfying in particular $W_0'\cup W_1' =\mathfrak{C}$ (see Subsections \ref{mmdef} and \ref{subsection:existence} for definitions).

We apply the local min-max theorem, Theorem \ref{smoothminmax}, to the $H$-core $\mathfrak{C}$ and to $\Pi$. 
Strictly speaking, note that in the statement Theorem \ref{smoothminmax}, in order to define the saturated set $\Pi$, we fixed a partition $\Gamma_0 \cup \Gamma_1$ of the boundary of the compact 3-manifold. Here we do not fix a partition of the boundary $\partial \mathfrak{C}$ when defining $\Pi$. But there are only finitely many ways to partition the boundary 
\begin{equation}\label{partition}
\partial \mathfrak{C} = \Gamma_0 \cup \Gamma_1\end{equation}
into two distinct surfaces 
$\Gamma_0$ and $\Gamma_1$, and so Theorem \ref{smoothminmax} applies to each such partition.
In any case, by Theorem \ref{smoothminmax}, we get an embedded minimal surface 
$\Sigma \subset \mathfrak{C}$ with total area $W(\mathfrak{C},\Pi)$ after taking into account multiplicities, and which is not entirely contained in the boundary $\partial \mathfrak{C}$. We will divide our arguments into two cases: the non-orientable and orientable cases.

\textbf{Case where $\Sigma$ is non-orientable}: 
Suppose that  $\Sigma$ is non-orientable, then it has to arise via surgeries (Theorem 1.2 in \cite{Ketgenusbound}) and occur with multiplicity at least $2$ (Theorem \ref{simonsmith}).
Proposition \ref{strongirreduce} implies that there is a unique non-sphere component $\Sigma'$ of $\Sigma$ in the interior $\interior(\mathfrak{C})$. It is non-orientable. 
By Proposition \ref{strongirreduce}, $H$ is isotopic to $\partial N_\epsilon (\Sigma')$ with a vertical handle attached and the complement of a small tubular neighborhood $N_\mu(\Sigma')$ of $\Sigma'$ is a compression body by item (1) in Proposition \ref{strongirreduce}.
The oriented double cover $\tilde{\Sigma}'$ of $\Si'$ is necessarily unstable by our assumption on $H$ at the beginning of the proof. 
Since $\tilde{\Sigma}'$ is unstable, it follows that a small tubular neighborhood $N_\mu(\Sigma')$ has strictly mean concave boundary (with respect to the outward pointing normal unit vector) by bumpiness of the metric. We now apply Lemma \ref{constructsweep} to the complement of $N_\mu(\Sigma')$ inside $\mathfrak{C}$.
By definition of $H$-cores, only the second item of Lemma  \ref{constructsweep} is possible. 

Using \cite[Theorem 3.3]{KeMaNe}\footnote{Note that while an assumption in \cite[Theorem 3.3]{KeMaNe} is the positivity of the ambient Ricci curvature, the desired foliation of the tubular neighborhood about the non-orientable surface only uses the instability of the double cover of the non-orientable surface.}, we obtain a sweepout $\{\tilde{\Gamma}_t\}_{t\in [0,\epsilon]}$ of $N_\eta(\Sigma')$ that begins at the set $\tilde{\Gamma}_0$ which is a union of closed curves supported on $\Sigma'$ and ends at $\tilde{\Gamma}_\epsilon$ which consists of $\partial N_\eta(\Sigma')$ with an (arbitrarily) thin vertical handle attached.  Moreover, all of the areas in the sweepout are strictly less than $2\mbox{Area}(\Sigma')$. Then Lemma \ref{constructsweep} allows us to extend the  sweepout $\{\tilde{\Gamma}_t\}_{t\in [0,\epsilon]}$ to $\{\tilde{\Gamma}_t\}_{t\in [0,1]}$ which sweeps out the rest of $\mathfrak{C}\setminus N_\eta(\Sigma')$.  This gives rise to a smooth sweepout \footnote{In this sweepout the 2-sphere boundary components we obtain are all on one side in the partition \eqref{partition}, which our notion of saturated set allows.} $\{\tilde{\Gamma}_t\} \in \Pi$, such that 
$$W(\mathfrak{C}, \Pi) \leq \max_{t\in[0,1]} \mbox{Area}(\tilde{\Gamma}_t) < 2\mbox{Area}(\Sigma').$$
By Remark \ref{remark 6.3} on the multiplicity of nonorientable components, $2\mbox{Area}(\Sigma')\leq W(\mathfrak{C}, \Pi) $.  This is a contradiction with the inequality above. Thus, in fact $\Sigma$ cannot be non-orientable under our assumption on $H$.


\textbf{Case where $\Sigma$ is orientable}:
Suppose $\Sigma$ is orientable.
Note that we can suppose that there is no minimal 2-sphere in the interior of the $H$-core $\mathfrak{C}$. Indeed, suppose that there is a minimal 2-sphere $S''\subset \interior\mathfrak{C}$ which is stable.
Then since $M$ is irreducible, $S''$ bounds a $3$-ball on exactly one side (recall that $M$ is assumed to not be diffeomorphic to a $3$-sphere), so we can make sure that $H$ does not intersect $S''$, and then adding $S''$ to $\Gamma_i$ ($i=0$ or $1$) and maybe removing some sphere components of $\Gamma_0\cup \Gamma_1$, we clearly obtain a smaller $H$-compatible splitting.
This contradicts the definition of the $H$-core $\mathfrak{C}$.
Suppose, on the other hand, that there is a minimal 2-sphere $S'' \subset \interior(\mathfrak{C})$ which is unstable. We could 
{apply the result of Meeks--Simon--Yau \cite{MSY}} to minimize the area  of $S''$ in its isotopy class on the side not diffeomorphic to  a 3-sphere,
and obtain a limit inside $\interior(\mathfrak{C})$ {(note again that there is a non-trivial side because $M$ is not diffeomorphic to a 3-sphere by assumption). From this minimization procedure, we either get an $\mathbb{RP}^2$ component with \emph{stable} universal cover, in which case since the limit arose via surgeries (by Proposition \ref{strongirreduce}(1)) we obtain that $M$ is diffeomorphic to $\mathbb{RP}^3$. Moreover, $H$ is isotopic to a tubular neighborhood of this minimal $\mathbb{RP}^2$ with a vertical handle attached, which we have assumed is not the case. Or else we get a set of stable 2-spheres, which again contradicts the definition of the $H$-core $\mathfrak{C}$.}


The components of $\Sigma$ bound compression bodies or are included in the boundary of $\mathfrak{C}$, and none of the latter components which is not a 2-sphere is included in one of the compression bodies just mentioned.  Since $\Sigma$ is oriented it is obtained from surgeries by (Theorem 1.2 in  \cite{Ketgenusbound}), and it follows from Lemma \ref{sameside}, Proposition \ref{strongirreduce}, that all essential surgeries are on the same side and each component of $\Sigma$ with positive genus has multiplicity $1$. The union of these surfaces bound a union of disjoint compression bodies called $\mathcal{W}$.
Thus, by the conclusion of the local min-max theorem,  Theorem \ref{smoothminmax}, at least one component of $\Sigma$, called $\Sigma'$, is contained in the \emph{interior} $\interior(\mathfrak{C})$ and it is necessarily \emph{unstable}. To explain why $\Sigma'$ is unstable, note that if all components of $\Sigma$ are stable then we could find an $H$-core strictly contained in $\mathfrak{C}$ by applying Corollary \ref{existence core} to $N = M\setminus \mathcal{W}$, and this would contradict the definition of the $H$-core $\mathfrak{C}$. Next, this component $\Sigma'$ has positive genus (we are supposing that no minimal spheres are inside the $H$-core), occurs with multiplicity one.  
In fact, $\Sigma'$ has index  exactly 1, because the index bound of Marques--Neves says that the index of $\Sigma$ is at most 1 (see Theorem 1.2 and the discussion in Section 1.3 in \cite{MaNeindexbound}). 

We now claim that the component $\Sigma'$ is \emph{isotopic} to $H$. Suppose toward a contradiction that it is not.  Using again Meeks--Simon--Yau \cite{MSY}, we could then minimize area for surface isotopic to $\Sigma'$ on the side $M'\subset \interior(\mathfrak{C}) $ which is not a compression body, say the side of $\Gamma_1$ (this side exists because $M$ is not a $3$-sphere and $\Sigma'$ is not isotopic to $H$). 
Observe that $\Sigma'$ is incompressible inside $M'$,  as $\Sigma'$ is not isotopic to $H$ and comes from doing at least one essential neck-pinch surgery on $H$ on the side different from $M'$. 
By Meeks--Simon--Yau \cite{MSY} we obtain that a minimizing sequence for area in $M'$ converges to a (maybe disconnected) stable minimal surface $\Sigma''$ with integer multiplicities.  
Note that $\Sigma''$ is achieved after surgeries on $\Sigma'$ ($\gamma$-reductions in the terminology of \cite{MSY}), which itself was obtained by surgeries on $H$. 

Suppose first that a component of $\Sigma''$ is non-orientable. Then by Proposition \ref{strongirreduce}, its oriented double cover is stable and by adding a vertical handle we get an embedded surface isotopic to $H$. This case cannot occur by our assumption on $H$ at the beginning of the proof. 

Now suppose $\Sigma''$ is oriented: by Lemma \ref{sameside}, all essential surgeries must be on the side different from $M'$, and so there is a component $\Sigma'''$ of $\Sigma''$ that is isotopic to $\Sigma'$ (the other components being just 2-spheres) and is in the interior of $\mathfrak{C}$. But this would give an $H$-compatible splitting (see Subsection \ref{subsection:hcore}) strictly smaller than $\mathfrak{C}$, after replacing $\Gamma_0$ with $\Gamma_2:=(\Gamma_0 \cup\Sigma''')\backslash \Sigma'$ and considering the region bounded by $\Gamma_0 \cup\Gamma_2$. This contradicts the definition of the $H$-core $\mathfrak{C}$.
\vspace{1em}

To summarize what we have obtained so far after studying the two cases: when the metric is bumpy and when $H$ is neither isotopic to  a stable minimal surface, nor to the boundary of a tubular neighborhood of a non-orientable minimal surface $S$ with a vertical handle attached, where $S$ has stable oriented double cover, we have proved that $H$ is isotopic to an index $1$ minimal surface. 

\vspace{1em}

To finish the bumpy metric case, it remains to treat the cases where:
\begin{enumerate}[label=(\Alph*)]
\item  $H$ is isotopic to a stable minimal surface $T$,
\item $H$ is the boundary of a tubular neighborhood of a non-orientable minimal surface $S$ with a vertical handle attached, where $S$ has stable oriented double cover. 
\end{enumerate}
In case (A), by applying the local min-max theorem on one side of $T$,  Theorem \ref{smoothminmax}, and by the arguments of the first part of this proof, we obtain an oriented index 1 minimal surface in this handlebody isotopic to $H$.
In case (B), we verify that there is an index 1 minimal surface in $M$ of genus $genus(H)-1$ which is isotopic to the boundary of a tubular neighborhood of $S$. 
To see this, we pass to a double cover $\tilde{M}$ of $M$ so that $S$ lifts to an orientable stable minimal surface $\tilde{S}$ in $\tilde{M}$ bounding a handlebody whose interior projects diffeomorphically to $M\setminus S$.  By applying the local min-max theorem,  Theorem \ref{smoothminmax}, and by the arguments of the first part of this proof, we obtain an oriented index 1 minimal surface in this handlebody isotopic to $\tilde{S}$.  This surface descends to the desired embedded minimal surface in $M\backslash S$. This completes the proof of the theorem in the case where the metric is bumpy.

\vspace{1em}

Finally, if the metric is not bumpy, we use Lemma \ref{approximation} to approximate $g$ by bumpy metrics $g_m$. For $C>0$ large enough, for each $m$, there is a $C$-bounded $H$-core $\mathfrak{C}_m$ and $W(\mathfrak{C}_m, \Pi)$ converges to $W(\mathfrak{C}, \Pi)$ (see Definition \ref{coredef}). If $C$ is chosen bigger than $2W(\mathfrak{C}, \Pi)$, then we check without difficulty that the above arguments hold for a $C$-bounded $H$-core instead of an $H$-core for large $m$, since $\lim_m W(\mathfrak{C}_m,\Pi) = W(\mathfrak{C},\Pi)$. We get for each $m$ large enough 
\begin{itemize}
\item either a minimal surface area at most $C$, of index 0 or 1, and isotopic to $H$,
\item or a non-orientable minimal surface with area at most $C$, with stable oriented double cover, such that when we attach a vertical handle to the boundary of a tubular neighborhood, we get an embedded surface isotopic to $H$.
\end{itemize}
These minimal surfaces having area uniformly bounded by $C$, they subsequentially converges by Sharp (Theorem 2.3,  \cite{Sharp}) to a minimal surface $\Sigma^*$, smoothly except around at most one point. By strong irreducibility of the Heegaard splitting $H$ and Proposition \ref{strongirreduce}, either the limit is two-sided and the convergence is smooth and of multiplicity one, or the limit is one-sided, the oriented double cover of $\Sigma^*$ is stable (by a standard positive Jacobi field argument, see \cite[Proof of Theorem 2.3]{Sharp} for instance) and $H$ is isotopic to the boundary of a tubular neighborhood of $\Si^*$ with a vertical handle attached. So $\Sigma^*$ is as in the first part of the statement of Theorem \ref{positivegenus} and the theorem is proved.
\end{proof}

It remains to treat the case where $M$ is diffeomorphic to the 3-sphere, which is the easier case.
\begin{thm} \label{sphere case}
Let $(M,g)$ be a Riemannian $3$-sphere.  Then $M$ admits a minimal 2-sphere of index at most $1$. If the metric is assumed bumpy, then it admits an index $1$ minimal 2-sphere.
\end{thm}

\begin{proof}

For general non-bumpy metrics, this theorem was already known by the combination of Simon--Smith's theorem (Theorem \ref{simonsmith}) which produces some minimal 2-sphere, and the index bounds of Marques--Neves (Theorem 1.2 in \cite{MaNeindexbound}) which ensures that its index is at most $1$.

Let us assume that the metric is bumpy. Let $H$ be a genus $0$ Heegaard surface and consider any $H$-core guaranteed by Corollary \ref{existence core} (applied to $N=M$). Applying the local min-max theorem, Theorem \ref{smoothminmax}, we obtain a collection of pairwise disjoint minimal two-spheres in the interior of the $H$-core (some components may occur mulltiplicities larger than $1$).  By the index bounds of Marques--Neves (Theorem 1.2 in \cite{MaNeindexbound}), at most one component has index $1$ and all others are stable.  If they are all stable, this violates the minimality of the $H$-core.  Thus, there exists an unstable component of index $1$ contained in the interior of the $H$-core.   

\end{proof}

\vspace{2em}

\section{Applications} \label{sec8}

\subsection{Minimal Heegaard splittings with a priori bounds}
We use the term ``minimal Heegaard splitting" to denote a closed connected embedded minimal surface which is a Heegaard splitting.

In this subsection, by combining the main Theorem \ref{introthm} with elementary topological arguments, we obtain the existence of minimal Heegaard splittings with explicit information on their genus and index.

First, as an immediate corollary of Theorem 
\ref{positivegenus} and Theorem \ref{sphere case}
 we have: 

\begin{coro}
Conjecture \ref{rubinstein}, namely Pitts--Rubinstein's conjecture, is true.
\end{coro}


The next applications focus on spherical space forms, namely 3-manifolds homeomorphic to quotients of a 3-sphere $\mathbb{S}^3/\Gamma$. 
We note that, since a spherical space form admits a positive curvature metric, it does not contain incompressible surfaces so any Heegaard splitting of minimal genus is strongly irreducible by \cite{CassonGordon}.

Since the genus $1$ splitting of $\mathbb{RP}^3$ is strongly irreducible, the following is immediate from Theorem \ref{positivegenus}.
\begin{coro} \label{plot}
Any $(\mathbb{RP}^3,g)$ contains
\begin{itemize}
\item either a minimal Heegaard splitting of genus one and index at most one, 
\item or a minimal $\mathbb{RP}^2$ with stable oriented double cover.
\end{itemize}
If the metric $g$ is bumpy, then 
\begin{itemize}
\item either there is an index 1 Heegaard splitting of genus one,
\item or there is an index 1 minimal sphere.
\end{itemize}
\end{coro}

\begin{coro}\label{rp3corollary}
Any lens space not diffeomorphic to $\mathbb{S}^3$ or $\mathbb{RP}^3$ contains a minimal Heegaard splitting of genus one with index at most one. 
\end{coro}
\begin{proof}
The genus $1$ splitting of lens spaces $L(p,q)$ (with $p\geq 2$) is irreducible, and thus strongly irreducible by Casson-Gordon (Theorem 3.1 in  \cite{CassonGordon}) since such manifolds are non-Haken.  Thus by Theorem \ref{positivegenus}, if item i) occurs, the conclusion follows.  If item ii) holds, then by the genus bounds, the non-orientable surface is either an embedded projective plane or a Klein bottle.  But lens spaces $L(p,q)$ with $p\geq 3$ contain no embedded projective planes. Moreover, the boundary of a Klein bottle with a vertical handle added gives a surface of genus $2$, which is impossible as we began with a genus $1$ splitting. Thus, only item i) in Theorem \ref{positivegenus} occurs, as desired.
\end{proof}


Our next application concerns the positive scalar curvature case. There, we can often rule out the degenerate case, item (ii), of Theorem \ref{positivegenus} from occurring.

\begin{coro}\label{cor:topo bd}
Let $(M,g)$ be a spherical space form not diffeomorphic to $\mathbb{S}^3$ or $\mathbb{RP}^3$, and with positive scalar curvature. Let $ H$ be a strongly irreducible Heegaard splitting of $M$. Then $M$ contains an index 1 minimal surface isotopic to $H$.
\end{coro}

\begin{proof}
A classical property of positive scalar curvature metric is that any embedded orientable stable surface is a 2-sphere and any non-orientable minimal surface with stable oriented double cover is an $\mathbb{RP}^2$ (see \cite{SchoenYau}).
Since the spherical space form  $M$ is  not diffeomorphic to $\mathbb{S}^3$ or $\mathbb{RP}^3$, $M$ contains no embedded projective planes. 
The corollary is then a direct consequence of Theorem \ref{positivegenus}. 
\end{proof}

In \cite[Theorem 23]{Antoine}, the third-named author showed using Almgren--Pitts min-max theory that closed 3-manifolds with scalar curvature $R_g\geq 6$ contain a closed embedded minimal surface of area at most $4\pi$, with equality if and only if the 3-manifold is a round 3-sphere. This area bound, and the topological bound of Corollary \ref{cor:topo bd}, raise the following question:
\begin{question} \label{question}
If $(M,g)$ is a spherical space form not diffeomorphic to $\mathbb{S}^3$ or $\mathbb{RP}^3$, with $R_g\geq 6$, is there a strongly irreducible minimal Heegaard splitting with area less than $4\pi$?
\end{question}
In the special case where $\Ric_g>0$, this is true by combining \cite{MaNe} and \cite{KeMaNe} but it seems arduous to extend the method in \cite{MaNe} based on the Ricci flow to the case  where $R_g>0$ (see \cite{Antoinespheres}).

\subsection{How optimal is the main theorem?}

We now show that in general Theorem \ref{positivegenus} is optimal in the sense that both items can occur, by giving an example where the second case occurs in Corollary \ref{plot}.
Roughly speaking, we take the quotient of a capped off long cylinder.
\begin{prop}\label{example}
There exist a metric $g$ on $M\cong \mathbb{RP}^3$ with positive scalar curvature so that $M$ admits \emph{no}  minimal Heegaard torus of index at most $1$.
\end{prop}
\begin{proof}

Let $f: [-N,N] \rightarrow [0,1]$ satisfy:
\begin{enumerate}
    \item $f(r) = 1$ for $r\in [-(N-1),(N-1)]$;
    \item $f$ is smooth on $(-N,N)$ and continuous on $[-N,N]$;
    \item $f(-r)=f(r)$ for all $r\in [-N,N]$;
    \item $f(N)=f(-N)=0$;
    \item $f'(r)\leq 0$, $f''(r)\leq 0$ for all $r\in [0,N]$ and $f^{(k)}(N) = -\infty $ for $k \geq 1$.
\end{enumerate}

For $x_1 \in [-N,N]$ define $\tilde{M}_N \subset \mathbb{R}^4$ by 
$$\tilde{M}_N = \{ (x_1, x_2, x_3, x_4): x_2^2 +x_3^2 + x_4^2 = f(x_1)^2 \} $$

Let $\tau:\tilde{M}_N\to\tilde{M}_N$ be the isometry acting freely
\begin{equation}
\tau(x_1, x_2, x_3, x_4)=-(x_1, x_2, x_3, x_4), 
\end{equation}
and let $M_N:=\tilde{M}_N/\mathbb{Z}_2$ (as Riemannian manifolds) where $\mathbb{Z}_2$ is the group of isometries generated by $\tau$. 

Let $\pi: \tilde{M}_N\to M_N$ denote the projection map.  
The quotient $M_N$ is diffemorphic to $\mathbb{RP}^3$ and, choosing $f$ appropriately, we can arrange that the induced Riemannian metric on $M_N$ is smooth and satisfies
\begin{equation}
\mbox{Scal}_{M_N}\geq 1.
\end{equation}
For any $0\leq t\leq N-1$, define the surface:
\begin{equation}
P_t=\pi(\{t\}\times\mathbb{S}^2)\subset M_N.
\end{equation}
Note that $P_0$ is a minimal projective plane and $P_t$ is a minimal two-sphere for $t \in (0, N-1]$.

Suppose $M_N$ contains an embedded minimal Heegaard torus $\Sigma_N$ with index at most $1$.   By a result of Schoen-Yau (Theorem 5.1 in \cite{SY}) the only stable orientable minimal surface in three-manifolds with positive scalar curvature are two-spheres.  Thus we can assume that the index of $\Sigma_N$ is exactly $1$.  Observe that $\Sigma_N\cap P_0\neq \emptyset$ for each $N$. Otherwise, one of the two handlebodies determined by the Heegaard surface $\Sigma_N$ admits an embedded projective plane $P_0$, which is impossible.

Note that $\Sigma_N$ must also intersect each two-sphere $P_t$ for $0\leq t \leq N-1$.  Otherwise, letting $L>0$ be the infimal number such that $P_t\cap \Sigma_N$ is empty, we would get that $\Sigma_N$ is tangent to the minimal sphere $P_L$ and lies on one side of it.  This violates the maximum principle.  Thus, we can choose a point 
$$p_N^k \in P_k \cap\Sigma_n$$ for each $k\in \{1,...,N-2\}$.

Let $k\in \{1,...,N-2\}$.
The monotonicity formula (Proposition 1.16 and equation 7.5 in \cite{CM}) gives some $0<r_0<1/2$ and $C_0>0$ independent of $N$ and $k\in \{1,...,N-2\}$ (since the $[0,1]\times\mathbb{S}^2$ regions between $P_k$ and $P_{k+1}$ for $k \in \{1,...,N-2\}$ are pairwise isometric and the computations leading to the monotonicity formula are local) such that
\begin{equation}\label{eachone}
\mbox{Area}(\Sigma_N\cap B_{p_N^k}(r_0))\geq C_0r_0^2\quad \mbox{for } k=1,...,N-2.
\end{equation}

Since $r_0<1$ the balls $\{B_{p_N^1}(r_i), B_{p_N^2}(r_i),...B_{p_N^{N-2}}(r_i)\}$ are pairwise disjoint and thus we obtain from \eqref{eachone}
\begin{equation}\label{arealower}
\mbox{Area}(\Sigma_N)\geq \sum_{k=1}^{N-2}\mbox{Area}(\Sigma_N\cap B_{p_N^k}(r_0))\geq (N-2) C_0r_0^2.
\end{equation}
On the other hand, according to \cite[Proposition A.1 (iii)]{MaNe}, if a three-manifold has $\mbox{Scal}\geq k_0$ then any immersed orientable index $1$ genus $g$ minimal surface $\Sigma$ satisfies:
\begin{equation}
k_0\mbox{Area}(\Sigma)\leq 24\pi + 16\pi(g/2-\lfloor g/2\rfloor).
\end{equation}
Specializing to our case where $k_0=1$ and $\Sigma_N \subset M_N$, we get
\begin{equation}\label{areascalar}
\mbox{Area}(\Sigma_N)\leq 32\pi.
\end{equation}
Taking $N$ large enough, \eqref{areascalar} contradicts \eqref{arealower}.
\end{proof}

We also give an example of a three-manifold admitting a genus $1$ Heegaard surface that is not strongly irreducible and in fact contains no minimal Heegaard torus of index at most one.  In the following, let $\mathbb{S}^n(a)$ denote the round $n$-sphere of radius $a$. 

\begin{prop}\label{counterexample}
    Let $M_a=\mathbb{S}^2(1)\times\mathbb{S}^1(a)$ be endowed with the product metric.  For $a$ large enough, $M_a$ admits no minimal Heegaard torus of index at most $1$.
\end{prop}

\begin{proof}
The scalar curvature of $M_a$ is identically equal to $2$.  Thus again by \cite[Proposition A.1 (iii)]{MaNe} the area of a purported index $1$ minimal torus $T_a\subset M_a$ satisfies
\begin{equation}\label{uppertori}
\mbox{Area}(T_a)\leq 16\pi.
\end{equation}
On the other hand, $T_a$ must intersect each $\mathbb{S}^2(1)\times \{t\}$ for all $t\in\mathbb{S}^1(a)$.  Otherwise since $\mathbb{S}^2(1)\times\{t\}$ are minimal surfaces we obtain a contradiction to the maximum principle. 

Fix $(x,t)\in \mathbb{S}^2(1)\times\mathbb{S}^1(a)$.  Note that the metrics $g_a$ based at $(x,t)$ converge to $\mathbb{S}^2(1)\times\mathbb{R}$ and recall that the computations leading to the monotonicity formula are local.  Thus for $a$ large enough, there are $r_0>0$ and $C_0>0$ so that any minimal surface $\Sigma$ passing through $(x,t)$ satisfies
\begin{equation} \label{mon form 2}
\mbox{Area}(T_a\cap B_{(x,t)}(r_0))\geq C_0 r^2.
\end{equation}
for all $0<r\leq r_0$.
Divide $\mathbb{S}^1(a)$ into $\lfloor a/r_0\rfloor$ sub-intervals $I_1,...,I_{\lfloor a/r_0\rfloor}$ of length at least $r_0$.  It follows from (\ref{mon form 2}) that 
\begin{equation}
\mbox{Area}(T_a\cap (\mathbb{S}^2(1)\times I_i))\geq \frac{C_0}{4}r_0^2.
\end{equation}
Since $\{\mathbb{S}^2(1)\times I_i\}_{i=1,...,\lfloor a/r_0\rfloor}$ are pairwise disjoint, 
\begin{equation}\label{areabound}
\mbox{Area}(T_a)\geq \frac{C_0}{4}r_0^2 \lfloor a/r_0\rfloor. 
\end{equation}
For $a$ large enough, \eqref{areabound} contradicts \eqref{uppertori}.
\end{proof}

\vspace{2em}

\section*{Appendix: Proof of (\ref{boundtrue}) in the proof of Theorem \ref{smoothminmax}}

The following lemma was essentially proved in \cite{MorganRos} but we give another proof for the sake of completeness. The definition of ``saturated set'' is recalled in Subsection \ref{mmdef}.

\begin{lemma} \label{canapply1}
Assume that ${N}\subset (M^{n+1},g)$, that the metric $g$ is bumpy and that we are given  a partition $\partial N = \Gamma_0 \cup \Gamma_1$, where $\Gamma_0$, $\Gamma_1$ are unions of connected components. Let ${\Pi} $ be a saturated set generated by smooth sweepouts $\{\Sigma_t\}$ such that
\begin{itemize}
\item $\mathcal{A}(\{\Sigma_t\}) =[|N|]$,
\item $\Sigma_0=\Gamma_0$ and $\Sigma_1=\Gamma_1$.
\end{itemize}
Suppose that $\Gamma_0$ is a stable minimal hypersurface. Then 
$$ W(N,\Pi) > \mbox{Area}(\Gamma_0).$$
Consequently, if $\Gamma_0$ and $\Gamma_1$ are both stable minimal hypersurfaces, then (\ref{boundtrue}) holds.
\end{lemma}

\begin{proof}

Given an embedded surface $\Gamma'$, denote by $|\Gamma'|$ the varifold induced by $\Gamma'$ with multiplicity one. 
In this proof, we will use the notion of integral cycles (namely integral currents without boundary): by abuse of notations, an oriented closed surface $T$ will be identified with the integral current it induces. Let $\mathbf{M}$ denote the mass defined on the space of currents, and let $||T||(U)$ be the mass of the restriction of $T$ to a subset $U$. Let $\mathbf{F}$ denote the varifold metric
as defined in \cite[page 66]{P}. For more information on geometric measure theory, see \cite{Simon84}.

Since any sweepout $\{\Sigma_t\}\in \Pi$ sweeps out $N$ non trivially, for all $\epsilon_0>0$ small enough there is an element $t\in[0,1]$ for which $\mathbf{F}(|\Sigma_t|, |\Gamma_0|) \in (\epsilon_0,2\epsilon_0)$ when $i$ is large enough. Let $\mathbf{B}^{\mathbf{F}}_{\epsilon}(|\Gamma_0|)$ be the $\mathbf{F}$-ball of radius $\epsilon$ centered at $|\Gamma_0|$ in $$\mathcal{V}_2(M):=\text{closure in the weak topology of the space of rectifiable $2$-varifolds}.$$ 

\textbf{Claim:} For $\epsilon_0$ small enough, there is a $\delta>0$ such that for any integral cycle $T\in\mathcal{Z}_2(M)$ satisfying 
$$|T|\in A:= \mathbf{B}^{\mathbf{F}}_{2\epsilon_0} (|\Gamma_0|)\backslash \mathbf{B}^{\mathbf{F}}_{\epsilon_0}(|\Gamma_0|),$$
we have $\mathbf{M}(T) \geq \mathbf{M}(\Gamma_0) +\delta$.

The lemma readily follows from this claim. To prove the latter, we argue by contradiction and consider a sequence of cycles $T_i\in\mathcal{Z}_2(M)$ with $|T_i|\in A$, and a sequence of positive numbers $\delta_i$ going to zero such that 
$$ \mathbf{M}(T_i) \leq \mathbf{M}(\Gamma_0) +\delta_i.$$ Let $\Omega_r$ be an $r$-neighborhood of $\Gamma_0$ so that there is a family of area-decreasing maps $\{P_t\}_{t\in[0,1]}$ as in \cite[Proposition 5.7]{MaNeindexbound}. Note that $||T_i||(M\backslash \Omega_{r/2})$ is smaller than $\kappa.\epsilon_0$ where $\kappa=\kappa(\Gamma_0,r)\geq 1$ is a constant. By the properties of $\Omega_r$, if we fix $\epsilon_0$ small enough then for any integral cycle $\hat{T}$ with $|\hat{T}| \in \mathbf{B}^{\mathbf{F}}_{(1+2\kappa)\epsilon_0}(|\Gamma_0|)$ and support in $\bar{\Omega}_r$ we have by the constancy theorem \cite[Theorem 26.27]{Simon84}:
\begin{equation} \label{constancy}
\quad (P_1)_\sharp \hat{T} = \pm \Gamma_0 \text{ and }\mathbf{M}(\Gamma_0)\leq \mathbf{M}(\hat{T}).
\end{equation}
Here $(P_1)_\sharp \hat{T}$ denotes the push-forward of the current $\hat{T}$ by the Lipschitz map $P_1$.
For almost all $r'\in(r/2,r)$, we can minimize the part of $T_i$ outside $\Omega_{r'}$, by the monotonicity formula (fix $\epsilon_0$ small) we get an integral cycle $T_i'$ coinciding with $T_i$ inside $\Omega_{r'}$ but area-minimizing outside $\bar{\Omega}_{r'}$, and satisfying 
$$\spt(T_i') \subset \Omega_r \text{ and } \mathbf{M}(T_i')\leq \mathbf{M}(T_i).$$ 
By the choice of the sequence $T_i$, 
\begin{equation} \label{intuitif}
\mathbf{M}(\Gamma_0)\leq \mathbf{M}(T_i')\leq \mathbf{M}(T_i) \leq \mathbf{M}(\Gamma_0) +\delta_i.
\end{equation}
Note that by construction $\mathbf{F}(|T_i|,|T_i'|) \leq 2 ||T_i||(M\backslash \Omega_{r'})$. Since $\delta_i$ goes to zero, $r'=r'(i)\in(r/2,r)$ can be chosen so that the mass $ ||T_i||(M\backslash \Omega_{r'}) $ also converges to zero. Indeed, either $||T_i||(\Omega_{r}\backslash\Omega_{r/2})$ goes to zero or not. In the first case, let $f$ be the function defined before Proposition 5.7 in \cite{MaNeindexbound}. By the coarea formula, we find $r'\in(r/2,r)$ such that $\langle T_i, f, r'\rangle $ is an integral current with arbitrarily small mass (here, $\langle T_i, f, r'\rangle $ denotes the slice current as defined in \cite[Definition 28.4]{Simon}). Consequently $||T_i'||(M) -||T_i||(\Omega_{r'})$ is arbitrarily small. Then $||T_i||(\Omega_{r'})$ is arbitrarily close to $\mathbf{M}(\Gamma_0)$ because the integral cycle $T_i'$ satisfies (\ref{intuitif}). Since $\mathbf{M}(T_i) = ||T_i||(\Omega_{r'}) +||T_i||(M\backslash \Omega_{r'})$, it forces $||T_i||(M\backslash\Omega_{r'})$ to go to zero too. In the second case, namely if $||T_i||(\Omega_{r}\backslash\Omega_{r/2})$ does not go to zero, we choose $r'$ tending to $r$ as $i\in \infty$, such that $||T'_i||(\Omega_{r}\backslash\Omega_{r/2})$ is also bounded away from zero (for a subsequence in $i$). Then by the computation in the proof of \cite[Proposition 5.7]{MaNeindexbound} and supposing (\ref{intuitif}) true, the derivative of $(P_t)_\sharp|T_i'|$ is uniformly bounded above by a negative constant for $t\in[0,t_0]$ where $t_0 >0$ is independent of $i$. This contradicts the upper bound in (\ref{intuitif}). To summarize, we just showed that for a certain choice of $r'=r'(i)$, we have $\lim \mathbf{F}(|T_i|,|T_i'|)=0$.

Besides, by the compactness theorem for integral cycles in the flat topology \cite[Theorem 27.3]{Simon84} and by lower semicontinuity of the mass  \cite[26.13]{Simon84}, $T_i'$ converges subsequentially to a current $T'_\infty$ in the flat topology with mass at most $\mathbf{M}(\Gamma_0)$ and $(P_1)_\sharp T'_i$ is equal to $\pm \Gamma_0$ for $i$ large, by (\ref{constancy}). Hence, since $T'_\infty$ has support in $\Gamma_0$ by \cite[Proposition 5.7, (iv)]{MaNeindexbound}, $T'_\infty=\pm \Gamma_0$ and there is no mass cancellation (i.e. the mass of the limit is the limit of masses) so $|T'_i|$ converges to $|\Gamma_0|$ as varifolds. In conclusion, $\mathbf{F}(|T_i'|,|\Gamma_0|)$ goes to zero, and 
$$\mathbf{F}(|T_i|,|\Gamma_0|) \leq \mathbf{F}(|T_i|,|T_i'|)+ \mathbf{F}(|T_i'|,|\Gamma_0|)$$
also converges to zero, contradicting the fact that $T_i \in A$. This proves the Claim and thus the Lemma.

\end{proof}

\bibliographystyle{plain}
\bibliography{bib.bib}

\end{document}